\documentclass[12pt]{amsart}
\usepackage[a4paper,hmargin=2cm,vmargin=3cm]{geometry}
\usepackage[utf8]{inputenc}
\usepackage{color}
\usepackage{amsmath,bm,bbm}
\usepackage{amsthm}
\usepackage{amssymb}
\usepackage{graphicx}
\usepackage[style=alphabetic,maxnames=10,giveninits=true,backend=biber,doi=false,url=false]{biblatex}
\usepackage[shortlabels]{enumitem}
\usepackage{hyperref}
\usepackage{amsaddr}
\usepackage{appendix}
\usepackage[makeroom]{cancel}
\usepackage[normalem]{ulem}
\usepackage{url}
\usepackage{pdfcomment}

\usepackage{marginnote}

\newcommand{\DB}[1]{\marginnote{\color{blue}\setlength{\fboxrule}{1.5pt}%
    \tiny\sffamily\fbox{\parbox{1.3cm}{\raggedright{DB: #1}}}}}

\usepackage{rotating}
\usepackage{mathtools}
    
\hypersetup{unicode=true,
	colorlinks=true, citecolor=blue, linkcolor=blue,
}

\DeclareSourcemap{
	\maps[datatype=bibtex]{
		\map{
			\step[fieldset=issn, null]
		}
	}
}
\numberwithin{equation}{section}

\theoremstyle{plain}
\newtheorem{theorem}{Theorem}[section]
\newtheorem{lemma}[theorem]{Lemma}
\newtheorem{corollary}[theorem]{Corollary}
\newtheorem{proposition}[theorem]{Proposition}
\newtheorem{defn}[theorem]{Definition}

\theoremstyle{remark}
\newtheorem{remark}[theorem]{Remark}

\newcommand{\bs}{\bm{\sigma}}
\newcommand{\bJ}{\bm{J}}
\newcommand{\bmm}{\bm{m}}

\newcommand{\bh}{\bm{h}}

\newcommand{\E}{\mathbb{E}}	

\newcommand{\cM}{ {\rm{Plef}}_N(\bQ,\beta) }

\newcommand{\cF}{\mathcal{F}}

\renewcommand{\P}{\mathbb{P}}
\newcommand{\R}{\mathbb{R}}

\renewcommand{\epsilon}{\varepsilon}

\newcommand{\iid}{\text{i.i.d. }}

\DeclareMathOperator{\pP}{\mathbb{P}}

\DeclareMathOperator{\1}{\mathbbm{1}}

\newcommand{\bt}{\bm{\tau}}

\newcommand{\bb}{\bm{\beta}}
\newcommand{\bbm}{\bm{m}}
\newcommand{\bu}{\bm{u}}

\newcommand{\bQ}{\bm{Q}}
\newcommand{\bA}{\bm{A}}
\newcommand{\bO}{\bm{O}}
\newcommand{\bL}{\bm{\Lambda}}
\newcommand{\bD}{\bm{\Delta}}
\newcommand{\bd}{\bm{D}}

\newcommand{\bU}{\bm{U}}
\newcommand{\bC}{\bm{C}}
\newcommand{\bX}{\bm{X}}
\newcommand{\bY}{\bm{Y}}

\newcommand{\bB}{\bm{B}}

\newcommand{\bM}{\bmm \bmm^{\Trans}}

\newcommand{\bI}{\bm{I}}

\newcommand{\diag}{\mathrm{diag}}

\DeclareMathOperator{\supp}{\mathrm{supp}}
\newcommand{\Trans}{\mathsf{T}}
\newcommand{\Tr}{\mathrm{Tr}}

\newcommand\numberthis{\addtocounter{equation}{1}\tag{\theequation}}

\newcommand{\Var}{\operatorname{Var}}

\newcommand{\TAP}{F_{\mathrm{TAP}}}
\newcommand{\tTAP}{\tilde F_{\mathrm{TAP}}}

\begin{document}
	
	\title[
	TAP variational principle for the constrained multiple spherical SK model]
	{
	TAP variational principle for the constrained overlap multiple spherical Sherrington-Kirkpatrick model}
	
	\author{David Belius$^{*,**}$, Leon Fr\"ober$^{*}$, Justin Ko$^\dagger$}	
	
	\email{david.belius@cantab.net}
	\email{leon.froeber@unibas.ch}
	\email{justin.ko@ens-lyon.fr}
	\thanks{* Supported by SNSF grant 176918.}
        \thanks{** Supported by SNSF grant 206148.} 
        \thanks{$\dagger$ Supported by ERC Project LDRAM:  ERC-2019-ADG}

	\begin{abstract}
		Spin glass models involving multiple replicas with constrained overlaps have been studied in \cite{replicaonandoff,PTSPHERE,PVS}. For the spherical versions of these models
		\cite{kocs,kosphere} showed that the limiting free energy is given by a Parisi type minimization.
		In this work we show that for Sherrington-Kirkpatrick (i.e. $2$-spin) interactions, it can also be expressed in terms of a Thouless-Andersson-Palmer (TAP) variational principle. This is only the second spin glass model where a mathematically rigorous TAP computation of the free energy at all temperatures and external fields has been achieved. The variational formula we derive here also confirms that the model is replica symmetric, a fact which is natural but not obviously deducible from its Parisi formula.
	\end{abstract}
	
	\maketitle
	
	\section{Introduction}
	We study the free energy of the constrained multiple replica spin glass model of \cite{replicaonandoff,PTSPHERE,PVS}, also called the vector spin model. In physics this free energy is known as the Franz--Parisi potential \cite{replicaonandoff}. The model involves multiple replicas with constrained overlaps and was originally introduced to study metastable states of standard one replica spin glasses \cite{replicaonandoff}, and has since been used to study several other of properties of one replica models \cite{PanchenkoChaos,ChenChaos2,JagSpectralGap,JagSphericalSpectralGap,JagMetastable,FlorentAhmedFranzParisi,Franzlargedeviations,JagStatisticalThresholds}.
    
    We introduce a new approach to studying the model by adapting the Thouless-Andersson-Palmer (TAP) approach of \cite{belius-kistler} to the model's spherical Sherrington-Kirkpatrick (SK; i.e. $2$-spin) version. We prove a variational formula for its free energy in terms of a TAP free energy, and compute a formula for the maximal TAP free energy, thus yielding a concrete formula for the original free energy. After \cite{belius-kistler} this represents only the second setting where the free energy of a spin glass model has been computed at all temperatures and external fields using a mathematically rigorous TAP approach.
    
	
	We now formally introduce the model. The $2$-spin SK Hamiltonian is a Gaussian process of the form
	\begin{equation}\label{eq: hamiltonian def}
		H_N(\sigma) = \sqrt{N} \sum_{i,j = 1}^N J_{ij} \sigma_i \sigma_j
	\end{equation}
            indexed by $\sigma \in \R^N$, where $J_{ij}$ are \iid standard Gaussian random variables. For $n \geq 1$, we consider the multiple spin configuration of $n$ replicas denoted by the matrix
	\[
	\bm{\sigma} = (\sigma^1, \dots, \sigma^n) \in \R^{n \times N},
	\]
	where each $\sigma^k \in \R^N$ denotes the $k$-th row of $\bs$ and $\sigma_{i}^k$ the entry in the $i$-th column and $k$-th row.
	Let
	$$
	\mathcal{S}_{N-1} = \{\sigma \in \R^N: \sigma_1^2+...+\sigma_N^2 =1\}
	$$
	denote the unit sphere in $\R^N$.
	Let $\bh = (h^1, \dots , h^n) \in \R^{n \times N}$ and $\beta = (\beta_1, \dots, \beta_n)$ denote the external fields and inverse temperatures of each replica. Furthermore, assume that $|h^k| = h_k \in \R$ and let $h = (h_1, \dots, h_n) \in \R^n$. Let $\bQ \in \mathbb{R}^{n\times n}$ be a positive semi-definite matrix with $1$'s along the diagonal giving a constraint on the overlaps of the replicas.
	For a matrix $\bA$ let $\|\bA\|_\infty$ denote the $\sup$-norm $\max_{k,l}|A_{k,l}|$, and for $\varepsilon>0$ let
	\begin{equation}\label{eq: eps nbhood}
		\bQ_\varepsilon = \{ \bs: \|\bs \bs^\Trans - \bQ \|_\infty \leq \varepsilon \},
	\end{equation}
	denote the set of replicas with overlaps close to $\bQ$.
	
	Our goal is to compute the limit of the replica constrained free energy
	\begin{equation}\label{def:FE}
		F^\varepsilon_N(\beta,\bh,\bQ) = \frac{1}{N} \log \int_{\bQ_\varepsilon} e^{\sum_{k = 1}^n \beta_k H_N(\sigma^k) + N h^k \cdot \sigma^k } d\bs
	\end{equation}
	for fixed model parameters $\bQ, \bh, \beta$,
	where $d\bs = (d \sigma)^{\otimes n}$ is the product of uniform measures $d\sigma$ on the sphere $\mathcal{S}_{N-1}$.
	Note that the integral can not be trivially reduced to a one replica integral using Fubini's theorem because the replica overlaps are constrained to the corresponding values of $\bQ$. Note further that each replica shares the same disorder $J_{ij}$ but can be subject to different inverse temperatures $\beta_k$ and external fields $h^k$.
	
	The TAP free energy we derive for this model is given by
	\begin{equation}\label{eq:TAP}
		\TAP(\bbm) =   \frac{N}{2} \log| \bQ - \bM| + \sum_{k = 1}^n \beta_k H_N ( m^k ) + N \sum_{k = 1}^n h^k \cdot m^k  + \frac{N}{2} \beta^\Trans (\bQ - \bM)^{\odot 2} \beta,
	\end{equation}
	where $\bbm = (m^1,\dots, m^n) \in \R^{n \times N}$ are magnetization vectors, $| \cdot |$ denotes the determinant, and $\bA^{\odot 2} = \bA \odot \bA = (A_{k,l}^2)_{k,l=1,\ldots,n}$ denotes the Hadamard square of the entries of $\bA$.
	We further introduce a Plefka condition \cite{thoulessSolutionSolvableModel1977,plefka1982convergence} for the vector spin model given by $\bmm \in {\rm{Plef}}_N(\bQ,\beta)$ for
	\begin{equation}\label{eq:Plefkadefn}
	\begin{array}{rcl}
	    {\rm{Plef}}_{n}(\bQ,\beta) &=& 
	    \left\{ \tilde{\bm{Q}} \in [-1,1]^{n \times n} :   \bm{0} \leq  \tilde{\bm{Q}} < \bQ ,\, \| \bb^{\frac{1}{2}} (\bQ - \tilde{\bm{Q}}) \bb^{\frac{1}{2}} \|_{2} \leq \frac{1}{\sqrt{2}} \right\},
	    \\
		{\rm{Plef}}_N(\bQ,\beta)&=& \left\{ \bbm \in \R^{n \times N} :  \bbm \bbm^{\Trans} \in {\rm{Plef}}_{n}(\bQ,\beta) \right\},
	\end{array}
	\end{equation}
	where $\bb = \diag(\beta) \in \R^{n \times n}$, $\| \cdot \|_2$ denotes the spectral norm (largest eigenvalue for symmetric positive semi-definite matrices) and $\leq$ is the Loewner partial order on matrices (so that $\bA \ge 0$ for $\bA \in \mathbb{R}^{n \times n}$ means that $\bA$ is positive semi-definite).
	
	Our main theorem is a TAP variational principle giving the limiting free energy of the model as a supremum over $\bbm \in {\rm{Plef}}_{n}(\bQ,\beta)$.
	\begin{theorem}[TAP Variational Principle]\label{thm:TAPformula}
		Let $n\ge1$ and $\bQ \in [-1,1]^{n\times n}$ be positive definite with $Q_{k,k}=1$ for $k=1,\ldots,n$. It holds that
		\begin{equation}\label{eq: main thm}
			\lim_{\varepsilon \to 0} \limsup_{N\to \infty} | F^\varepsilon_N(\beta,\bh,\bQ) - \sup_{\bbm \in \cM} \frac{1}{N} \TAP(\bbm) |=0,
		\end{equation}
		where the limits are in probability.
	\end{theorem}
	If $\bQ$ is not positive definite then $\lim_{\varepsilon \to 0} \lim_{N\to \infty} F^\varepsilon_N(\beta,\bh,\bQ) = -\infty$ (see \eqref{lem: small eval}).
	
	We also compute the supremum in \eqref{eq: main thm} when $h^1,\ldots,h^n$ are multiples of a single vector. To this end we let for any $\beta$, $h \in \mathbb{R}^n$ and positive definite $n\times n$ constraint matrix $\tilde \bQ$
	\begin{equation}\label{eq:Lfunctopt}
		{\rm{GSE}}(\beta,h,\tilde{\bm{Q}}) = \sqrt{2} \Tr \left( \sqrt{ \Big( \frac{1}{2}  h h^\Trans  + \bb \tilde{\bm{Q}} \bb\Big)^{\frac{1}{2}} \tilde \bQ \Big( \frac{1}{2}  h h^\Trans  + \bb \tilde{\bm{Q}} \bb\Big)^{\frac{1}{2}} }   \right).
	\end{equation}
    Note that the trace on the right-hand side is the sum of the singular values of $( \frac{1}{2}  h h^\Trans  + \bb \tilde{\bm{Q}})^{1/2} \tilde{\bm{Q}}^{1/2}$. The ground state of the energy over magnetizations $\bbm$ with constrained overlaps converges to this limit:
	\begin{theorem}[Ground state energy]\label{thm:groundstate} Assume that $h^i = h_i u$ for a sequence of unit vectors $u \in \mathbb{R}^N$ for $i=1,\ldots,n$. For all $\beta,h$ and positive definite $\tilde{\bQ}$
		\begin{equation}\label{eq: max TAP free energy}
		 \lim_{N \to \infty} \sup_{\bbm \bbm^\Trans = \tilde{\bm{Q}}} \bigg( \sum_{k = 1}^n \frac{1}{N} \beta_k H_N ( m^k) + \sum_{k = 1}^n m^k \cdot h^k  \bigg) = {\rm{GSE}}(\beta,h,\tilde{\bm{Q}}),
		\end{equation}
  where the limit is in probability.
	\end{theorem}
	
	The next theorem expresses the limiting maximum TAP free energy as a lower dimensional optimizion, namely as one of $n\times n$ (so bounded in $N$) rather than $n\times N$ dimensions.
	It follows immediately from \eqref{eq:TAP} and Theorem \ref{thm:groundstate}.
	
	\begin{corollary}[Low Dimensional Variational Principle]\label{cor:varprinc} Assume that $h^i = h_i u$ for a sequence of unit vectors $u \in \mathbb{R}^N$ for $i=1,\ldots,n$. For all $\beta,h$ and positive definite $\bQ$ it holds that
		\begin{equation}\label{eq:TAPFEnoH in terms of m}
			\lim_{N \to \infty} \sup_{\bbm \in \cM} \frac{1}{N} \TAP(\bbm) = \sup_{\tilde{\bQ} \in {\rm{Plef}}_n(\bQ,\beta)} \bigg( {\rm{GSE}}(\beta,h,\tilde{\bQ}) + \frac{1}{2}\log | \bQ - \tilde{\bQ}|  + \frac{1}{2}   \beta^\Trans (\bQ - \tilde{\bQ})^{\odot 2} \beta  \bigg) 
			,
		\end{equation}
		where the limit is in probability.
	\end{corollary}
	
	It follows immediately from Theorems \ref{thm:TAPformula} and 
 Corollary \ref{cor:varprinc} that also the limiting free energy is given by the same low dimensional optimization problem.
	\begin{corollary}  Assume that $h^i = h_i u$ for a sequence of unit vectors $u \in \mathbb{R}^N$ for $i=1,\ldots,n$. For all $\beta,h$ and positive definite $\bQ$ the limit of the free energy is
	    \begin{equation}\label{eq:TAPFEnoH}
			\lim_{\varepsilon \to 0} \lim_{N \to \infty} F^\varepsilon_N(\beta,\bh, \bQ) = \sup_{\tilde{\bm{Q}} \in {\rm{Plef}}_n(\bQ,\beta)} \bigg( {\rm{GSE}}(\beta,h,\tilde{\bm{Q}}) + \frac{1}{2}\log | \bQ - \tilde{\bm{Q}}| + \frac{1}{2}   \beta^\Trans (\bQ - \tilde{\bm{Q}})^{\odot 2} \beta  \bigg),
		\end{equation}
  where the limits are in probability.
	\end{corollary}

	\begin{remark}\label{ex:1}
    When $n=1$ we recover the results of \cite{belius-kistler}. Indeed the only valid constraint is $\bQ=1$, and with this constraint $F_{\rm{TAP}}$ coincides with $H_{\rm{TAP}}$ of  \cite{belius-kistler}, and Theorem \ref{thm:TAPformula} coincides with \cite[Theorem 1]{belius-kistler}.
    The functional \eqref{eq:Lfunctopt} is
		\begin{equation}\label{eq: case n=1}
		{\rm{GSE}}(\beta,h,\tilde{q})= \sqrt{2 \beta^2 \tilde{q}^2 + h^2 \tilde{q} },
		\end{equation}
		cf. \cite[(1.6) and Lemma 20]{belius-kistler}. Corollary~\ref{cor:varprinc} says
		\[
		\lim_{N \to \infty} F^\varepsilon_N(\beta, h, 1) = \sup_{m : \beta(1-\tilde{q}) \leq \frac{1}{\sqrt{2}}} \bigg( \sqrt{2 \beta^2 \tilde{q}^2 + h^2 \tilde{q} } + \frac{\beta^2}{2}  (1 - \tilde{q})^2 + \frac{1}{2} \log | 1 - \tilde{q}|  \bigg),
		\]
		for all $\varepsilon>0$, cf. \cite[Lemma~2]{belius-kistler}.
	\end{remark}

    \subsection{Discussion}
    \,
	The most important result about one replica ($n=1$) spin glass models \cite{sherrington1975solvable}\footnote{See \cite{kosterlitzSphericalModelSpinGlass1976,derridaRandomEnergyModelLimit1980,grossSimplestSpinGlass1984,crisantiSphericalpspinInteractionSpin1992,talagrandMultipleLevelsSymmetry2000,crisantiSphericalSpinglassModel2004,talagrandFreeEnergySpherical2006} for the various generalizations of the original Ising type $2$-spin SK model.} is the Parisi formula \cite{parisi1980sequence,parisiInfiniteNumberOrder1979,mezard1987spin} for the limiting free energy which has been proved rigorously using the methods of Guerra, Aizenman--Sims--Starr, Talagrand and Panchenko \cite{guerra,ASS,CASS,talagrand,talagrandFreeEnergySpherical2006,PUltra,panchenkoParisiFormulaMixed2014}. The TAP approach is an attractive proposal \cite{thoulessSolutionSolvableModel1977} of an alternative framework to compute the free energy which is under active investigation, with at least three projects underway to implement it mathematically rigorously (\cite{bolthausenIterativeConstructionSolutions2014,bolthausenMoritaTypeProof2018,brennecke2022replica}, \cite{subagGeometryGibbsMeasure2017,subagFreeEnergyLandscapes2020,chenGeneralizedTAPFree2022,subagFreeEnergySpherical2021}, \cite{belius-kistler, BeliusUpperBound}).
	
	Concerning constrained multiple spin glass models ($n\ge1$; \cite{replicaonandoff,PTSPHERE,PVS}) an upper bound for the free energy of spherical models was proved in \cite{PTSPHERE} using the Guerra interpolation scheme. The matching lower bound for this model was proved in \cite{kosphere,kocs} by adapting the synchronization property derived for constrained multiple spin models with respect to product measures by Panchenko in \cite{PVS,PPotts} and the Aizenman--Sims--Starr scheme.
	In this article, we investigate the $2$-spin constrained multiple spherical spin model using the TAP approach of \cite{belius-kistler} (see also \cite{BeliusUpperBound}) and derive the new variational expression \eqref{eq:TAPFEnoH} for the limiting free energy.
	The variational formula is expressed as the maximum of a functional defined on $n \times n$ matrices. It is much simpler than the $2$-spin version of the Parisi variational formula from \cite{kocs} defined in terms of matrix paths \cite[Theorem~1 and Theorem~3]{kocs}. After \cite{belius-kistler} our results represents only the second setting where the free energy of a spin glass model has been computed at all temperatures and external fields using a mathematically rigorous TAP approach (\cite{subagFreeEnergySpherical2021} uses a different version of the TAP approach to compute the free energy for pure $p$-spin spherical spin glasses without external field at all temperatures). We hope that in the future a further improvement of the present TAP approach can be extended to a wider class of spin glass models.
	
	A well-known property of the classical ($n=1$) spherical $2$-spin model is that it is replica symmetric at any inverse temperature and external field, as can be verified by studying the Parisi formula for the model \cite[Section~2]{talagrandFreeEnergySpherical2006}. For the constrained multiple spin model, \cite{tucaproperties} gives a zero temperature Parisi formula for the ground state and shows that the minimizer is replica symmetric in the case of $2$-spin interaction \cite[Proposition~7]{tucaproperties}. A similar computation at positive temperature seems infeasible, so presently one can not deduce that the free energy of the constrained multiple $2$-spin model is replica symmetric from its Parisi formula. Since we use the TAP approach we do not directly study the Parisi formula for the model, instead obtaining the different formula \eqref{eq:TAPFEnoH}. However the formula \eqref{eq:TAPFEnoH} expresses that the free energy is replica symmetric, since the maximization is over only one matrix $\tilde{\bQ}$.	
    
    \subsection{Outline of proof}\label{subsec: proof outline}

    The starting point of the proof is the computation of the free energy at high temperature in the absence of external field. When $n=1$ (with the unique possibility $\bQ=1$ as the constraint) the annealed free energy is $\frac{1}{2}\beta^2$, and this is also the quenched free energy if the Hamiltonian is at high temperature, which is the case if $\beta \le \frac{1}{\sqrt{2}}$. When $n\ge2$ with a constraint $\bQ$ the annealed free energy turns out to be
    $\frac{1}{2} \beta^\Trans \bQ^{\odot 2} \beta$ (after subtracting the normalizing factor  $\frac{1}{2}\log |\bQ|$ corresponding to log-scale volume of spin vectors that satisfy the constraint; see Lemmas \ref{lem:volumeconstraint}, \ref{lem:firstmoment}). Similarly this is also the quenched free energy if the Hamiltonian is at high temperature, which turns out to be the case if $\|\bb^{\frac{1}{2}} \bQ \bb^{\frac{1}{2}} \|_2 \le \frac{1}{\sqrt{2}}$. As is well-known, these properties of the model with $n=1$ can be verified using a second moment method \cite[Section 2.2]{talagrand2003spin}. In this paper we find that a second moment computation also gives the aforementioned properties of the model with $n\ge2$, though the second moment computation is more challenging (see Lemmas \ref{lem:firstmoment}, \ref{lem:secondmoment} and Propositions \ref{prop:sec_mom_maximized}, \ref{prop:freenergy-whileplefka}). As an aside, note that the aforementioned claim about the quenched free energy is the special case $h=0$ and $\|\bb^{\frac{1}{2}} \bQ \bb^{\frac{1}{2}} \|_2 \le \frac{1}{\sqrt{2}}$ of \eqref{eq:TAPFEnoH}, in which it can be seen that the maximizer is $\tilde{\bQ}=0$.

    Armed with this knowledge of the high temperature phase, the proof of Theorem \ref{thm:TAPformula} splits into a lower and an upper bound for $F^\varepsilon_N = F^\varepsilon_N(\beta,\bh, \bQ)$, both of which proceed by estimating the partition function integral restricted to certain subsets of $\mathcal{S}_{N-1}^n$ that are neighborhoods of a magnetization vector $\bm{m}$. That is, for each such subset $A(\bm{m}) \subset \mathcal{S}_{N-1}^n$ we estimate $\int_{A(\bm{m}) \cap \bQ_\varepsilon} e^{f(\bm{m})} d\bs$ where $f(\bs) = \sum_{k=1}^{n}\left(\beta_k H_{N}\left(\sigma^{k}\right)+ N h^{k}\cdot\sigma^{k}\right)$. We normalize the integral, subtract the centering term $f(\bm{m})$ and take the $\log$ to obtain  

\begin{align*}
\begin{array}{ccccccc}
\rlap{\textrm{   (I)}}
&&&{\color{white}....}\log\int_{A(\bm{m}) \cap \bQ_\varepsilon}e^{f\left(\sigma\right)}d\bs=&&&{\color{white}\Bigg|}
\\
\rlap{$\overbrace{\hspace{0.95\textwidth}}$}&&&&&&
\\
& \log \int_{A(\bm{m}) \cap \bQ_\varepsilon}1d\bs &  + 
    &  f\left(\bm{m}\right) &
     + 
    & \log \tfrac{ \int_{A(\bm{m}) \cap \bQ_\varepsilon}e^{f\left(\bs\right)-f\left(\bm{m}\right)}
    \ d\bs}{\int_{A(\bm{m}) \cap \bQ_\varepsilon} 1 d\bs} & 
\\
\rlap{These terms each give rise to one of the terms of $\TAP(\bm{m})$, through the approximations}{\color{white}\frac{\bigg|}{\bigg|}}&&&&&&
\\
&{\color{white}|}{\begin{rotate}{90}$\approx$\end{rotate}}& &{\color{white}|}{\begin{rotate}{90}$=$\end{rotate}}& &{\color{white}|}{\begin{rotate}{90}$\approx$\end{rotate}}& \\
& \frac{N}{2}\log|\bQ-\bm{m}\bm{m}^{T}| & + 
& \sum_{k=1}^{n}\left(\beta_k H_{N}\left(m^{k}\right)+N h^{k}\cdot m^{k}\right) & 
+ & \frac{N}{2}\beta^{\Trans}\left(\bQ-\bm{m}\bm{m}^{T}\right)^{\odot2}\beta & 
\\
\rlap{$\underbrace{\hspace{0.95\textwidth}}$}&&&&&&
\\
\rlap{\textrm{   (II)}}
&&&{\color{white}.....}=F_{{\rm TAP}}\left(\bm{m}\right)&&&{\color{white}\Bigg|}
\\
\rlap{Furthermore each term has the natural interpretation}
\vspace{-20pt}
{\color{white}\frac{|}{\bigg|}}&&&&&&
\\
& \text{Entropy} & + 
& \text{Local mean energy} & 
+ &
\begin{array}{c}
     \text{Local free energy} \\
     \text{(Onsager term)} 
\end{array}
&
\\ 
\vspace{-20pt}
\end{array}
\end{align*}
as we now explain. 

Indeed the first term in \textrm{(I)} is precisely the $\log$-volume of $\bs$ that lie in $A(\bm{m})$ and satisfy the constraint given by $\bQ$, and is thus an entropy. The neighborhood $A(\bm{m})$ is chosen essentially as a subset of the ``slice'' passing through $\bm{m}$, i.e. the hyperplane with normal $\bm{m}$ passing through $\bm{m}$ intersected with $\mathcal{S}_{N-1}^n$. Such a slice turns out to have $\log$-volume approximately given by $\tfrac{N}{2}\log|\bQ - \bbm\bbm^\Trans|$ (i.e. by the first term in \textrm{(II)}), and the subset we choose retains enough of the volume of the slice to have approximately the same $\log$-volume.

The centering term $f(\bm{m})=\sum_{k=1}^{n}\left(\beta H_{N}\left(m^{k}\right)+N h^{k}\cdot m^{k}\right)$ of \textrm{(I)}, \textrm{(II)} represents the ``local'' mean energy on $A(\bm{m})$.

For the last term of \textrm{(I)} we use the knowledge of the high temperature phase of the first paragraph of this subsection. The identity
$$H_{N}({\sigma})=H_{N}({m})+\nabla H_{N}({m})\cdot(\sigma-m)+ H_{N}(\sigma-m)$$
valid for all $m,\sigma\in \mathbb{R}^N$ implies that
$$
f\left(\bs\right)-f\left(\bm{m}\right)=\sum_{k=1}^{N}\underset{\text{effective external fields}}{\underbrace{(\beta_k\nabla H_{N}({m^{k}})+h^{k})}}\cdot(\sigma^{k}-m^{k})+\underset{\text{effective Hamiltonian}}{\sum_{k=1}^{N}\underbrace{\beta_k H_{N}(\sigma^{k}-m^{k})}}.
$$
From this one sees that the last term in \textrm{(I)} can be interpreted as the free energy of an effective Hamiltonian on the spin configuration space $A(\bm{m}) \cap \bQ_\varepsilon$ subject to effective external fields. In the proof we construct the sets $A(\bm{m})$ so that the effective external field term vanishes for $\sigma \in A(\bm{m})$ (in the easiest case, simply by intersecting the slice with a hyperplane with normal given by the effective external field). Furthermore after normalizing $\sigma^k-m^k$ it turns out that the recentered Hamiltonian is essentially the original Hamiltonian with an effective constraint $\hat{Q}(\bm{m})_{ij}=(\bQ - (\bmm\bbm^\Trans))_{ij}/(\sqrt{1-|m^i|^2}\sqrt{1-|m^j|^2})$ subject to an effective temperature $\bb_{\bbm} = (\beta_1(1-|m^1|^2),...,\beta_n(1-|m^n|^2))$. Therefore applying the approximations for the high temperature free energy in the first paragraph of the subsection one obtains that if $\|\bb_{\bm{m}} \hat{\bQ}(\bm{m})\bb_{\bm{m}}\|_2 \le \frac{1}{\sqrt{2}}$ then the third term of \textrm{(I)} can be approximated by $\tfrac{N}{2}\beta_{\bm{m}}^{\Trans}\hat{\bQ}(\bm{m})^{\odot 2}\beta_{\bm{m}}$. Since $\|\bb_{\bm{m}} \hat{\bQ}(\bm{m})\bb_{\bm{m}}\|_2=\| \bb^{\frac{1}{2}} (\bQ - \tilde{\bm{Q}}) \bb^{\frac{1}{2}} \|_{2}$ the former condition is precisely Plefka's condition, and since $\tfrac{N}{2}\beta_{\bm{m}}^{\Trans}\hat{\bQ}(\bm{m})^{\odot 2}\beta_{\bm{m}}=\frac{N}{2}\beta^{\Trans}\left(\bQ-\bm{m}\bm{m}^{T}\right)^{\odot2}\beta$ the latter is precisely the approximation of the last term of \textrm{(I)} by the Onsager term in \textrm{(II)}.

This justifies the approximation $\log\int_{A(\bm{m}) \cap Q_{\varepsilon}} e^{f(\bs)}d\bs \approx \TAP(\bm{m})$ provided Plefka's condition holds for $\bm{m}$.

Finally, it turns out that only $\bm{m}$ satisfying Plefka's condition are relevant. Indeed, in Section~\ref{section: lower bound} we prove the lower bound for $F_N^{\varepsilon}$ by simply only considering $\bm{m}$ that satisfy Plefka's condition, and deduce that $F_N^{\varepsilon}$ is lower bounded by $\TAP(\bbm)$ for any $\bm{m}$ that satisfies the condition.

The central difficulty in proving the upper bound for $F_N^{\varepsilon}$ in Section~\ref{section: upper bound} is that we cannot a priori ignore $\bm{m}$ that do not satisfy Plefka's condition. Instead we approximate the recentered Hamiltonian by one that is in some sense always at high temperature, even when Plefka's condition is not satisfied. This gives rise to an upper bound of $F_N^{\varepsilon}$ in terms of a modified TAP free energy which has a different Onsager term. We then show that any maximizer of this modifed TAP free energy in fact must satisfy Plefka's condition, and that in this case its Onsager term is close to the usual Onsager term.

The above constitutes a multidimensional ($n\ge2$) adaption of the method (for $n=1$) in \cite{belius-kistler} (elements of the above ideas are also used by TAP \cite{thoulessSolutionSolvableModel1977}, Bolthausen \cite{bolthausenIterativeConstructionSolutions2014, bolthausenMoritaTypeProof2018} and Subag and collaborators \cite{subagGeometryGibbsMeasure2017,subagFreeEnergyLandscapes2020,chenGeneralizedTAPFree2022,subagFreeEnergySpherical2021}).

Lastly in Section~\ref{section: solution variational problem}, we express the ground state of the Hamiltonian as a finite dimensional variational problem over positive semi-definite matrices using the method of Lagrange multipliers. The resulting variational problem can be solved explicitly yielding the closed form representation in Theorem~\ref{thm:groundstate}.

    \section{Preliminaries}
    We denote constants, whose value may change from line to line or even in the same expression, by $c$. They may depend on the number of replicas $n$, but are independent of all other parameters unless otherwise stated.
    
    At certain points in the proof we will use the standard fact that
	\begin{equation}\label{eq: hamilt UB}
	    \lim_{N\to\infty}\mathbb{P}\left( \sup_{m \in \mathbb{R}^N:|m|\le1} |H_N(m)| \le cN, \sup_{m \in \mathbb{R}^N:|m|\le1} |\nabla H_N(m)| \le cN\right)=1.
	\end{equation}
	This follows for instance by writing $\bJ = (J_{i,j})_{i,j=1,\ldots,N}$ so that $\frac{\bJ+\bJ^{\Trans}}{2}$ is a GOE random matrix and $H_N(m)=m^{\Trans} \frac{\bJ+\bJ^{\Trans}}{2} m$, and noting that $\nabla H_N(m)=(\bJ+\bJ^{\Trans})m$ and 
	 \begin{equation}\label{eq: GOE spec norm}
	    \lim_{N\to\infty}\mathbb{P}\left( \|\bJ+\bJ^{\Trans}\|_2 \le cN \right)=1.
	\end{equation}
	We will also use that writing $N\lambda_1<\ldots<N\lambda_N$ for the eigenvalues of $\frac{\bJ+\bJ^{\Trans}}{2}$ we have
	\begin{equation}\label{eq: gap}
	    \max_{i=1,\ldots,N-1} |\lambda_{i+1}-\lambda_i| \overset{\mathbb{P}}{\to} 0\text{ as }N\to\infty.
	\end{equation}
	Finally for the upper bound we will use that 
	\begin{equation}\label{eq: eval class location}
	\max_{i=1,\ldots,N} |\lambda_i-\theta_{i/N}| \overset{\mathbb{P}}{\to} 0\text{ as }N\to\infty,
	\end{equation}
	(see \cite[Theorem 2.2]{ev_rigidity}) where $\theta_{i/N}$ are the classical locations
	\begin{equation}\label{eq:classical}
		\theta_{i/N} = \inf \bigg\{ \theta : \int_{-\sqrt{2}}^\theta d\mu_{\rm{sc}}(x) = \frac{i}{N} \bigg\},
	\end{equation}
	defined in terms of the semi-circle distribution
	\begin{equation}\label{eq:semicircledistribtuion}
		d\mu_{\rm{sc}}(x) = \frac{1}{\pi} \sqrt{2  - x^2} \1_{ [-\sqrt{2}, \sqrt{2}] }(x) dx.
	\end{equation}
    It follows from \eqref{eq:classical} that
	\begin{equation}\label{eq:classical error}
	   \lim_{N\to\infty}\lim_{\varepsilon\to0}\sup_{\left|i-j\right|\le\varepsilon N}\left|\theta_{i/N}-\theta_{j/N}\right|=0.
	\end{equation}
	Note that \eqref{eq: GOE spec norm} and \eqref{eq: gap} are 
 consequences of \eqref{eq: eval class location}-\eqref{eq:classical error}.

	\section{Lower bound}\label{section: lower bound}
 
	In this section, we will prove the following lower bound of the free energy.

	\begin{proposition}[TAP lower bound]\label{prop: TAP LB} Let $n\ge1$ and $\bQ \in [-1,1]^{n\times n}$ be positive definite with $Q_{k,k}=1$ for $k=1,\ldots,n$.	Let $h_1,\ldots,h_n\in[0,\infty)$ and $h^1,\ldots,h^n$ be a sequence of vectors with $h^k \in \R^N$ and $|h^k| = h_k$. Then there exists a 
    $c=c(n,h_1,\ldots,h_n,\bQ)>0$ 
    such that for all $\varepsilon>(0,c^{-1})$
		\begin{equation}\label{eq: LB claim}
    		\lim_{N\rightarrow\infty}
    		\mathbb{P} \left( 
    		F_N^\varepsilon(\beta, \bh, \bQ) \ge 
    		\frac{1}{N} \sup_{\bbm \in {\rm{Plef}}_N(\bQ,\beta) } \TAP(\bbm) - c \sqrt{\varepsilon} \right) = 1.
		\end{equation}
	\end{proposition}

        To prove this we first compute the free energy at high temperature in the absence of external field using the second moment method in Subsection \ref{subsec: sec mom}. Then in Subsection \ref{subsec:lwbdextfield} we consider the model with external field  at arbitrary temperature, and as described in Subsection \ref{subsec: proof outline} proceed by fixing a $\bm{m}$ that satisfies Plefka's condition, constructing a set $A(\bm{m})$ (see \eqref{eq: A def}) that is ``centered around'' $\bm{m}$, recentering the Hamiltonian around this $\bm{m}$ (see \eqref{eq:recentering with ext field}) and estimating the free energy of the recentered Hamilontian on the set $A(\bm{m})$ (see \eqref{eq: recent Hamilt LB}).

	\subsection{Free energy without external field}\label{subsec: sec mom}
	
	\noindent
	Let us define
	\begin{equation}\label{eq: part func}
		Z^\varepsilon_N(\beta, \bh, \bQ)
		=
		\int_{\bQ_\varepsilon} e^{\sum_{k = 1}^n \beta_k H_N(\sigma^k) + N  h^k \sigma^k} \, d\bs \overset{\eqref{def:FE}}{=} e^{N F_N^\varepsilon(\beta, \bh, \bQ)}.
	\end{equation}
	The goal of this subsection is to use the second moment method on $Z^\varepsilon_N(\beta, 0, \bQ)$ to show that it concentrates as $N\rightarrow\infty$.
	In the first lemma of this section we show that the volume of the $\bQ$-constrained $n$-fold product of spheres is approximately $\frac{1}{2} \log |\bQ|$ at exponential scale, with which we can calculate the moments of $Z^\varepsilon_N(\beta, 0, \bQ)$.
	\\ \\
	Recall that $\|\bA\|_2$ denotes the spectral norm of $\bA$.
	\begin{lemma}[Constrained volume]\label{lem:volumeconstraint}
		Let $n\ge1$. There is a constant $c>0$ such that for all symmetric positive semi-definite $\bQ \in [-1,1]^{n\times n}$ with $1$'s on the diagonal and all $\varepsilon \in (0,c^{-1})$ and $N\ge c\varepsilon^{-1}$ we have 
		\begin{equation}\label{eq: constrained volume main}
		\bigg |
		\frac{1}{N} \log \int 
		    \1_{\left\{\|\bs \bs^\Trans - \bQ\|_\infty \leq \varepsilon\right\}}
		    \, d \bs
		-
		\frac{1}{2} \log |\bQ|   
		\bigg |
		\le c (1+\|\bQ^{-1}\|_2) \varepsilon,
		\end{equation}
		and
		\begin{equation}\label{eq: constrained volume small eval}
		    \frac{1}{N} \log \int \1_{\{\|\bs \bs^\Trans - \bQ\|_\infty \leq \varepsilon\}} \, d \bs \le \frac{1}{2}\log\left| \varepsilon \bI + \bQ \right|+c.
		\end{equation}
	\end{lemma}

	\begin{proof}
		
		Let $u_{i,k},k=1,\ldots,n,i=1,\ldots,N$ be \iid standard normal
		random variables, and let $u^{k}=\left(u_{1,k},\ldots,u_{n,k}\right)\in\mathbb{R}^{N}$
		and $u_{i}=\left(u_{i,1},\ldots,u_{i,n}\right)\in\mathbb{R}^{n}$.
		Note that the conditional law of $u^{k}/\left|u^{k}\right|$ on the event $|\frac{|u^k|}{\sqrt{N}} - 1| < \varepsilon$ is the same as the law of $\sigma_{k}$ for $k=1,\ldots,n$, so
		\begin{align*}
			&\phantom{{}={}} \frac{1}{N} \log \int \1_{\{\|\bs \bs^\Trans - \bQ\|_\infty \leq \varepsilon\}} \, d \bs\\
			&=\frac{1}{N} \log \mathbb{P}\left(\max_{k,l}\left|\frac{u^{k}\cdot u^{l}}{\left|u^{k}\right|\left|u^{l}\right|}-Q_{k,l}\right|\le\varepsilon \;\middle|\; \sup_{k} \left|\frac{|u^k|}{\sqrt{N}} - 1\right| < \varepsilon \right)\\
			&	=\frac{1}{N} \log \mathbb{P}\left(\max_{k,l}\left|\frac{u^{k}\cdot u^{l}}{\left|u^{k}\right|\left|u^{l}\right|}-Q_{k,l}\right|\le\varepsilon,\sup_{k} \left|\frac{|u^k|}{\sqrt{N}} - 1\right| < \varepsilon \right) - \frac{1}{N} \log \mathbb{P}\left( \sup_{k} \left|\frac{|u^k|}{\sqrt{N}} - 1\right| < \varepsilon  \right).
		\end{align*}
		By the Chebyshev inequality $\mathbb{P}\left( \left|\frac{|u^k|}{\sqrt{N}} - 1\right| > \varepsilon \right) \le \frac{c}{\varepsilon^2N^2}$, which implies the last term is bounded by $c \varepsilon$ if $c$ is large enough and $N\ge c\varepsilon^{-1}$,
		so it suffices to control the probability of the event
		\[
		A = \bigg\{ \max_{k,l}\left|\frac{u^{k}\cdot u^{l}}{\left|u^{k}\right|\left|u^{l}\right|}-Q_{k,l}\right|\le\varepsilon,\sup_{k} \left|\frac{|u^k|}{\sqrt{N}} - 1\right| < \varepsilon  \bigg\}.
		\]
		We apply the standard proof of Cram\'er's theorem to the \iid vectors $u_1$,\ldots,$u_N$, taking care to obtain a bound that is uniform in $\bQ$. 
		There exists constants $c_{1},c_{2}$ such that for all $\varepsilon$ smaller than some constant and all $\bQ$ with $\|\bQ\|_\infty\le1$
		\[
		\left\{ \max_{k,l}\left|u^{k}\cdot u^{l}-NQ_{k,l}\right|\le c_{1}\varepsilon N\right\} \subseteq A \subseteq \left\{\max_{k,l} \left|u^{k}\cdot u^{l}-NQ_{k,l}\right|\le c_{2}\varepsilon N\right\}.
		\]
		
		We begin with the upper bound. For all symmetric $n\times n$ matrices $\bL$
		\[
		\mathbb{P}\left(\max_{k,l}\left|u^{k}\cdot u^{l}-NQ_{k,l}\right|\le c_{2}\varepsilon N\right)\le\mathbb{E}\left[\exp\left(\sum_{k,l = 1}^n \Lambda_{k,l} u^{k}\cdot u^{l}\right)\right]e^{-N\sum_{k,l = 1}^n \Lambda_{k,l} Q_{k,l} + c_{2}\varepsilon n^2 \|\bL\|_\infty N}.
		\]
		If $2\bL<\bI$, it holds that
		\[
		\mathbb{E}\left[\exp\left(\sum_{k,l = 1}^n \Lambda_{k,l}u^{k}\cdot u^{l}\right)\right]=\mathbb{E}\left[\exp\left(\sum_{i = 1}^N \left(u_{i}\right)^{\Trans}\bL u_{i}\right)\right]=\left(\mathbb{E}\left[\exp\left(\left(u_{1}\right)^{\Trans}\bL u_{1}\right)\right]\right)^{N}=\left|\bI-2\bL\right|^{-\frac{N}{2}}.
		\]
		Thus for all such $\bL$
		\begin{equation}\label{eq:upboundlambda}
		\mathbb{P}\left(\max_{k,l}\left|u^{k}\cdot u^{l}-NQ_{k,l}\right|\le c_{2}\varepsilon N\right)\le\exp\left(-\frac{N}{2}\log\left|\bI-2\bL\right|-N\sum_{kl}\Lambda_{k,l}Q_{k,l}+c_2 \varepsilon n^2 \| \bL \|_\infty  N\right).
		\end{equation}
		The non-error terms on the r.h.s are minimized by choosing $\bL=\frac{\bI-\bQ^{-1}}{2}$, for
		which $-\frac{1}{2}\log\left|\bI-2\bL\right|=\frac{1}{2}\log\left|\bQ\right|$
		and 
		\begin{equation}\label{eq: Lambda Q}
		\sum_{kl}\Lambda_{k,l}Q_{k,l}=\text{Tr}\left(\bL \bQ\right)=\text{Tr}\left(\frac{\bQ-\bI}{2}\right)=0.
		\end{equation}
		Thus we have that
		\[
		\mathbb{P}\left(\max_{k,l}\left|u^{k}\cdot u^{l}-NQ_{k,l}\right|\le c_{2}\varepsilon N\right)\le\exp\left(\frac{N}{2}\log\left|\bQ\right| + c (1+ \|\bQ^{-1}\|_2) \varepsilon  N\right),
		\]
		since $c_2 n^2 \| \bL \|_\infty \le c (1+ \|\bQ^{-1}\|_2)$ for a large enough $c$ depending only on $n$. This proves the upper bound of \eqref{eq: constrained volume main}.
		
		To obtain \eqref{eq: constrained volume small eval} let $\bL=\frac{\bI - (\bQ+\varepsilon \bI)^{-1}}{2}$ and note that then $\sum_{kl}\Lambda_{k,l}Q_{k,l}=\text{Tr}\left(\bL \bQ\right)=\text{Tr}\left(\frac{\bQ-(\bQ + \varepsilon \bI)^{-1}\bQ}{2}\right)\ge-\frac{n}{2}$ and $\| \bL \|_\infty \le c\varepsilon^{-1}$.
		
		For the lower bound of \eqref{eq: constrained volume main} we use the change of measure 
		\[
		\frac{d\mathbb{Q}}{d\mathbb{P}}=\prod_{i=1}^N\frac{\exp\left(\sum_{kl}\Lambda_{k,l}u^{k}\cdot u^{l}\right)}{\left|\bI-2\bL\right|^{-1/2}}= \prod_{i=1}^N\frac{\exp\left(\sum_{i}\left(u_{i}\right)^{\top}\bL u_{i}\right)}{\left|\bI-2\bL\right|^{-1/2}}
		\]
		for $\bm{\bL}=\frac{\bI-\bQ^{-1}}{2}$. Under the measure $\mathbb{Q}$ the $u_{i}$
		are i.i.d. centered Gaussian vectors in $\mathbb{R}^{n}$ with covariance
		$\bQ$. We have
		\begin{equation}
			\begin{array}{l}
				\mathbb{P}\left(\left|u^{k}\cdot u^{l}-NQ_{k,l}\right|\le c_{1}\varepsilon N\right)\\
				=  \mathbb{Q} \left(\1_{\{ \left|u^{k}\cdot u^{l}-NQ_{k,l}\right|\le c_{1}\varepsilon N \}}\frac{d\mathbb{P}}{d\mathbb{Q}}\right)\\
				\ge  \mathbb{Q}\left(\max_{k,l}\left|u^{k}\cdot u^{l}-NQ_{k,l}\right|\le c_{1}\varepsilon N\right)\frac{\exp\left(-\sum_{kl}\Lambda_{k,l}Q_{k,l}-c_1 n^2 \varepsilon \|\bL\|_\infty N\right)}{\left|\bI-2\bm{\bL}\right|^{N/2}}\\
				\overset{\eqref{eq: Lambda Q}}{\ge} \mathbb{Q}\left(\max_{k,l} \left|\frac{1}{N}\sum_{i}u_{i}^{k}u_{i}^{l}-Q_{k,l}\right|\le c_{1}\varepsilon\right)\exp\left(\frac{N}{2}\log\left|\bQ\right|-c (1+ \|\bQ^{-1}\|_2) \varepsilon N\right).
			\end{array}\label{eq: lower bound}
		\end{equation}
		Using a union bound and the Chebyshev inequality (recall $\mathbb{Q}\left(u_{i}^{k}u_{i}^{l}\right)=Q_{k,l}$)
		we obtain
		\begin{equation}
			\begin{array}{ccl}
				\mathbb{Q}\left(\left\{\max_{k,l}\left|\frac{1}{N}\sum_{i}u_{i}^{\Trans}u_{i}-Q_{k,l}\right|\le c_{1}\varepsilon\right\}^{c}\right) & \le & \sum_{k,l}\mathbb{Q}\left(\left|\frac{1}{N}\sum_{i}u_{i}^{k}u_{i}^{l}-Q_{k,l}\right|\ge c_{1}\varepsilon\right)\\
				& \le & \sum_{k,l}\frac{\text{Var}_{\mathbb{Q}}\left(\frac{1}{N}\sum_{i}u_{i}^{k}u_{i}^{l}\right)}{c_{1} 2 \varepsilon^2}\\
				& = & \sum_{k,l}\frac{\text{Var}_{\mathbb{Q}}\left(u_{1}^{k}u_{1}^{l}\right)}{c_{1}^2\varepsilon^2 N}.
			\end{array}\label{eq: chebyshev}
		\end{equation}
		Now crudely bounding
		\[
		\text{Var}_{\mathbb{Q}}\left(u_{1}^{k}u_{1}^{l}\right)\le \mathbb{Q}\left(\left(u_{1}^{k}u_{1}^{l}\right)^{2}\right)\le\sqrt{\mathbb{Q}\left(\left(u_{1}^{k}\right)^{4}\right)}\sqrt{\mathbb{Q}\left(\left(u_{1}^{l}\right)^{4}\right)}\le c,
		\]
		where the last constant is independent of $\bQ$ since $u_{1}^{k}$
		is Gaussian with variance $Q_{kk}=1$ under $\mathbb{Q}$, for all
		$k$. Thus provided $\varepsilon$ is smaller than some constant depending
		only on $n$ the r.h.s. of (\ref{eq: chebyshev}) is at most $\frac{1}{2}$,
		and so from (\ref{eq: lower bound}) it follows that
		\[
		\mathbb{P}\left(\left|u^{k}\cdot u^{l}-NQ_{k,l}\right|\le c_{1}\varepsilon N\right)\ge\exp\left(\frac{N}{2}\log\left|\bQ\right|-c (1+ \|\bQ^{-1}\|_2) \varepsilon N-c\right),
		\]
		giving the lower bound of \eqref{eq: constrained volume main}.
	\end{proof}

	We now compute the first moment, or equivalently the annealed free energy. Recall that $\bA \ge \delta \bI$ means that all eigenvalues of $\bA$ are greater than $\delta$. We also use the notation $\bA^{\odot 2} = \bA \odot \bA = (A_{k,l}^2)_{k,l=1,\ldots,n}$ to denote the Hadamard square of the entries of $\bA$.
	\begin{lemma}[First moment; Annealed free energy in absence of external field]\label{lem:firstmoment}
		Let $n\ge1$. For all $\delta\in(0,1), C>0$ there exists a constant $c=c(\delta, C)>0$ such that for all $\varepsilon$ less than a universal constant, $|\beta| \le C$ and $N\ge c(\delta,\varepsilon)$ we have
		\begin{equation}\label{eq:firstmomentlemma}
			\sup_{\bQ \ge \delta \bI}
			\bigg |
			\frac{1}{N} \log \mathbb{E}[Z^\varepsilon_N(\beta, 0, \bQ)]
			- 
			\left(\frac{1}{2}\beta^\Trans \bQ^{\odot 2} \beta
			+ \frac{1}{2} \log |\bQ|\right)
			\bigg |	
			\le  c \varepsilon,
		\end{equation}    
	    where the supremum is taken over all symmetric $\bQ \in [-1,1]^{n\times n}$ with $1$'s on the diagonal.
	\end{lemma}
	\begin{proof}
		We have
		\begin{equation}
			\mathbb{E}[Z^\varepsilon_N(\beta, 0, \bQ)]
			\overset{\eqref{eq: part func}}{=}
			\mathbb{E}  \int_{\bQ_\varepsilon} e^{\sum_{k = 1}^n \beta_k H_N(\sigma^k)} d\bs 
			= \int_{\bQ_\varepsilon}  \E\left[\exp\left(
			\sum_{k=1}^n \beta_k H_N(\sigma^k)
			\right)\right] d \bs.
			\numberthis\label{firstmoment}
		\end{equation}
		Since the Hamiltonian is a sum of Gaussians for fixed $\bs$ we have for $\bs \in \bQ_\varepsilon$
		
		\begin{equation}\label{eq: expectation}
		    \E\left[\exp\left(
				\sum_{k=1}^n \beta_k H_N(\sigma^k)
				\right)\right]
		    =	\exp\left(\frac{1}{2}\Var\left(
				\sum_{k=1}^n \beta_k H_N(\sigma^k)
				\right)\right).
	    \end{equation}
	Since
	$\E[ H_N(\sigma)H_N(\sigma') ] = N (\sigma \cdot \sigma')^2$
	we have
	\begin{equation}\label{eq: var ident}
		\Var\left( \sum_{k=1}^n \beta_k H_N(\sigma^k) \right) 
		= \sum_{k,\ell=1}^n \beta_k \beta_\ell (\sigma_k \cdot \sigma_l)^2.
	\end{equation}
	For $\bs \in Q_\varepsilon$ we have
	$|\sum_{k,\ell=1}^n \beta_k \beta_\ell (\sigma_k \cdot \sigma_l)^2 - \sum_{k,\ell=1}^n \beta_k \beta_\ell Q_{k,l}^2|\le c \varepsilon N$ for a constant $c$ depending only on $C$ and $n$, and $\sum_{k,\ell=1}^n \beta_k \beta_\ell Q_{k,l}^2 = \beta^\Trans \bQ^{\odot 2} \beta$.
	Therefore for all $\bs \in \bQ_\varepsilon$
	$$
	\left|\log \E\left[\exp\left(
				\sum_{k=1}^n \beta_k H_N(\sigma^k)
				\right)\right] - \frac{1}{2}\beta^\Trans \bQ^{\odot 2} \beta \right| \le c\varepsilon N.
	$$
		Lemma \ref{lem:volumeconstraint} implies that for $\bQ \ge \delta \bI$
		$$
		\bigg |
		\log \int_{\bQ_\varepsilon}  1 d \bs - 
		\frac{N}{2} \log |\bQ|
		\bigg |
		\le  c \varepsilon N,
		$$
		which completes the proof.
	\end{proof}
	Next we compute the second moment.
	\begin{lemma}[Second moment]\label{lem:secondmoment} Let $n\ge1$. For all $\delta, C>0$ there exists a constant $c=c(\delta, C)>0$ so that for all $\varepsilon \in(0,c^{-1}),|\beta|\le C$ and all $N\ge c(\varepsilon,\delta)$ we have
		\begin{equation}\label{eq:secondmomentlemma}
			\sup_{\bQ \ge \delta \bI}
			\bigg (
			\frac{1}{N} \log \mathbb{E}[Z^\varepsilon_N(\beta, 0, \bQ)^2] 
			-
			\left(\beta^\Trans \bQ^{\odot 2} \beta +
			\sup_{\bA} V(\bA)\right)
			\bigg )
			\le
			c \varepsilon  , 
		\end{equation}    
		where the supremum is taken over all symmetric $\bQ \in [-1,1]^{n\times n}$ with $1$'s on the diagonal and
		\begin{equation}\label{eq:Vfunc}
			V(\bA) = \beta^\Trans \bA^{\odot 2} \beta + \frac{1}{2} \log \begin{vmatrix}
				\bQ &\bA\\
				\bA^\Trans &\bQ
			\end{vmatrix}.
		\end{equation} 
	\end{lemma}
	\begin{proof} 
		We have
		\begin{align*}
			\E[Z^\varepsilon_N(\beta, 0, \bQ)^2] 
			& =
			\mathbb{E} \left[
			\int_{\bQ_\varepsilon}\int_{\bQ_\varepsilon}  e^{\sum_{k = 1}^n \beta_k H_N(\sigma^k)}
			e^{\sum_{k = 1}^n \beta_\ell H_N(\tau^\ell)} d\bs d\bt
			\right]
			\\
			& = 
			\int_{\bQ_\varepsilon}\int_{\bQ_\varepsilon}
			\E\left[ 
			e^{\sum_{k=1}^n \beta_k H_N(\sigma^k) + \sum_{\ell=1}^n \beta_\ell H_N(\tau^\ell)}
			\right]
			d\bs d\bt.
		\end{align*}
		Similarly to in the proof of the previous lemma the inner expectation is a Gaussian exponential moment that satisfies
		$$ 
		    \left| \log \E\left[ 
			e^{\sum_{k=1}^n \beta_k H_N(\sigma^k) + \sum_{\ell=1}^n \beta_\ell H_N(\tau^\ell)}
			\right] - N\beta^\Trans \bQ^{\odot 2} \beta - N \sum_{k,\ell=1}^n \beta_k \beta_\ell \left(\sigma^k\cdot\tau^\ell\right)^2 \right| \le N c \varepsilon ,
		$$
		for all $\sigma,\tau\in \bQ_\varepsilon$, where $c$ depends only on $C,n$.
		
		It remains to prove that 
		\begin{equation}\label{eq: suff to show sec mom}
		\int_{\bQ_\varepsilon}\int_{\bQ_\varepsilon}
		e^{
			N \sum_{k,\ell=1}^n \beta_k \beta_\ell \left(\sigma^k\cdot\tau^\ell\right)^2
		}
		d\bs d\bt
		\le
		\exp\left( 
		N \sup_{\bA} V(\bA) 
		+
		N c \varepsilon
		\right).
		\end{equation}
		By partitioning the space $\left[-1,1\right]^{n\times n}$ into at most $\lceil\frac{1}{2\varepsilon}\rceil^{n^2}$
        subsets of diameter of order $\varepsilon$ one obtains that for all $\varepsilon>0$
        \begin{equation}\label{eq: sup A}
        \begin{array}{l}
            \int_{\bQ_{\varepsilon}}\int_{\bQ_{\varepsilon}}e^{
			N \sum_{k,\ell=1}^n \beta_k \beta_\ell \left(\sigma^k\cdot\tau^\ell\right)^2
		}d\bs d\bt\\
		\le
		\exp\left( N\sup_{A\in\left[-1,1\right]^{n\times n}}\left\{ \beta^{\Trans}\bA^{\odot2}\beta+\frac{1}{N}\log\int_{\bQ_{\varepsilon}}\int_{\bQ_{\varepsilon}}\1_{\{ \|\bs\bt^{\Trans}-\bA\|_{\infty}\le\varepsilon \}} d\bs d\bt 
		\right\}
		+N c \varepsilon
		\right),
	    \end{array}
        \end{equation}
        for a $c$ depending on $C$, for all $N\ge c(\varepsilon)$.
        Note that the second term in the supremum equals
        \[
        \frac{1}{N}\log\int\1_{\left\{ \left\|\bm{\nu}\bm{\nu}^{\Trans}-\left(\begin{matrix}\bQ & \bA\\
        \bA^{\Trans} & \bQ
        \end{matrix}\right)\right\|_{\infty}\le\varepsilon \right\}}d\bm{\nu},
        \]
        where the integral is over $\bm{\nu} \in \mathcal{S}_{N-1}^{2n}$. 
        
        Note that for any matrix $\bB\in\mathbb{R}^{n\times n}$ such that $\|\bB\|_\infty\le1$, we have
        \begin{equation}\label{eq: mat det general}
        \begin{array}{c}
        \|\bB\|_2\le n \text{ and } \delta^{n}\le\left|\bB\right|\le\delta n^{n-1}\text{ if }\delta\ge0\text{ is }\bB\text{'s smallest eigenvalue}.
        \end{array}
        \end{equation}
        
        Fix a $\tilde{\delta}\in\left(0,\frac{1}{2}\right)$ small enough depending
        only on $\delta,n,C$ such that
        \[
        \frac{1}{4}\log(2\tilde{\delta})+\frac{1}{2}\log\left(4n\right)^{2n-1}+n^{2}\max_{i}\beta_{i}^{2}\le n \log \delta \overset{\eqref{eq: mat det general}}{\le} V\left(0\right).
        \]
        Then if $\bA$ is s.t. $\left(\begin{matrix}\bQ & \bA\\
        \bA^{\Trans} & \bQ
        \end{matrix}\right)$ has an eigenvalue smaller or equal to $\tilde{\delta}$ and $\varepsilon \in (0,\tilde{\delta})$ then by \eqref{eq: constrained volume small eval} and \eqref{eq: mat det general}
        \begin{equation}\label{eq: small eval bound}
        \frac{1}{N}\log\int\1_{\left\{ \left\|\bm{\nu}\bm{\nu}^{\Trans}-\left(\begin{matrix}\bQ & \bA\\
        \bA^{\Trans} & \bQ
        \end{matrix}\right)\right\|_{\infty}\le\varepsilon \right\}}d\bm{\nu}\le V\left(0\right)-N\beta^{\Trans}\bA^{\odot2}\beta.
        \end{equation}
        (after possible decreasing $\tilde{\delta}$ further depending on
        the constant in \eqref{eq: constrained volume small eval}). If on the other hand $\left(\begin{matrix}\bQ & \bA\\
        \bA^{\Trans} & \bQ
        \end{matrix}\right)>\delta'\bI$ then \eqref{eq: constrained volume main} implies that
        \begin{equation}\label{eq: no small eval bound}
        \frac{1}{N}\log\int\1_{\left\{ \left\|\bm{\nu}\bm{\nu}^{\Trans}-\left(\begin{matrix}\bQ & \bA\\
        \bA^{\Trans} & \bQ
        \end{matrix}\right)\right\|_{\infty}\le\varepsilon \right\}}d\bm{\nu}\le 
        \frac{1}{2}\log
        \left|\begin{matrix}\bQ & \bA\\
        \bA^{\Trans} & \bQ
        \end{matrix}\right|
        +
        c\tilde{\delta}^{-1}\varepsilon.
        \end{equation}
        The bounds \eqref{eq: sup A}, \eqref{eq: small eval bound} and \eqref{eq: no small eval bound} imply \eqref{eq: suff to show sec mom}.
	\end{proof}
	
	\noindent
	We will show that $ V(\bA)$ is maximized at zero for $\beta$ that lie in
	\begin{equation}\label{eq:HT-condition}
	{\rm{HT}}(\bQ) :=
		\left\{\beta\in\R^n : \|\bb^{\frac{1}{2}} \bQ \bb^{\frac{1}{2}} \|_2 \leq \frac{1}{\sqrt{2}}  \right\},
	\end{equation}
	which is the high temperature region of the model (recall that $\bb = \diag(\beta_1,\dots, \beta_n)$). 
    Together with Lemma \ref{lem:secondmoment} this will imply
	$$
	\frac{1}{N} \log \mathbb{E}[Z^\varepsilon_N(\beta, 0, \bQ)^2] -
	\frac{1}{N} \log \mathbb{E}[Z^\varepsilon_N(\beta, 0, \bQ)]^2
	\le c \varepsilon 
	$$
    for $\beta \in {\rm{HT}}(\bQ)$, with which we can use a second moment method to prove concentration of $Z^\varepsilon_N(\beta, 0, \bQ)$ for such $\beta$.

    In the computation showing that $V(\bA)$ is maximized at zero a different form of the high temperature condition naturally appears.  The next lemma shows that this form is equivalent to the condition in \eqref{eq:HT-condition}. 
	\begin{lemma}[Equivalence of the two forms of high temperature condition]
		For any positive definite symmetric matrix $\bQ$,
		\begin{equation}\label{eq:HTOPform}
		{\rm{HT}}(\bQ) =
		\left\{\beta\in\R^n : \sup_{\|\bB\|_F = 1} 	\| \bb^{\frac{1}{2}} \bQ^{\frac{1}{2}} \bB \bQ^{\frac{1}{2}} \bb^{\frac{1}{2}} \|_F \leq \frac{1}{\sqrt{2}} \right\},
		\end{equation}
  where $\|\cdot\|_F$ denotes the Frobenius norm.
	\end{lemma}
	\begin{proof}
	    
The claim follows once we have shown that
\begin{equation}
\sup_{\|\bB\|_{F}^{2}=1}\|\bb^{1/2}\bQ^{1/2}\bB\bQ^{1/2}\bb^{1/2}\|_{F}=\|\bb^{\frac{1}{2}} \bQ \bb^{\frac{1}{2}} \|_2.\label{eq: suff to show}
\end{equation}
To this end note that
\[
\begin{array}{ccl}
\|\bb^{1/2}\bQ^{1/2}\bB\bQ^{1/2}\bb^{1/2}\|_{F}^{2} & = & \text{Tr}\left(\bb^{1/2}\bQ^{1/2}\bB\bQ^{1/2}\bb^{1/2}\left(\bb^{1/2}\bQ^{1/2}\bB\bQ^{1/2}\bb^{1/2}\right)^{\Trans}\right)\\
 & = & \text{Tr}\left(\bB\bQ^{1/2}\bb \bQ^{1/2}\bB^{\Trans}\bQ^{1/2}\bb \bQ^{1/2}\right)\\
 & = & \|\bB\bQ^{1/2}\bb \bQ^{1/2}\|_{F}^{2}.
\end{array}
\]
Let $\tilde{\bb}$ be the diagonal matrix of eigenvalues of $\bQ^{1/2}\bb \bQ^{1/2}$,
and let $\tilde{\bB}$ denote $\bB$ in the (orthogonal) diagonalizing
basis of $\bQ^{1/2}\bb \bQ^{1/2}$. Then $\|\tilde{\bB}\|_{F}^{2}=\|\bB\|_{F}^{2}$
and $\|\bB\bQ^{1/2}\bb \bQ^{1/2}\|_{F}^{2}=\|\tilde{\bB}\tilde{\bb}\|_{F}^{2}$,
so
\[
\sup_{\|\bB\|_{F}^{2}=1}\|\bb^{1/2}\bQ^{1/2}\bB\bQ^{1/2}\bb^{1/2}\|_{F}^{2}=\sup_{\|\tilde{\bB}\|_{F}^{2}=1}\|\tilde{\bB}\tilde{\bb}\|_{F}^{2}.
\]
Since $\|\tilde{\bB}\tilde{\bb}\|_{F}^{2}=\sum_{i}\left(\sum_{j}\tilde{B}_{i,j}^{2}\right)\tilde{\beta}_{i,i}^{2}$
the r.h.s. clearly equals $\max_{i}\tilde{\beta}_{i,i}^{2}$. Since $AB$ and $BA$ have the same eigenvalues for any square matrices $A,B$, also
\begin{equation}\label{eq: QbQ bQb same eigenvalues}
    \bQ^{1/2}\bb \bQ^{1/2}
    \ \text{ and } \
    \bb^{1/2}\bQ\bb^{1/2}
    \ \text{ have the same eigenvalues,}
\end{equation}
and this proves (\ref{eq: suff to show}).
\end{proof}
	We are now ready to show that $V(\bA)$ is maxmized for $A=0$ when $\beta \in {\rm{HT}}(\bQ)$.
	\begin{proposition}\label{prop:sec_mom_maximized}For any positive definite $\bQ$ and $\beta \in  {\rm{HT}}(\bQ)$ it holds that
		\begin{equation}
			\sup_{\bA} V(A) = V(0).
		\end{equation}
	\end{proposition}
	\begin{proof}
		Using the Schur complement formula
		\[
		\begin{vmatrix}
			\bQ &\bA\\
			\bA^\Trans &\bQ
		\end{vmatrix} = |\bQ| |\bQ - \bA^\Trans \bQ^{-1} \bA|.
		\]
		We have
		$$ | \bQ - \bA^\Trans \bQ^{-1} \bA| =  | \bQ - \bA^\Trans \bQ^{-\frac{1}{2}} \bQ^{-\frac{1}{2}} \bA|
			=  | \bQ - (\bQ^{-\frac{1}{2}} \bA)^\Trans ( \bQ^{-\frac{1}{2}} \bA)|.$$
		By the matrix determinant lemma this equals
		\begin{equation}
			|\bQ| | \bI - (\bQ^{-\frac{1}{2}} \bA) \bQ^{-1} ( \bQ^{-\frac{1}{2}} \bA)^\Trans|
			 =  |\bQ| | \bI - (\bQ^{-\frac{1}{2}} \bA \bQ^{-\frac{1}{2}}) ( \bQ^{-\frac{1}{2}} \bA^\Trans \bQ^{-\frac{1}{2}})^{\Trans}|.
		\end{equation}
		Thus
		\[
		V(\bA) = \beta^\Trans \bA^{\odot 2} \beta + \log | \bQ| + \frac{1}{2} \log| \bI - (\bQ^{-\frac{1}{2}} \bA \bQ^{-\frac{1}{2}}) ( \bQ^{-\frac{1}{2}} \bA^\Trans \bQ^{-\frac{1}{2}})^{\Trans}|.
		\]
		Now make the change of variables $\bB = \bQ^{-\frac{1}{2}} \bA^\Trans \bQ^{-\frac{1}{2}} \Leftrightarrow \bQ^{\frac{1}{2}} \bB^\Trans \bQ^{\frac{1}{2}} =  \bA$ to obtain	\begin{equation}\label{eq: rewritten VA}
			V(\bA) = \beta^\Trans (\bQ^{\frac{1}{2}} \bB^\Trans \bQ^{\frac{1}{2}})^{\odot 2} \beta
			+ \log|\bQ|
			+ \frac{1}{2} \log| \bI - \bB^\Trans \bB|.
		\end{equation}
		It thus suffices to show that the right-hand side is maximized for $\bB=0$.
		
		To this end we first optimize along rays by fixing $\bB$ and considering
		\[
		v(t) = V( \sqrt{t} \bB) = t \beta^\Trans (\bQ^{\frac{1}{2}} \bB^\Trans \bQ^{\frac{1}{2}})^{\odot 2} \beta + \log|\bQ| + \frac{1}{2} \log| \bI - t \bB^\Trans \bB|, t\ge0.
		\]
		The functional $v(t)$ is clearly concave in $[0,\infty)$ because the first term is linear in $t$, the second term is constant, and the last term is concave in $t$ (for instance by diagonalizing $\bB \bB^{\Trans}$). Thus to show that $v(t)$ has a global maximum at $0$, it suffices to show that $v'(0) \leq 0$.
		
		We have
		$$v'(0) =  \beta^\Trans (\bQ^{\frac{1}{2}} \bB^\Trans \bQ^{\frac{1}{2}})^{\odot 2} \beta + \frac{1}{2} \frac{d}{dt} \log| \bI - t \bB^\Trans \bB| \Big|_{t = 0}.$$
		Since
		$$ \frac{d}{dt} \log| \bI - t \bB^\Trans \bB| \Big|_{t = 0} =- \Tr (\bI - t \bB^\Trans \bB)^{-1} \bB^\Trans \bB) \Big|_{t = 0} = -\Tr(\bB^\Trans \bB) = -\|\bB\|_F^2,$$
		and $w^\Trans \bA^{\odot 2}w = \| \diag(w)^{\frac{1}{2}} \bA \diag(w)^{\frac{1}{2}}\|_F^2$ for any vector $w$ and matrix $\bA$
		we obtain
		$$ v'(0) = \| \bb^{1/2} \bQ^{\frac{1}{2}} \bB \bQ^{\frac{1}{2}} \bb^{1/2} \|_F^2 - \frac{1}{2} \| \bB\|_F^2 = \|\bB\|_F^2 \left(\| \bb^{1/2} \bQ^{\frac{1}{2}} \bm{\hat{B}} \bQ^{\frac{1}{2}} \bb^{1/2} \|_F^2 - \frac{1}{2}\right),$$
		where $\bm{\hat{B}}=\bB/\|\bB\|_F$. If $\beta$ satisfies the high temperature condition \eqref{eq:HTOPform} then the r.h.s. is non-negative, so $v'(0)\le0$ and indeed $v(t),t\in [0,\infty)$ is maximized at $t=0$. 
		
		But since this holds for any $\bB$, it must be that the r.h.s. of \eqref{eq: rewritten VA} is maximized when $\bB = \bm0$, so $\bA = \bQ^{\frac{1}{2}}\bB^\Trans \bQ^{\frac{1}{2}} = \bm0$ is the global maximizer of the functional $V(\bA)$. 
	\end{proof}
	\begin{remark}
		Note that for $v$ in the previous proof $v'(0)\le 0$ for all $\bB$ \emph{only if} the condition \eqref{eq:HT-condition} is satisfied, so the reverse of the implication of the Proposition also holds (though we do not need this fact).
	\end{remark}
	
	For the lower bound of the free energy with the second moment method one needs the standard exponential concentration inequality.
	
	\begin{lemma}[Exponential concentration for free energy]\label{lem: exponential concentration}Let $n\ge1$ and $C>0$. There exists a $c=c(C)>0$ such that for all $\varepsilon>0$, $|\beta|\le C$ and $\bQ > 0$ 
		\begin{align*}
			\pP\left(
			\bigg |F^\varepsilon_N(\beta,0,\bQ) - \mathbb{M}\left[F^\varepsilon_N(\beta,0,\bQ)\right] \bigg | \ge t\right)
			\le
			\exp\left(-c t^2 N \right)
			\numberthis\label{eq:concentrationofmeasure},
		\end{align*}
		where $\mathbb{M}$ denotes the median.
	\end{lemma}
	\begin{proof}
		This follows by Gaussian concentration \cite[Theorem 10.17]{concentration-inequalities}, since for all $i,j,\beta,\bQ$
		\[
		\partial_{J_{ij}}F_{N}^{\varepsilon}\left(\beta,0,\bQ\right) \overset{\eqref{def:FE}}{=} \frac{1}{\sqrt{N}} \sum_{k=1}^{n}\beta_{k}\left\langle \sigma_{i}^{k}\sigma_{j}^{k}\right\rangle ,
		\]
		where $\langle\cdot\rangle = {\int_{\bQ_\varepsilon} \cdot \ e^{\sum_{k=1}^n \beta_k H_N(\sigma^k)}d\bs / {\int_{\bQ_\varepsilon}e^{\sum_{k=1}^n \beta_k H_N(\sigma^k)}d\bs}}$ denotes the expectation over the Gibbs measure, so that when $\left|\beta\right|\le C$
		\[
		\left|\nabla_{J}F_{N}^{\varepsilon}\left(\beta,0,\bQ\right)\right|^{2}\le \frac{C^{2}}{N}\sum_{i,j}\sum_{k=1}^{n}\left\langle \sigma_{i}^{k}\sigma_{j}^{k}\right\rangle ^{2}\le \frac{C^{2}}{N} \sum_{i,j}\sum_{k=1}^{n}\left\langle \left(\sigma_{i}^{k}\right)^{2}\left(\sigma_{j}^{k}\right)^{2}\right\rangle = \frac{C^{2}n}{N},
		\]
		implying that the map $J\to F_{N}^{\varepsilon}\left(\beta,0,\bQ\right)$
		is Lipschitz with Lipschitz constant $N^{-1/2}Cn^{1/2}$.
	\end{proof}

	To obtain an estimate for the free energy uniformly over $\beta$ and $\bQ$ (see \eqref{eq:freenergy-withplefka}) we will use the next result.
	
	\begin{lemma}[Lipschitz property of the free energy]\label{lem:lipschitz}
	Let $n\ge1,C>0,\varepsilon > 0$. There exists a $L = L(C)>0$ such that 
		\begin{align*}
			\lim_{N\to\infty}\mathbb{P}\left(
			\forall \bQ > 0 \text{ and } |\beta^1|,|\beta^2| \le C: \
			\left|
			F^\varepsilon_N(\beta^1, 0, \bQ)
			-
			F^\varepsilon_N(\beta^2, 0, \bQ)
			\right|
			\le
			L
			\left|
			\beta^1 - \beta^2
			\right|
			\right)=1.
		\end{align*}
	\end{lemma}
	\begin{proof}
		We have for any $k \in \{1,...,n\}$ that
		\begin{equation}
			\frac{\partial}{\partial \beta^k} F_N^\varepsilon(\beta,0,\bQ)
			\overset{\eqref{def:FE}}{=}
			\langle
			\tfrac{1}{N}\beta_k H_N(\sigma^k)
			\rangle
		\end{equation}
		where $\langle\cdot\rangle$ denotes the expectation over the Gibbs measure as in the previous lemma. Thus by \eqref{eq: hamilt UB} we have for $L=L(C)$ large enough
		$$
		\mathbb{P}\left( \sup_{\bQ,|\beta|\le C}\left|
		\frac{\partial}{\partial \beta^k} F_N^\varepsilon(\beta,0,\bQ)
		\right| \le L \right) \to 1.
		$$
		This implies that $F_N^\varepsilon(\beta,0,\bQ)$ is Lipschitz continuous in $\beta$ with probability tending to one.
	\end{proof}
	
	The next proposition will now combine all previous arguments to show that $F_N^\varepsilon$ concentrates as $N\rightarrow\infty$ if the external field $\bh$ is zero and $\beta$ lies in the high temperature region.
	
	\begin{proposition}[Free energy at high temperature]\label{prop:freenergy-whileplefka}
		Let $n\ge 1$. Let $\delta, C > 0$ be some constants. There exists a $c=c(\delta, C)>0$ such that for all $\varepsilon \in (0,c^{-1})$
		\begin{equation}\label{eq:freenergy-withplefka}
			\lim_{N\to\infty} \P \left( 
			\sup_{\bQ \ge \delta \bI} \sup_{\substack{\beta \in {\rm{HT}}(\bQ) \\ |\beta| \le C}} \bigg | F^\varepsilon_N(\beta, 0, \bQ) - \frac{1}{2} (\beta^\Trans \bQ^{\odot 2} \beta + \log |\bQ| ) \bigg |
			\le c \sqrt{\varepsilon}
			\right) =1,
		\end{equation}
		where the supremum is taken over all symmetric $\bQ \in [-1,1]^{n\times n}$ with $1$'s on the diagonal.
	\end{proposition}
	\begin{proof}
	    Let $\bA^i\in[-1,1]^{n\times n},i=1,\ldots,M,$ with $M \le \lceil \frac{1}{2\varepsilon} \rceil^{n^2}$ such that for all $\bQ \in[-1,1]^{n\times n}$ there exists an $i$ such that
    	\begin{equation}\label{eq:A-Q small}
            \|\bA^{i}-\bQ\|_{\infty}\le\varepsilon, \quad \|(\bA^i)^{\odot2}-\bQ^{\odot2}\|_{\infty}\le2\varepsilon, \quad \left|\left|\bA^i\right|-\left|\bQ\right|\right|\le n^{n+1}\varepsilon.
        \end{equation}
        For this $i$ we have 
		${\bA^i}_{\frac{1}{2}\varepsilon} \subset \bQ_{\varepsilon} \subset {\bA^i}_{2\varepsilon}$
		and thus for all $\beta$
		\begin{equation}\label{eq: inclusion}
		F_N^{\frac{1}{2}\varepsilon}(\beta, 0,\bA^i)
		\ \le \ 
		F_N^\varepsilon(\beta, 0, \bQ) 
		\ \le \ 
		F_N^{2\varepsilon}(\beta, 0,\bA^i).
		\end{equation}
		We also construct a finite sequence $\beta^1,...,\beta^{L}$ with $L \le \lceil \frac{C}{\varepsilon} \rceil^{n}$ such that for each $\beta\in \{b \in \mathbb{R}^n: |b|\le C\}$ there is a $j\le L$ with $|\beta-\beta^j|\le\varepsilon$. Then with probability tending to one  by Lemma \ref{lem:lipschitz} there is for each $\beta$ with $|\beta|\le C$ some $j$ such that
		\begin{equation}\label{eq: F(beta)-F(beta^j)}
		\left| 
		F_N^\varepsilon(\beta, 0, \bQ)  -
		F_N^\varepsilon(\beta^j, 0, \bQ) 
		\right| \le c(C) \varepsilon \ \text{ for all }\bQ,
		\end{equation}
		and for each $\bQ$ with $\bQ>\delta\bI$ and the $i$ such that \eqref{eq: inclusion} holds
		\begin{equation}\label{eq: F(Q)-F(A)}
		    \left|\beta^{\top}\bQ^{\odot2}\beta+\log\left|\bQ\right|-(\beta^j)^{\top}\bA^{i\odot2}\beta^j-\log\left|\bA^i\right|\right|\le c\left(\delta,n,C\right)\varepsilon,
		\end{equation}
provided $\varepsilon$ is small enough depending on $\delta$.
		
		\noindent
		\textit{Upper bound:}
		This implies that if the constant $c$ is chosen large enough depending on $\delta,C,n$ then
		\begin{align*}
			& \mathbb{P}\left(
			\forall \bQ \ge \delta \bI, \ |\beta|\le C: \ Z^\varepsilon_N(\beta,0,\bQ) > \exp\left(
			\frac{N}{2}\left(\beta^\Trans \bQ^{\odot 2} \beta + \log |\bQ| + c\varepsilon\right)
			\right) 
			\right)
			\\
			\le &
			\mathbb{P}\left(
			\exists i=1,...,M , \ j=1,...,L: Z^{2\varepsilon}_N(\beta^j,0,\bA^i) > \exp\left(
			\frac{N}{2}\left((\beta^j)^\Trans (\bA^i)^{\odot 2} \beta^j + \log |\bA^i| + \frac{c}{2}\varepsilon\right)
			\right) 
			\right),
		\end{align*}
		where by Markov's inequality and Lemma \ref{lem:firstmoment} the r.h.s. is bounded by
		\begin{align*}
			&\sum_{i=1}^M \sum_{j=1}^L
			\frac{\mathbb{E}\left[
				Z^{2\varepsilon}_N(\beta^j,0,\bA^i)\right]}{\exp\left(
				\frac{N}{2}\left((\beta^j)^\Trans (\bA^i)^{\odot 2} \beta^j + \log |\bA^i| + \frac{c}{2}\varepsilon\right)
				\right)}
			\le 
			\exp\left(-N\frac{c}{4}\varepsilon\right),
		\end{align*}
		and thus
		\begin{align*}
			\mathbb{P}\left( 
			\forall \bQ \ge \delta \bI, \ |\beta|\le C: \ 
			F_N^\varepsilon(\beta,0,\bQ)  \le \frac{1}{2}(\beta^\Trans \bQ^{\odot 2} \beta) + \frac{1}{2} \log |\bQ| 
			+ c \varepsilon\right) \to 1.
		\end{align*}
		\\
		\textit{Lower bound:}        By the Paley-Zygmund inequality and Lemmas \ref{lem:firstmoment} and \ref{lem:secondmoment} we have for any large enough $c$
        depending on $\delta,C,n$ that for all $\varepsilon\in(0,c^{-1})$
        and $N\ge c\left(\varepsilon,\delta\right)$ and $i,j$
		\begin{equation}
			\begin{array}{ccl}
				& & 
				\pP\left(
				F^{\frac{\varepsilon}{2}}_N(\beta^j,0,\bA^i) > 
        			\frac{1}{2}(\beta^j)^\Trans (\bA^i)^{\odot 2} \beta^j + \frac{1}{2} \log |\bA^i| - \frac{c}{8}\varepsilon\right)
				\\
				& \ge & 
				\mathbb{P} \left( Z_N^{{\frac{\varepsilon}{2}}}(\beta^j,0,\bA^i) >  \frac{1}{2} \mathbb{E}[Z_N^{\frac{\varepsilon}{2}}(\beta^j,0,\bA^i)] \right) 
				\\
				& \ge &
				\frac{1}{4} \frac{\mathbb{E}[Z_N^{\frac{\varepsilon}{2}}(\beta^j,0,\bA^i)]^2}{\mathbb{E}[Z_N^{\frac{\varepsilon}{2}}(\beta^j,0,\bA^i)^2]} 
				\\
				& \ge & \frac{1}{4} e^{-c\varepsilon N}.
			\end{array}
		\end{equation}
		
		Since otherwise there is a contradiction by Lemma \ref{lem: exponential concentration} (after possibly enlarging $c$) this implies that $\mathbb{M}\left(F^{\frac{\varepsilon}{2}}_N(\beta^j,0,\bA^i)\right)\ge \frac{1}{2}(\beta^j)^\Trans (\bA^i)^{\odot 2} \beta^j + \frac{1}{2} \log |\bA^i|-\frac{c}{4}\sqrt{\varepsilon}$
        for all $\varepsilon\in(0,c^{-1})$ and $N\ge N\left(\varepsilon,\delta\right)$, and then another use of Lemma \ref{lem: exponential concentration} implies that
        \[
        \lim_{N\to\infty}\pP\left(
				F^{\frac{\varepsilon}{2}}_N(\beta^j,0,\bA^i) <
        			\frac{1}{2}(\beta^j)^\Trans (\bA^i)^{\odot 2} \beta^j + \frac{1}{2} \log |\bA^i| - \frac{c}{2}\sqrt{\varepsilon}\right)=0\text{ for all }\varepsilon\in(0,c^{-1}),i,j.
        \]
        Then (possibly enlarging $c$ again) we have
		\begin{align}
			& \mathbb{P}\left( \exists \bQ \ge \delta \bI, \ |\beta|\le C: \ 
			F_N^\varepsilon(\beta,0,\bQ)  <\frac{1}{2}(\beta^\Trans \bQ^{\odot 2} \beta) + \frac{1}{2} \log |\bQ| 
			- c \sqrt{\varepsilon}\right) 
			\\
			\le &
			\sum_{i=1}^M \sum_{j=1}^L
			\mathbb{P}\left(
			F^{\frac{\varepsilon}{2}}_N(\beta^j,0,\bA^i) < \exp\left(
			\frac{1}{2} (\beta^j)^\Trans (\bA^i)^{\odot 2} \beta^j + \frac{1}{2} \log |\bA^i| - \frac{c}{2}\sqrt{\varepsilon}\right)
			\right)\to0
		\end{align}
		for all $\varepsilon\in(0,c^{-1})$,	which gives the lower bound.
	\end{proof}

	\subsection{With external field}\label{subsec:lwbdextfield}
	
	We now prove the lower bound at all temperatures in the presence of an external field. We will follow the proof of \cite[Lemma 5]{belius-kistler}. We start by showing that Lemma~\ref{lem:secondmoment} still holds if we restrict the integral in the parition function to the intersection of the product of unit spheres with hyperplanes of high dimension. 

	In the following it will be convenient to denote the integral $\int\cdot d\sigma$
	over the sphere $\mathcal{S}_{N-1}$ and the integral $\int\cdot d{\bs}$
	over $\mathcal{S}_{N-1}^{n}$ by $E\left[\cdot\right]$. For a subspace $U \subset \mathbb{R}^{n \times N}$ let us write $E^{U}$ to denote the expectation/integral with respect to $\bs$ conditioned on $\bs\in U$. 
	
	\begin{lemma}\label{lem:dimensiondrop}
		For all $\delta>0$ it holds that
		\begin{equation}\label{eq: dimensiondrop}
		\mathbb{P}\left(
		{\sup_{\bQ \ge \delta \bI}}
		\sup_{\substack{\beta \in {\rm{HT}}(\bQ)
		}} 
		\sup_{U}
		\Bigg | 
		\frac{1}{N}\log 
		E^{U^\perp} \left[ \1_{\bQ_{\varepsilon}}
		e^{\sum_{k=1}^n \beta_k H_N({\sigma^k})} 
		\right]    
		- \bigg( \frac{\beta^\Trans \bQ^{\odot 2} \beta}{2} + \frac{1}{2} \log|\bQ| \bigg) \Bigg |
		\le c \sqrt{\varepsilon} \right) \to 1,
		\end{equation}
		as $N\to\infty$, where the innermost supremum is over all subspaces of dimension $N-2n$  and the outermost supremum is taken over all symmetric $\bQ \in [-1,1]^{n\times n}$ with $1$'s on the diagonal.
	\end{lemma}
	\begin{proof}
		Define an orthonormal basis $w_{1},...,w_{N}$ of $\R^N$ such that
		$$
		U = 
		\langle w_{N-2n},\ldots,w_{N}\rangle.
		$$
		Let $\bA$ be the top left $(N-2n) \times (N-2n)$-minor of $\frac{\bJ+\bJ^{\Trans}}{2}$ when written the basis $w_{1},...,w_{N}$. 
		For $\sigma \in U^\perp$ we have $H_N({\sigma}) = \sum_{i,j=1}^{N-2n} {\tilde{\sigma}_i} {\tilde{\sigma}_j} A_{ij}= N \sum_{i=1}^{N-2n} a_i ({\sigma_i})^2$ where $\tilde{\sigma}$ is $\sigma$ in the basis $w_1,\ldots,w_N$, and $N a_1<...<N a_{N-2n}$ are the eigenvalues of $\bA$. 
		Thus
		\[
		E^{U^\perp} \left[ \1_{\bQ_{\varepsilon}}\exp\left(\sum_{k=1}^n \beta_k H_N(\sigma^k)\right) \right]
		=
		E^{N-2n} \left[ \1_{\bQ_{\varepsilon}}\exp\left(N \sum_{k=1}^n \beta_k \sum_{i=1}^{N-2n} a_i ({\sigma_i^k})^2\right)  \right],
		\numberthis\label{eq:Ldrop1}
		\]
		where $E^{N-2n}$ is the expectation over $\bs$ uniform on $\mathcal{S}_{N-2n-1}^{n}$.
		Let $\bB$ be the top left $(N-2n) \times (N-2n)$-minor of $\frac{\bJ+\bJ^{\Trans}}{2}$ when written in standard basis and let $N b_1<...<Nb_{N-2n}$ be its eigenvalues. Recalling that $N \lambda_1<\ldots<N \lambda_N$ are the eigenvalues of $\frac{\bJ+\bJ^{\Trans}}{2}$ so by Cauchy's eigenvalue interlacing inequality (see \cite[Theorem 10.1.1]{symmetric-eigenvalue-problem}) we have $\lambda_i < a_i,b_i < \lambda_{i+2n}$. Thus \eqref{eq: gap} implies that
		\begin{equation}\label{eq: a and b}
		\sup_{U} \sup_{\bs \in \mathcal{S}_{N-2n-1}^n} \Bigg | \sum_{k=1}^n \beta_k \sum_{i=1}^{N-2n} a_i (\sigma_i^k)^2 - \sum_{k=1}^n \beta_k \sum_{i=1}^{N-2n} \frac{\sqrt{N-2n}}{\sqrt{N}} b_i ({\sigma_i^k})^2\Bigg | \overset{\mathbb{P}}{\longrightarrow}  0  \mbox{ as } N \rightarrow\infty,
		\end{equation}
		and so
		\begin{equation}\label{eq: Ldrop2}
		E^{U^\perp} \left[ \1_{\bQ_{\varepsilon}}\exp\left(\sum_{k=1}^n \beta_k H_N(\sigma^k)\right) \right] = 	E^{N-2n} \left[\1_{\bQ_{\varepsilon}} \exp\left( \sum_{k=1}^n \beta_k (\sigma^k)^\Trans \frac{\sqrt{N-2n}}{\sqrt{N}} \bB \sigma^k \right) \right]e^{o(N)},
		\end{equation}
		uniformly in $U$. By applying Proposition \ref{prop:freenergy-whileplefka} with $N-2n$ in place of $N$ the r.h.s. equals
		\begin{equation}
			e^{\frac{N-2n}{2}(\beta^\Trans \bQ^{\odot 2} \beta + \log|\bQ|) (1+\mathcal{O}(\sqrt{\varepsilon}))},
		\end{equation}
		for all $\bQ \ge \delta \bI$ and $\beta \in {\rm HT}(\bQ)$ (note that ${\rm HT}(\bQ)$ is a bounded set), proving the claim
	\end{proof}

        After ``recentering around $\bm{m}$ '' in the proof of Proposition \ref{prop: TAP LB} below the ``effective constraint matrix'' $\hat{\bQ}(\bbm) \in \mathbb{R}^{n \times n}$ given by
	\begin{equation}\label{def: Q hat}
		\hat{Q}_{k,\ell} = \frac{Q_{k,\ell} - m^k \cdot m^\ell}{\sqrt{1-|m^k|^2}\sqrt{1-|m^\ell|^2}},
	\end{equation}
        will appear (see \eqref{eq: eff constr 1}-\eqref{eq: replace Q with Q hat}).
    The next lemma will be used in the proof of Proposition \ref{prop: TAP LB} to exclude $\bmm$ such that $\hat{\bQ}(\bbm)$ has a small eigenvalue or $\bmm \bmm^{\Trans}$ has an entry close to $1$.
	\begin{lemma}\label{lem: small eval}
        For any $\beta,(h_1,\ldots,h_n)$ and positive symmetric $\bQ\in[-1,1]^{n\times n}$ with $1's$ on the diagonal, there is a $\delta\in(0,1)$  such that
	    $$ \P\left( \sup_{ \stackrel{\bbm \in {\rm{Plef}}_N(\bQ,\beta): \hat{\bQ}(\bbm)>0,\hat{\bQ}(\bbm)\text{ has eval. }<\delta}
	    {\text{ or } \|\bmm \bmm^{\Trans}\|_\infty >1-\delta}
	    }
	    \TAP(\bbm) \le \sup_{\bbm \in {\rm{Plef}}_N(\bQ,\beta): \hat{\bQ}(\bbm)\ge\delta\bI,\|\bmm \bmm^{\Trans}\|_\infty \le1-\delta} \TAP(\bbm)\right)\to1.$$
	 \end{lemma}
	 \begin{proof}
	    Choose $\eta \in (0,1)$ satisfying
		$ \eta \le (\sqrt{2} \|\bb^{\frac{1}{2}}\bQ\bb^{\frac{1}{2}}\|_2)^{-1}$ and let $\tilde{\bm{Q}} = (1-\eta) \bQ$, so that $\tilde{\bm{Q}} \in {\rm{Plef}}_n(\bQ,\beta)$ by \eqref{eq:Plefkadefn}. Let $\tilde{m}^k_i = (\tilde{\bm{Q}}^{1/2})_{k,i}$ for $k,i=1,\ldots,n$ and $\tilde{m}^k_i=0$ otherwise and $\tilde{\bmm}=(\tilde{m}^1,\ldots,\tilde{m}^n)$, so that $\tilde{\bmm} {\tilde{\bmm}}^\Trans = \tilde{\bm{Q}}$ and so $\tilde{\bmm} \in {\rm{Plef}}_N(\bQ,\beta)$. Then $\|\tilde{\bmm} \tilde{\bmm}^{\Trans}\|_\infty=1-\eta$. Let $\tilde{\delta}$ be the minimum of $\eta > 0$ and the smallest eigenvalue of $\hat{\bQ}(\tilde{\bmm})$. We then have
		\begin{equation}\label{eq: eta Q plefka bound}
		    \begin{array}{l}
			\displaystyle{ \frac{1}{N}\sup_{\bbm \in {\rm{Plef}}_N(\bQ,\beta): \hat{\bQ}(\bbm) \ge \tilde{\delta}\bI, \|\bmm \bmm^{\Trans}\|_\infty \le 1-\tilde{\delta}}} \TAP(\bbm)
			\\
			\overset{\eqref{eq:TAP},\eqref{eq: hamilt UB}}{\ge}
			-cn \max_{i}(|\beta_i|+ h_i)
			+ \frac{\eta}{2} \beta^\Trans \bQ^{\odot 2} \beta + \frac{n}{2}\log\eta+\frac{1}{2}\log\left|\bQ\right|
			\end{array}
		\end{equation}
		with probability going to $1$.
		
	    On the other hand assume that $\delta\in(0,1)$ and $\bmm \in {\rm{Plef}}_N(\bQ,\beta)$ and $\|\bmm \bmm^{\Trans}\|_\infty \ge 1-\delta$. Then $(\bmm \bmm^{\Trans})_{k,k}\ge1-\delta$ for some $k$, and because $|\bA|\le \prod_k A_{k,k}$ for any positive semi-definite $\bA$ we have $|\bQ-\bM|\le \delta$ and
		\begin{equation}\label{eq: delta dependent trivial tap bound 1}
			\frac{1}{N}\TAP(\bbm) \overset{\eqref{eq:TAP},\eqref{eq: hamilt UB}}{\le}
		    cn \max_{i}(|\beta_i|+ h_i) + n^2 \max_{i} \beta_i^2 + \frac{1}{2} \log \delta.
		\end{equation}
		
		Assume now instead that $\bmm \in {\rm{Plef}}_N(\bQ,\beta)$ and the smallest eigenvalue of $\hat{\bQ}(\bbm)$ is less than $\delta$. Let $\bm{S}=\text{Diag}((1-|m^1|^2)^{-1/2},...,(1-|m^n|^2)^{-1/2})$. Since $\hat \bQ(\bbm)  = \bm{S} (\bQ - \bM) \bm{S} $ and $\bm{S}_{ii} > 1 $, the 
	  smallest eigenvalue of $\bQ-\bM$ is bounded above by the smallest eigenvalue of $\hat{\bQ}(\bbm)$.  Thus if $\bmm$ is such that the smallest eigenvalue of $\hat{\bQ}(\bbm)$ is less than $\delta$, then the smallest eigenvalue of $\bQ-\bM$ is also less than $\delta$, so with probability tending to $1$ all such $\bbm$ satisfy
		\begin{equation}\label{eq: delta dependent trivial tap bound 2}
			\frac{1}{N}\TAP(\bbm) \overset{\eqref{eq:TAP},\eqref{eq: hamilt UB},\eqref{eq: mat det general}}{\le}
		    cn \max_{i}(|\beta_i|+ h_i) + n^2 \max_{i} \beta_i^2 + \frac{1}{2}(n-1)\log n
			+ \frac{1}{2} \log \delta.
		\end{equation}
	     If $\delta$ is picked small enough depending on $n,\beta,\bQ,(h_1,\ldots,h_n)$  then the r.h.s of both \eqref{eq: delta dependent trivial tap bound 1} and \eqref{eq: delta dependent trivial tap bound 2} are less than the bottom line of \eqref{eq: eta Q plefka bound}, and if we also ensure that $\delta\le \tilde{\delta}$ this proves the claim.
	\end{proof}
	We are now ready to prove the TAP lower bound Proposition \ref{prop: TAP LB}. The idea of the proof revolves around recentering the $\bs$ around some vector $\bbm$ and then restricting the partition function integral to a set where the contribution of the external field is negligible, which enables us to use the results of the previous subsection about the free energy without an external field.
	\begin{proof}[Proof of Proposition \ref{prop: TAP LB}]
	    By Lemma \ref{lem: small eval} there is a $\delta\in(0,1)$ such that \eqref{eq: LB claim} follows once we have shown that
	    \begin{equation}\label{eq: LB suff to show}
		\lim_{N\rightarrow\infty}
		\mathbb{P} \left( 
		F_N^\varepsilon(\beta, \bh, \bQ) \ge 
		\frac{1}{N} \sup_{\bbm \in {\rm{Plef}}_N(\bQ,\beta): \hat{\bQ}(\bbm)\ge\delta\bI,  \|\bmm \bmm^{\Trans}\|_\infty\le 1-\delta} \TAP(\bbm) - c \sqrt{\varepsilon} \right) = 1,
	    \end{equation}
	    for all $\varepsilon \in (0,c^{-1})$.

	    Fix some ${m^{1}},...,{m^{n}} \in\R^N$ with $ \|\bmm \bmm^{\Trans}\|_\infty \le 1-\delta$, $\hat{\bQ}(\bmm)\ge \delta \bI$ and $\bmm \bmm^\Trans \in {\rm{Plef}}_N(\bQ,\beta)$.
		By definition \eqref{eq: hamiltonian def} it follows that for all $\sigma, m \in \mathbb{R}^N$
		\begin{equation}\label{eq:centering}
			H_N(\sigma)
			= 
			H_N(m) + \nabla H_N(m) \cdot (\sigma-m) +  H_N(\sigma-m),
		\end{equation}
		so that for all $k$
			\begin{equation}\label{eq:recentering with ext field}
			\beta_k H_N({\sigma}^k) + N h^{k} \cdot {\sigma}^k 
			=
			\beta_k H_N({m^{k}}) + N h^{k} \cdot {m^{k}} + 
			N h^{m,k} \cdot (\sigma^k-m^k) + \beta_k H_N(\sigma^k-m^k),
		\end{equation}
		where 
		\begin{equation}\label{eq: eff ext field}
		      h^{m,k} = \frac{\beta_k}{N} \nabla H_N({m^{k}}) + h^{k}, \ k=1,\ldots,n,
		\end{equation}
		is the \textit{effective external field}.
		Using this we obtain
		\begin{equation}\label{eq:LB-recentering}
			Z^\varepsilon_N(\beta,\bh,\bQ) 
			=
			e^{\sum_{k=1}^n \left( \beta_k H_N({m^{k}}) + N h^{k} \cdot {m^{k}}\right)} \
			E\left[
			\1_{\bQ_\varepsilon}
			e^{
				\sum_{k=1}^n
				\left( \beta_k H_N(\sigma^k-m^k) +
				N h^{m,k} \cdot (\sigma^k-m^k) \right)
			}
			\right].
		\end{equation}
		Let $U$ be a $2n$-dimensional space whose span includes $m^k, h^{m,k}$ for $k=1,\ldots,n$.	We will now bound the expectation on the r.h.s. from below by inserting another indicator $\1_{A}$ given by
		\begin{equation}\label{eq: A def}
		A :=
		\left\{
		\bs: |P^{U} (\sigma^k-m^k)|\le \frac{\varepsilon}{4} \text{ for }k=1,\ldots,n\right\}.
		\end{equation}
		Note that on the event \eqref{eq: hamilt UB} for {$\bs\in A$}
		\[
		|(\sigma^k-m^k) \cdot m^\ell| \le \frac{\varepsilon}{4}
		\quad \text{and} \quad
		|(\sigma^k-m^k) \cdot h^{m,\ell}| \overset{\eqref{eq: hamilt UB}}{\le} c\varepsilon,
		\numberthis\label{eq:LB-slicebounds}
		\]
		for all $k,\ell \in \{1,...,n\}$. Therefore we obtain
		\begin{align*}
			E\left[
			\1_{\bQ_\varepsilon \cap A}
			e^{
				\sum_{k=1}^n
				\left( 
				\beta_k H_N(\sigma^k-m^k) +
				N h^{m,k} \cdot (\sigma^k-m^k)
				\right)
			}
			\right]
			\ge
			e^{-c \varepsilon N}
			E\left[
			\1_{\bQ_\varepsilon \cap A}
			e^{
				\sum_{k=1}^n
				\beta_k H_N(\sigma^k-m^k)
			}
			\right].
		    \numberthis\label{eq:LB-killing-the-external-field}
		\end{align*}
		Define the normalised projection of $\sigma^k-m^k$ onto $U^\perp$ by 
		$$
		\hat{\sigma}^k = \frac{P^{U^\perp}\left(\sigma^k-m^k\right)
		}{
			\big |
			P^{U^\perp}\left(\sigma^k-m^k\right)
			\big |
			}, k=1,\ldots,n.
		$$
		For $\bs\in A$ and $k$ it holds that
		\[
		|\sigma-m^{k}|^2 = |{\sigma}|^2 - |{m^{k}}|^2 - 2 (\sigma^k-m^k) \cdot {m^{k}} = 1 - |{m^{k}}|^2 + \mathcal{O}(\varepsilon)
		\numberthis\label{eq:LB-recentermagnitude}
		\]
		and so
        \[
        \left||P^{U^{\bot}}\left(\sigma-m^k\right)|^{2}-(1-\left|m^{k}\right|^{2})\right|\le c\varepsilon.
        \]
        Thus using that $H_{N}$ is $2$-homogeneous we obtain that on the
event \eqref{eq: hamilt UB}
\[
H_{N}\left(\sigma-m^{k}\right)\ge(1-\left|m^{k}\right|^{2})H_{N}\left(\hat{\sigma}\right)-c\varepsilon,
\]
        for all $k$, and
        
        \begin{equation}\label{eq: replace with sigma hat}
        E\left[
			\1_{\bQ_\varepsilon \cap A}
			e^{
				\sum_{k=1}^n
				\beta_k H_N(\sigma^k-m^k)
			}
			\right]
			\ge
			e^{- c \varepsilon N}
			E\left[
			\1_{\bQ_\varepsilon \cap A}
			e^{
				\sum_{k=1}^n
				\beta_k (1-|m^k|^2) H_N(\hat{\sigma}^k)
			}
			\right].
		\end{equation}
		To replace $\1_{\bQ_\varepsilon}$ with an indicator that is a function only of the $\hat{\bs} = (\hat{\sigma}^1,...,\hat{\sigma}^n)$,
		note that if
		\begin{equation}\label{eq: eff constr 1}
    		\big |
    		\hat{\sigma}^k\cdot\hat{\sigma}^\ell - \hat{Q}_{k,\ell}
    		\big |
    		\overset{\eqref{def: Q hat}}{=}
    		\bigg | 
    		\hat{\sigma}^k\cdot\hat{\sigma}^\ell
    		-
    		\frac{Q_{k,\ell} - m^k \cdot m^\ell}{\sqrt{1-|m^k|^2}\sqrt{1-|m^\ell|^2}}
    		\bigg |
    		\le 
    		\frac{\varepsilon}{2},
    		\ \ \forall k,\ell = 1,...,n
    	\end{equation}
		then
		\begin{equation}
		\begin{array}{rcl}
			|\sigma^k \cdot \sigma^\ell - Q_{k,\ell}|
			& \le &
			|(\sigma^k-m^k) \cdot (\sigma^\ell-m^\ell) - (Q_{k,\ell} - m^k \cdot m^\ell)|
			+
				|m^k \cdot (\sigma^\ell-m^\ell) + m^\ell \cdot (\sigma^k-m^k)|
			\\
			&\overset{\eqref{eq:LB-slicebounds}}{\le}  &
			\sqrt{1-|m^k|^2}\sqrt{1-|m^\ell|^2} \frac{\varepsilon}{2} + \frac{\varepsilon}{2}
			\le
			\varepsilon.
		\end{array}
		\end{equation}
		Thus we obtain
		$$
		\{\bs: \hat{\bs} \in \hat{\bQ}(\bbm)_{\frac{\varepsilon}{2}}\}
		\subset
		\{\bs: \bs \in \bQ_{\varepsilon}\}
		$$
		and
		\begin{equation}\label{eq: replace Q with Q hat}
			E\left[
			\1_{\bQ_\varepsilon \cap A}
			e^{
				\sum_{k=1}^n
				\beta_k (1-|m^k|^2) H_N(\hat{\sigma}^k)
			}
			\right]\ge
		E\left[ 
		\1_{A}
		\1_{\{\hat{\bs}  \in \hat{\bQ}(\bbm)_{\frac{\varepsilon}{2}}\}}
		e^{
			\sum_{k=1}^n
			\beta_k (1-|m^k|^2) H_N({\hat{\sigma}^{k}})
		}
		\right].
		\end{equation}
		
		Let $\mathcal{A}$ be the $\sigma$-algebra generated by $P^{U}\sigma^k$ for $k=1,\ldots,n$. Note that the $\hat{\sigma}^k$ are independent and uniform on $\mathcal{S}_{N-1}\cap U^{\bot}$ under $P[\cdot|\mathcal{A}]$. Thus
		\begin{equation}\label{eq:LB-conditioning}
		\begin{array}{l}
		E\left[ 
		\1_{A}
		\1_{\{\hat{\bs}  \in \hat{\bQ}(\bbm)_{\frac{\varepsilon}{2}}\}}
		e^{
			\sum_{k=1}^n
			\beta_k (1-|m^k|^2) H_N({\hat{\sigma}^{k}})
		}
		\right]\\
		=
		E\left[ 
		\1_{A}E\left[
		\1_{\{\hat{\bs}  \in \hat{\bQ}(\bbm)_{\frac{\varepsilon}{2}}\}}
		e^{
			\sum_{k=1}^n
			\beta_k (1-|m^k|^2) H_N({\hat{\sigma}^{k}})
		}
		\bigg| \mathcal{A}
		\right]
		\right]\\
		=
		E\left[ 
		\1_{A}E^{U^{\bot}}\left[
		\1_{\hat{\bQ}(\bbm)_{\frac{\varepsilon}{2}}}
		e^{
			\sum_{k=1}^n
			\beta_k (1-|m^k|^2) H_N(\sigma)
		}
		\right]
		\right]
		\end{array}
		\end{equation}
        Note that by letting
        \begin{equation}\label{def: beta_m}
            {(\beta_{\bbm})}_k = \beta_k (1-|m^k|^2), k=1,...,n
            \quad \text{and} \quad 
            \bb_m = \diag(\beta_{\bbm}) \in\mathbb{R}^{n \times n},
		\end{equation}
        we have
        $$ \|\bb^{\frac{1}{2}} \bQ \bb^{\frac{1}{2}} \|_2 = \|\bb_{\bbm}^{\frac{1}{2}} \hat{\bQ}(\bbm) \bb_{\bbm}^{\frac{1}{2}} \|_2, $$
        so that        
        \begin{equation}\label{eq: mmT plefka equivalence}
        \begin{array}{rcccl}
        \bbm\bbm^{\Trans}\in{\rm {Plef}}_{n}(\bQ,\beta) & \overset{\eqref{eq:Plefkadefn}}{\Leftrightarrow} & \|\bb^{\frac{1}{2}}\bQ\bb^{\frac{1}{2}}\|_{2}\leq\frac{1}{\sqrt{2}} & \Leftrightarrow & \|\bb_{\bbm}^{\frac{1}{2}}\hat{\bQ}(\bbm)\bb_{\bbm}^{\frac{1}{2}}\|_{2}\leq\frac{1}{\sqrt{2}}\\
         &  &  & \overset{\eqref{eq:HT-condition}}{\Leftrightarrow} & \beta_{\bbm}\in{\rm {HT}}(\hat{\bQ}(\bbm)).
        \end{array}
        \end{equation}
        Therefore Lemma \ref{lem:dimensiondrop} implies that on the event in \eqref{eq: dimensiondrop} the quantity \eqref{eq:LB-conditioning} is equal to
		\begin{equation}\label{eq: HT fe on A}
		\exp 
		\left(\frac{{\beta_{\bbm}}^\Trans \hat{\bQ}(\bbm)^{\odot 2} {\beta_{\bbm}}}{2} + \frac{1}{2} \log|\hat{\bQ}(\bbm)| + O(\sqrt{\varepsilon}) \right).
            \end{equation}
		Thus on that event \eqref{eq:LB-conditioning} is bounded below by
		
		\begin{align*}
			E\left[ 
			\1_{A}
			\right] \
			\exp\left(
			N \left(\frac{{\beta_{\bbm}}^\Trans \hat{\bQ}(\bbm)^{\odot 2} {\beta_{\bbm}}}{2} + \frac{1}{2} \log|\hat{\bQ}(\bbm)|\right) -  N c \sqrt{\varepsilon}
			\right).
		\numberthis\label{eq:converting-condional-expectation}
		\end{align*}
		We also have
		\begin{align*}
			E \left[ 
			\1_{A}
			\right] 
			\ge 
			 \prod_{k=1}^n \left( c \varepsilon^{2n} (1- |{m^{k}}|^2 - c\varepsilon)^{\frac{N-2n-2}{2}}\right)\ge \exp\left( \frac{N}{2} \sum_{k=1}^n \log(1-|m^k|^2) - c\varepsilon N\right),
			\numberthis\label{eq:A-volume}
		\end{align*}
		for a constant $c$ depending on $\delta$ and $N\ge c$, since $|m^k|^2 = ({\bmm \bmm^{\Trans}})_{k,k}\le 1-\delta$ (see \cite[(2.9)]{belius-kistler}). Combining \eqref{eq:LB-killing-the-external-field}, \eqref{eq: replace with sigma hat}, \eqref{eq: replace Q with Q hat}, \eqref{eq:LB-conditioning}, \eqref{eq:converting-condional-expectation}, \eqref{eq:A-volume} we obtain that
  \[
\begin{array}{l}
E\left[\1_{\bQ_{\varepsilon}\cap A}e^{\sum_{k=1}^{n}\left(\beta_{k}H_{N}(\sigma^{k}-m^{k})+Nh^{m,k}\cdot(\sigma^{k}-m^{k})\right)}\right]\\
\ge\exp\left(\sum_{k=1}^{n}\frac{N}{2}\log\left(1-|m^{k}|^{2}\right)+N\left(\frac{{\beta_{\bbm}}^{\Trans}\hat{\bQ}(\bbm)^{\odot2}{\beta_{\bbm}}}{2}+\frac{1}{2}\log|\hat{\bQ}(\bbm)|\right)-Nc\sqrt{\varepsilon}\right)
\end{array}
\]
Recall \eqref{def: Q hat} and \eqref{def: beta_m}, which imply that
\begin{equation}\label{eq: determinant transformation Q hat to Q-M}
\sum_{k=1}^n \log\left(1-|m^k|^2\right) + \log|\hat{\bQ}(\bbm)|
    = 
    \log\left(
    |\hat{\bQ}(\bbm)| \
    \prod_{k=1}^n \left(1-|m^k|^2\right)
    \right) = \log |\bQ - \bM|,
\end{equation}
and
\begin{equation}\label{eq: Onsager terms equivalence}
    {\beta_{\bbm}}^{\Trans}\hat{\bQ}(\bbm)^{\odot2}{\beta_{\bbm}}
    =
    \beta^{\Trans} \left(\bQ - \bM\right)^{\odot2} \beta.
\end{equation}
This implies that
\begin{equation}\label{eq: recent Hamilt LB}
\begin{array}{l}
E\left[\1_{\bQ_{\varepsilon}\cap A}e^{\sum_{k=1}^{n}\left(\beta_{k}H_{N}(\sigma^{k}-m^{k})+Nh^{m,k}\cdot(\sigma^{k}-m^{k})\right)}\right]\\
\ge\exp\left(\frac{N}{2}\beta^{\Trans}(\bQ-\bM)^{\odot2}\beta+\frac{N}{2}\log|\bQ-\bM|-Nc\sqrt{\varepsilon}\right).
\end{array}
\end{equation}
  
  Combining this with \eqref{eq:LB-recentering} we obtain that
		\begin{align*}
			Z^\varepsilon_N(\beta,\bh,\bQ) 
			\ge
			e^{\sum_{k=1}^{n}\left(\beta_{k}H_{N}({m^{k}})+Nh^{k}\cdot{m^{k}}\right)+\frac{N}{2}\beta^{\Trans}(\bQ-\bM)^{\odot2}\beta+\frac{N}{2}\log|\bQ-\bM|-Nc\sqrt{\varepsilon}
			},
		\numberthis\label{eq:LB-pre-cleaning-up}
		\end{align*}
		for all $\bmm$ with $\|\bmm \bmm^{\Trans}\| \le 1-\delta$, $\hat{\bQ}(\bmm)\ge \delta \bI$ and $\bmm \bmm^\Trans \in {\rm{Plef}}_N(\bQ,\beta)$, with probability tending to one.
		Recalling \eqref{eq:TAP}, we see that \eqref{eq:LB-pre-cleaning-up} is equivalent to
		\begin{align*}
			Z^\varepsilon_N(\beta,\bh,\bQ) 
			\ge
			e^{
				\TAP(\bbm) - N c \sqrt{\varepsilon}
			}.
		\end{align*}
        This proves \eqref{eq: LB suff to show}, so completes the proof of Proposition \ref{prop: TAP LB}.

	\end{proof}

	\bigskip

	\section{Upper bound}\label{section: upper bound}
	In this section we prove the following upper bound on the free energy.

	\begin{proposition}[The TAP Upper Bound]\label{prop:TAP_UB}
		Let $\bQ,\beta,h$ be as in Proposition \ref{prop: TAP LB}. For any ${{{\eta}}}>0$ there is a $c=c(\beta,\bh,\bQ,{{{\eta}}})$ such that for all $\varepsilon \in (0,c)$
        \begin{equation}
		\label{eq:TAP_UB}
    		\lim_{N\to\infty} \P \left( F_N^\varepsilon(\beta, \bh, \bQ) \le \frac{1}{N} \sup_{\bbm \in {\rm{Plef}}_N(\bQ,\beta)} \TAP(\bbm)  + {{{\eta}}} \right)=1.
        \end{equation}
	\end{proposition}
 	
    The proof involves constructing, in Subsection \ref{section:with_external_field}, a low-dimensional subspace of magnetizations $\mathcal{M}_N^n$, with the property that after recentering around $\bbm\in\mathcal{M}_N^n$ the effective external field is again almost completely contained in $\mathcal{M}_N^n$. The set $A(\bm{m})$ described in Subsection \ref{subsec: proof outline} is here essentially the hyperplane $\bm{m} + (\mathcal{M}^n_N)^{\bot}$ intersected with the cartesian product $\mathcal{S}_{N-1}^n$, where $(\mathcal{M}^n_N)^{\perp}$ is the perpendicular space. We write the integral in $F_N^\varepsilon(\beta,\bh,\bQ)$ using Fubini's theorem as a double integral first over $\mathcal{M}_N^n$ and then over the perpendicular space $(\mathcal{M}_N^n)^\perp$, so that the inner integral is an integral of the recentered Hamiltonian over the sets $A(\bm{m})$. The latter lacks external field and has a higher effective temperature than the original model (as long as $\bm{m}\ne0$). However, as opposed to in the proof of the lower bound, for some $\bm{m}$ Plefka's condition may not be satisfied, which means that this recentered Hamiltonian is not at high temperature.
    
    Therefore we replace the effective Hamiltonian by an approximation whose partition function is essentially a {\emph{low rank}} Harish-Chandra-Itzykson-Zuber (HCIZ) integral, and is in some sense always at high temperature. In Subsection \ref{subsec:HCIZ} we estimate such integrals. Using those estimates in Subsection \ref{section:with_external_field} we integrate out the inner integral so that the remaining outer integral is now the integral of a modified TAP free energy, in which the Onsager term $\frac{N}{2} \beta^\Trans (\bQ - \bM)^{\odot 2} \beta$ is replaced by the asymptotics of the HCIZ integral. The integral in $F_N^{\varepsilon}(\beta,\bh,\bQ)$ thus reduces to an integral of the exponential of $N$ times the modified TAP free energy over the low-dimensional space $\mathcal{M}_N^n$, and by the Laplace method the log of the integral turns into the maximizer of the modified TAP free energy over all $\bbm$.

    In Subsection \ref{subsec:location_max} we then show that if the Hessian of the modified TAP free energy at a critical point is negative semi-definite, as it must be at the maximizer, then $\bbm$ satisfies Plefka's condition. Furthermore we show that the Onsager terms of the modified TAP free energy and the original TAP free energy $\TAP(\bbm)$ are close, so that the upper bound on the free energy $F_N^\varepsilon(\beta,\bh,\bQ)$ 
 in terms of the modified TAP free energy implies the upper bound \eqref{eq:TAP_UB} in terms of the original TAP free energy.

    To implement the above strategy we will have to rely more heavily on random matrix calculations than in Section \ref{section: lower bound}. Define a deterministic version of the Hamiltonian by
	\begin{equation}\label{eq:deterministicHamiltonian}
		\tilde H_N(\sigma) = N \sum_{i = 1}^N \theta_{i/N} \sigma_i^2,
	\end{equation}
	where $\theta_{i/N}$ are the classical locations from \eqref{eq: eval class location}.    
	If $\bU$ is the change of basis matrix that diagonalizes $\bJ+\bJ^{\Trans}$ then by \eqref{eq: eval class location}
	\begin{equation}\label{eq:unifboundprob}
		\lim_{N \to \infty} \sup_{\sigma \in \mathcal{S}_{N-1}} \bigg| H_N(\sigma) - \tilde H_N(\bU\sigma) \bigg| = 0\text{ in probability,}
	\end{equation}
	so it suffices to prove the upper bound of the free energy for the deterministic Hamiltonian 
	\begin{equation}\label{eq: det FE}
	\tilde F_N^\varepsilon(\beta,\bh,\bQ) = \frac{1}{N} \log \int_{\bQ_\varepsilon} e^{\sum_{k = 1}^n \beta_k \tilde H_N(\sigma^k) + \tilde h^k \cdot \sigma^k} \, d \bs,
	\end{equation}
	where $\tilde h^k$ is the external field $h^k$ in the diagonalizing basis of the disorder matrix $\bJ+\bJ^{\Trans}$. The upper bound will be in terms of a corresponding TAP free energy
    \begin{equation}\label{eq:tildeTAP}
        \tilde{F}_{\textrm{TAP}}(\bbm) =  \sum_{k = 1}^n \beta_k \tilde{H}_N  ( m^k ) + N \sum_{k = 1}^n \tilde{h}^k \cdot m^k 
        + \frac{N}{2} \log| \bQ - \bM|
        + \frac{N}{2} \beta^\Trans (\bQ - \bM)^{\odot 2} \beta,
	\end{equation}
    which is simply $F_{\rm{TAP}}$ from \eqref{eq:TAP} with the original Hamiltonian and external field replaced with the deterministic diagonal Hamiltonian and rotated external field.

		We further discretize the deterministic Hamiltonian $\tilde H_N(\sigma)$. Given $K \geq 2$, we consider $K$ equally spaced numbers in $[-\sqrt{2},\sqrt{2}]$
		\begin{equation}\label{eq:discrete1}
		    	-\sqrt{2} = x_1 < x_2 < \dots < x_K  = \sqrt{2} - \frac{2 \sqrt{2}}{K} \quad \text{and} \quad x_{k + 1} - x_k = \frac{2 \sqrt{2}}{K}
		\end{equation}
	and the corresponding partition $I_1, \dots, I_K$ of $\{1, \dots, N\}$ given by
	\begin{equation}\label{eq:discrete2}
	    	I_k = \{i : x_k \leq \theta_{i/N} < x_{k + 1} \} \quad \text{and} \quad I_K = \{ i: x_K \leq \theta_{i/N} \}.
	\end{equation}
	Consider the ``binned'' Hamiltonian
	\begin{equation}\label{eq: binned hamilt def}
    	\tilde H^K_N(\sigma) = N \sum_{k = 1}^K \sum_{i \in I_k} x_k \sigma_i^2,
	\end{equation}
	where the eigenvalues $\theta_{i/N}$ are replaced with the left end point of the ``bin'' it belongs to.	We will compute an upper bound for the free energy of the binned Hamiltonian
	\begin{equation}\label{eq: binned FE}
	\tilde F_{N,K}^\varepsilon(\beta,\bh,\bQ) = \frac{1}{N} \log \int_{\bQ_\varepsilon} e^{\sum_{k = 1}^n \beta_k \tilde H^K_N(\sigma^k) + \tilde h^k \cdot \sigma^k} \, d \bs,
	\end{equation}
	and by taking $K\to\infty$ obtain an upper bound for \eqref{eq: det FE}.

	We first prove an upper bound of the free energy $\tilde F_{N,K}$ in the absence of an external field, i.e. $\bh = 0$. 	For this we use a result from \cite{GuHu21} about the asymptotics of HCIZ \cite{Harish,Itzyksonzuber} integral of rank $n$ (or $n$ dimensional spherical integrals in the terminology of \cite{GuHu21,GuMa05,HussonKoSublinear}). 

	\subsection{Binnned Hamiltonian without external field}\label{subsec:HCIZ}
    In this subsection we compute the free energy of the binned Hamiltonian without external field.
	
 We begin by using \cite{GuHu21} to compute the free energy with identity constraint $\bQ = \bI$, which is essentially an HCIZ integral of rank $n$. We now recall the limiting formula of \cite{GuHu21} (which are simplified due to the absence of outlier eigenvalues here). Given any measure $\nu$ let $G_\nu$ denote its Stieltjes transform defined on $\mathbb{C} \setminus \supp(\nu)$,
	\begin{equation}\label{eq: stieltjes transform G}
	    G_\nu(z) = \int (z-x)^{-1} \, d \nu(x),
	\end{equation}	
	and if $\lambda^*$ is the rightmost point in the support of $\nu$, we define as in \cite[Proposition~1]{GuHu21} the function
	\[
	J_\nu(z) = \lambda^* z  + (v_\nu(z) - \lambda^*) G_\nu(v_\nu(z)) - \log z - \int \log |v_\nu(z) - x| \, d \nu(x) - 1\text{ for }z>0
	\]
	where
	\begin{equation}\label{eq:vfunctional}
		v_\nu(z) = \begin{cases}
			\lambda^* & \text{if } G_\nu(\lambda^*) \leq z\\
			G^{-1}_\nu(z), & \text{if } G_\nu(\lambda^*) > z.
		\end{cases}
	\end{equation}
    Let $E_{{\rm{Haar}}}$ denote the probability measure where $(\sigma^1,\ldots,\sigma^n)$ are uniformly sampled orthonormal vectors (i.e. the top $k$ rows of a Haar distributed orthogonal random matrix). Also, suppose that $\bX_N$ is a matrix with empirical spectral distribution $\nu$ and suppose that the extremal eigenvalues of $\bX_N$ converge to the corresponding smallest and largest points in the support of $\nu$. The result \cite[Proposition~1]{GuHu21} implies that
		\begin{equation}\label{eq: guhu}
		\lim_{N \to \infty} \frac{1}{N} \log E_{{\rm{Haar}}} \left[ \exp\bigg( N \sum_{k= 1}^n \beta_k ( \sigma^k   )^\Trans \bm{X}_N \sigma^k  \bigg) \right] 
		= \frac{1}{2} \sum_{k = 1}^n J_{\nu}(2\beta_k).
		\end{equation}

    The Hamiltonian in \eqref{eq: hamiltonian def} can be written as
    \begin{equation}\label{eq: disc hamilt in terms of X}
        \sum_{k= 1}^n \beta_k \tilde H^K_N(\sigma^k) = N \sum_{k= 1}^n \beta_k (\sigma^k)^\Trans \bm{X}_K \sigma^k,
    \end{equation}
    for 
    \begin{equation}\label{eq:binnedmatrix}
        \bm{X}_K = \diag\bigg(\underbrace{x_1, \dots, x_1}_{|I_1|},\underbrace{x_2, \dots, x_2}_{|I_2|}, \dots, \underbrace{x_K, \dots, x_K}_{|I_K|}\bigg),
    \end{equation}
    and that the limiting spectral distribution of $\bm{X}_K$ of is equal to 
	\begin{equation}\label{eq:ESDmuK}
		\mu_K = \sum_{k = 1}^K \rho_k \delta_{x_{k}} \quad \text{where} \quad \rho_k = \lim_{N \rightarrow\infty} \frac{|I_k|}{N} = \int_{x_k}^{x_{k + 1}} \, d \mu_{\rm{sc}}(x),
	\end{equation}       
    (recall \eqref{eq:discrete1}-\eqref{eq:discrete2}).
    Thus defining
	\begin{equation}\label{eq:FKfunctional}
		\cF_K(\beta) = \frac{1}{2} J_{\mu_{K}} (2 \beta) \text{ for } \beta>0,
	\end{equation}
     it follows from \eqref{eq: guhu} that
	\begin{equation}\label{eq: apply guhu}
		\lim_{N \to \infty} \frac{1}{N} \log E_{{\rm{Haar}}} \left[ \exp\bigg( N \sum_{k= 1}^n \beta_k ( \sigma^k   )^\Trans \bm{X}_K \sigma^k  \bigg) \right] 
        = \sum_{k = 1}^n \mathcal{F}_K(\beta_k).
	\end{equation}
 
    The upper bound for the free energy \eqref{eq: binned FE} of the binned Hamiltonian will be given in terms of a modified TAP free energy where the $\beta_k$ in the right-hand side of \eqref{eq: guhu} are replaced by the eigenvalues $\tilde{\beta}_1{(\bbm)},\dots,\tilde{\beta}_n{(\bbm)}$ of $\bb^{1/2}(\bQ - \bM)\bb^{1/2}$, namely 
	\begin{equation}\label{eq:TAPK}
		\tilde{F}^K_{\textrm{TAP}}(\bbm) = 
		\sum_{k=1}^n \left(
		\beta_k \tilde{H}_N (m^k) 
		+ N  \tilde{h}^k \cdot m^k
		\right)
        + \frac{N}{2} \log |\bQ - \bM|
		+ N \sum_{k =1}^n \mathcal{F}_K(\tilde{\beta}_k{(\bbm)}),		
	\end{equation}
    for $\bbm = (m^1,\dots,m^n) \in \R^{n \times N}$, cf. \eqref{eq:TAP}, \eqref{eq:tildeTAP}.

  
    For $\nu=\mu_K$ we have $\lambda^*=\sqrt{2}$ and $G_{\mu_{K}}(\lambda^*)=\infty$ so that
	    \begin{equation}\label{eq: case nu=mu_K}
          v_{\mu_{K}}\left(z\right)=G_{\mu_{K}}^{-1}\left(z\right),
	    \end{equation}
        which is smooth on $(0,\infty)$.
	    Therefore $J_{\mu_{K}}$ is a smooth function on $(0,\infty)$, and so is $\mathcal{F}_K$. Note that
    \begin{equation}\label{eq:thisFKandBKFK}
        \mathcal{F}_K\text{ coincides with the }\mathcal{F}_K\text{ from \cite[(4.15)]{belius-kistler}},
    \end{equation}
    which can be verified by comparing \cite[(4.13)-(4.15)]{belius-kistler} and \eqref{eq: stieltjes transform G}-\eqref{eq:FKfunctional}, where by \eqref{eq: case nu=mu_K} the $\lambda_K(\beta)$ of \cite{belius-kistler} is the same as $v_{\mu_K}(2\beta)$.
    The representation of \cite[Lemma 10]{belius-kistler}, or the case $n=1$ of \eqref{eq: sperhicalintidentity} below, implies that $\mathcal{F}_K(\beta)\to0$ as $\beta\to0$, so 
setting $\mathcal{F}_K(0)=0$ gives a continuous extension of the function to $[0,\infty)$ (in fact, $\mathcal{F}_K$ is smooth on $[0,\infty)$, but we refrain from proving or using this fact). Furthermore the representation of \cite[Lemma 10]{belius-kistler} (or the case $n=1$ of \eqref{eq: sperhicalintidentity}) implies that $\mathcal{F}_K$ is convex, so that
    \begin{equation}\label{eq:FK Lipschitz}
        \mathcal{F}_K(x)\text{ is Lipschitz on compact subsets of }[0,\infty).
    \end{equation}    

 We now use an approximation argument to derive a formula for the limiting free energy 
 $$ \tilde F_{N,K}^\varepsilon(\beta,0,\bI) = \frac{1}{N}\log \int_{\bI_\varepsilon} \exp\bigg( \sum_{k= 1}^n \beta_k \tilde H^K_N(\sigma^k) \bigg)$$
 of the binned Hamiltonian with an identity constraint from \eqref{eq: apply guhu}. Since the r.h.s. of \eqref{eq: sperhicalintidentity} is smooth we see that for all finite $K$ the partition function of the right-hand side is at high temperature for all $\beta$.
	\begin{lemma}\label{lem:sperhicalintidentity}For any $\beta \in [0,\infty)^n$ and $K\ge1$ 
		\begin{equation}\label{eq: sperhicalintidentity}
		\lim_{\varepsilon \to 0} \lim_{N \to \infty} \frac{1}{N} \log \int_{\bI_\varepsilon} \exp\bigg( \sum_{k= 1}^n \beta_k \tilde H^K_N(\sigma^k) \bigg) \, d \bs = \sum_{k = 1}^n \cF_K(\beta_k),
		\end{equation}
		where the region of integration $\bI_\varepsilon$ is the neighborhood of the identity matrix $\bI$ as defined in \eqref{eq: eps nbhood}.
	\end{lemma}
	\begin{proof}
		 We approximate the integral of the l.h.s. of \eqref{eq: sperhicalintidentity} by the expectation in \eqref{eq: apply guhu}.
		
		To this end let $E$ denote the measure under which $(\sigma^1,\ldots,\sigma^n)$ are independent uniform unit vectors on the sphere. Define $\tilde \bs = (\bs \bs^\Trans)^{-\frac{1}{2}} \bs$, which exists $E$-almost surely. By construction the matrix $\tilde \bs$ has orthogonal rows.
		Furthermore, if $\bO$ is an arbitrary $N\times N$ orthogonal matrix then $\bs \stackrel{d}{=} \bs \bO$ under $E$ by rotational symmetry of the uniform measures on the product of spheres. It follows that under $E$
		\[
		\tilde \bs \stackrel{d}{=} (\bs \bO (\bs \bO)^\Trans)^{-\frac{1}{2}} \bs \bO =  (\bs \bs^\Trans)^{-\frac{1}{2}} \bs \bO = \tilde \bs \bO,
		\]
		so that the $E$-law of $\tilde \bs$ is $E_{\rm{Haar}}$. Since $( \bs \bO ) (\bs \bO)^{\Trans}= \bs \bs^{\Trans} $ for any orthogonal $\bO$, so that $P\left(\bs \in A,\bs\in \bI_{\varepsilon}\right)=P\left(\bs \bO\in A,\bs \bO\in \bI_{\varepsilon}\right)=P\left(\bs \bO\in A,\bs \in \bI_{\varepsilon}\right)$ for any measurable set $A$ also the $E[\cdot|\bI_\varepsilon]$-law of $\tilde{\sigma}$ is $E_{\rm{Haar}}$.  Lemma~\ref{lem:volumeconstraint} implies that $\lim_{N \to \infty} \frac{1}{N} \log P(\bI_\varepsilon) = \frac{1}{2} \log|\bI| = 0$ for all $\epsilon > 0$.
 		Thus for all $\varepsilon>0$ it follows from \eqref{eq: apply guhu} that
		\begin{equation}\label{eq: apply guhu 2}
		\lim_{N \to \infty} \frac{1}{N} \log E\left[ \1_{\bI_\varepsilon} \exp\bigg( N \sum_{k= 1}^n \beta_k ( \tilde \sigma^k   )^\Trans \bm{X}_K \tilde \sigma^k  \bigg) \right] 
		=  \sum_{k = 1}^n \mathcal{F}_K(\beta_k).
		\end{equation}
        Since $\left(\bI+\bA\right)^{-1/2}=I+O\left(\|\bA\|_{2}\right)$ if
        $\|\bA\|_{2}\le\frac{1}{2}$ we have for $\bs\in\bI_{\varepsilon}$
        that 
        \[
        \|\tilde{\bs}-\bs\|_{F}=\|\left(\bs\bs^{\Trans}\right)^{-1/2}\bs-\bs\|_{F}\le\sqrt{n}\|\left(\bs\bs^{\Trans}\right)^{-1/2}-\bI\|_{2}\le c\|\bI-\bs\bs^{\Trans}\|_{2}\le c\|\bI-\bs\bs^{\Trans}\|_{\infty}\le c\varepsilon,
        \]
        and so for $\bs\in\bI_{\varepsilon}$
        \[
        \left|\left(\tilde{\sigma}^{k}\right)^{\Trans}\bX_{k}\tilde{\sigma}^{k}-\left(\sigma^{k}\right)^{\Trans}\bX_{k}\sigma^{k}\right|\le c\varepsilon.
        \]
        Therefore \eqref{eq: sperhicalintidentity} follows from \eqref{eq: apply guhu 2} and \eqref{eq: disc hamilt in terms of X}.

	\end{proof}

	
	We will extend this formula to general positive definite constraints $\bQ > 0$. For this we will need the next lemma which uses a change of variables to estimates integrals with a general constraint in terms of integrals with an identity constraint. Let $\mathcal{B}_{N}(r) \subset \mathbb{R}^N$ denote the closed ball of radius $r$.
	
	\begin{lemma}\label{lem: from I to Q}
        For any $\delta\in(0,1)$ there is a constant $c=c\left(\delta\right)$
        such that for any symmetric $\bQ\in\left[-1,1\right]^{n\times n}$
        with $1$'s on the diagonal and $\bQ>\delta\bI$, any $\varepsilon>(0,c^{-1})$, and any Lipschitz $f:(\mathcal{B}_{N}(1+ c^{-1}\varepsilon))^{n}  \to\mathbb{R}$ with Lipschitz constant $L$
        we have for $N\ge c(\varepsilon)$
        \begin{equation}\label{eq: from I to Q}
            \begin{array}{l}
           \frac{1}{N}\log \int\1_{\bI_{c^{-1}\varepsilon}}\exp\left( f\left(\bQ^{1/2}\bs\right) \right) d\bs + \frac{1}{2}\log\left|\bQ\right|-c\left(1+\frac{L}{N}\right)\varepsilon \\
           
           \le \frac{1}{N}\log\int\1_{\bQ_{\varepsilon}}\exp\left(f\left(\bs\right)\right)d\bs\\
            \le\frac{1}{N}\log \int\1_{\bI_{c\varepsilon}}\exp\left( f\left(\bQ^{1/2}\bs\right) \right) d\bs + \frac{1}{2}\log\left|\bQ\right|+c(1+\frac{L}{N})\varepsilon.
            \end{array}
        \end{equation}
	\end{lemma}
	\begin{proof}
	    
Fix $\bQ$ with $\bQ>\delta\bI$. Let $\mathbb{Q}$ be the measure
under which $u_{1},\ldots,u_{N}$ are independent Gaussian vectors
in $\mathbb{R}^{n}$ with covariance $N\bQ$, and let $u^{k}=\left(u_{1,k},\ldots,u_{N,k}\right)$
for $k=1,\ldots,n$, as in the proof of Lemma \ref{lem:volumeconstraint}. Writing also $\bu=\left(u^{1},\ldots,u^{n}\right)\in\mathbb{R}^{n\times N}$
we have using the same change of measure as in \eqref{eq: lower bound} in that lemma that
\begin{equation}
\begin{array}{l}
\frac{1}{N}\log\mathbb{Q}\left(\exp\left(f\left(\frac{u^{1}}{\left|u^{1}\right|},\ldots,\frac{u^{n}}{\left|u^{n}\right|}\right)\right)\1_{\bu\bu^{\Trans}\in N\bQ_{\varepsilon c_{1}}}\right)+\frac{1}{2}\log\left|\bQ\right|-c\left(\delta\right)\varepsilon\\
\le\frac{1}{N}\log\int\1_{\bQ_{\varepsilon}}\exp\left(f\left(\bs\right)\right)d\bs\\
\le\frac{1}{N}\log\mathbb{Q}\left(\exp\left(f\left(\frac{u^{1}}{\left|u^{1}\right|},\ldots,\frac{u^{n}}{\left|u^{n}\right|}\right)\right)\1_{\bu\bu^{\Trans}\in N\bQ_{\varepsilon c_{2}}}\right)+\frac{1}{2}\log\left|\bQ\right|+c\left(\delta\right)\varepsilon,
\end{array}\label{eq: UB AND LB}
\end{equation}
for all $N$. Furthermore letting
$\mathbb{E}$ be the law of i.i.d. independent Gaussian vectors we
have that the $\mathbb{E}$-law of $\bQ^{1/2}\bu$ is the $\mathbb{Q}$-law
of $\bu$, so that for $l=1,2$
\begin{equation}
\begin{array}{l}
\mathbb{Q}\left(\exp\left(f\left(\frac{u^{1}}{\left|u^{1}\right|},\ldots,\frac{u^{n}}{\left|u^{n}\right|}\right)\right)\1_{\bu\in\bQ_{\varepsilon c_{l}}}\right)\\
=\mathbb{E}\left(\exp\left(f\left(\frac{\left(\bQ^{1/2}\bu\right)^{1}}{\left|\left(\bQ^{1/2}\bu\right)^{1}\right|},\ldots,\frac{\left(\bQ^{1/2}\bu\right)^{n}}{\left|\left(\bQ^{1/2}\bu\right)^{n}\right|}\right)\right)\1_{\bQ^{1/2}\bu\left(\bQ^{1/2}\bu\right)^{\Trans}\in N\bQ_{\varepsilon c_{l}}}\right).
\end{array}\label{eq: to E}
\end{equation}
Writing $a\asymp b$ if there is constant $c$ depending only on $n$
such that $c^{-1}\le\frac{a}{b}\le c$, and writing $a\asymp_{\delta}b$
if the constant is allowed to depend also on $\delta$, we have
\begin{equation}
\|\bQ^{1/2}\bu\bQ^{1/2}\bu^{\Trans}-N\bQ\|_{\infty}\asymp\|\bQ^{1/2}\bu\bQ^{1/2}\bu^{\Trans}-N\bQ\|_{2}\asymp_{\delta}\|\bu\bu^{\Trans}-N\bI\|_{2}\asymp\|\bu\bu^{\Trans}-N\bI\|_{\infty}.\label{eq: BLA}
\end{equation}
Let $\tilde{\sigma}^{i}=\frac{u^{i}}{\left|u^{i}\right|}$ so that
under $\mathbb{E}$ the $\tilde{\sigma}^{1},\ldots$,$\tilde{\sigma}^{n}$
are i.i.d. uniform on the unit sphere. The inequalities (\ref{eq: BLA})
imply that on the event in the indicator of (\ref{eq: to E}) we have
$\left|\frac{\left(\bQ^{1/2}\bu\right)^{i}}{\left|\left(\bQ^{1/2}\bu\right)^{i}\right|}-\bQ^{1/2}\tilde{\sigma}^{i}\right|\le c\left(\delta\right)\varepsilon$,
and that the bottom line of (\ref{eq: to E}) is bounded below by
\begin{equation}
\mathbb{E}\left(\exp\left(f\left(\bQ^{1/2}\tilde{\sigma}^{1},\ldots,\bQ^{1/2}\tilde{\sigma}^{n}\right)\right)\1_{\bu\bu^{\Trans}\in N\bI_{c\left(\delta\right)^{-1}\varepsilon}}\right)e^{-Lc\left(\delta\right)\varepsilon}.\label{eq: f to sigma tilde}
\end{equation}
We have
\begin{equation}
\left\{ \tilde{\bs}\tilde{\bs}^{\Trans}\in\bI_{c\varepsilon},\max_{i=1}^{n}\left|\left|u^{i}\right|-N\right|\le \varepsilon N\right\} \subset\left\{ \bu\bu^{\Trans}\in N\bI_{\varepsilon}\right\} ,\label{eq: inclusion LB}
\end{equation}
for a small enough $c$ and all $\varepsilon\in(0,c)$. Also $\tilde{\bs}$ is independent of $\left|u^{i}\right|$,
and assuming $\varepsilon\le\delta$ (as we may) we have $\mathbb{P}\left(\max_{i=1}^{n}\left|\left|u^{i}\right|-1\right|\le\varepsilon\right)\to1$
as $N\to\infty$, and we obtain from \eqref{eq: inclusion LB} with $c(\delta)^{-1}\varepsilon$ in place of $\varepsilon$ that (\ref{eq: f to sigma tilde})
is at least
\[
\frac{1}{2}\mathbb{E}\left(\exp\left(f\left(\bQ^{1/2}\tilde{\sigma}^{1},\ldots,\bQ^{1/2}\tilde{\sigma}^{n}\right)\right)\tilde{\bs}\tilde{\bs}^{\Trans}\in\bI_{c\left(\delta\right)^{-1}\varepsilon}\right)e^{-Lc\left(\delta\right)\varepsilon},
\]
for $N\ge c\left(\varepsilon\right)$. This implies the lower
bound of \eqref{eq: from I to Q}. The upper bound of \eqref{eq: from I to Q} follows similarly, with the simplification
that (\ref{eq: inclusion LB}) is replaced by the simpler $\left\{ \bu\bu^{\Trans}\in N\bI_{\varepsilon}\right\} \subset\left\{ \tilde{\bs}\tilde{\bs}^{\Trans}\in\bI_{c\varepsilon}\right\}$ for a large enough $c$ and all $\varepsilon \in (0,c^{-1})$, so that the independence of $\tilde{\bs}$ of $\left|u^{i}\right|$
need not be invoked.
	\end{proof}

	To extend Lemma \ref{lem:sperhicalintidentity} to non-identity constraints the next lemma will also be needed. Let $\tilde{\beta}_1,...,\tilde{\beta}_n$ denote the eigenvalues of $\bb^{1/2} \bQ \bb^{1/2}$.
	\begin{lemma}\label{lem:FK Lipschitz}
	    For any $\delta, C >0$ and $K\in\mathbb{N}$ there exists a constant $L=c(\delta,C,K)$ such that $(\beta, \bQ) \to \cF_{K}( \tilde \beta_k) + \frac{1}{2} \log|\bQ|$ is $L$-Lipschitz continuous for $\bQ \ge \delta \bI$ with $\|\bQ\|_\infty\le1$ and $|\beta| \le C$.
	\end{lemma}
	\begin{proof}
	    The eigenvalues of $\bQ$  are Lipschitz continuous in the entries of $\bQ$ with Lipschitz constant depending only on $n$. Since $\bQ \ge \bI \delta$ implies that all eigenvalues lie in $(\delta,n]$ it follows that $\frac{1}{2} \log|\bQ|$ is Lipschitz in the entries of $\bQ$ for such $\bQ$.
    Furthermore the $\tilde{\beta}_1,\ldots,\tilde{\beta}_n$ are
    also the eigenvalues of $\bb \bQ$, so they are Lipschitz as functions of the entries of $\bQ$ and $\beta$
    with Lipschitz constant depending on $n$ and $C$, 
    and they are bounded in terms of $C$. Since $\cF_K$ is Lipschitz on compact intervals (recall \eqref{eq:FK Lipschitz}) the claim follows.
    

	\end{proof}
	\bigskip 
	
	We can now compute the limiting free energy
    \begin{equation}\label{eq: tilde FE with Q}
        \tilde F_{N,K}^\varepsilon(\beta,0,\bQ) = \frac{1}{N}\log \int_{\bQ_\varepsilon} \exp\bigg( \sum_{k= 1}^n \beta_k \tilde H^K_N(\sigma^k) \bigg) d \bs,
    \end{equation}
    of the binned model without external field and with general constraint $\bQ$. Similarily to in \eqref{eq: sperhicalintidentity} the smoothness of $\mathcal{F}_K$ means that for all finite $K$ the partition function in \eqref{eq: tilde FE with Q} is at high temperature for all $\beta$.
	\begin{lemma}[Limiting free energy of the binned model]\label{lem:sperhicalintgeneral}
		For every $\delta > 0$ and $C > 0$ 
		we have
		\begin{equation}\label{eq: binned free energy limit}
		\lim_{\varepsilon\to 0} \lim_{N\to\infty}
		\sup_{\bQ \ge \delta \bI} 
		\sup_{|\beta| \le C}
		\bigg| \frac{1}{N} \log \int_{\bQ_\varepsilon} \exp\bigg( \sum_{k= 1}^n \beta_k \tilde H^K_N(\sigma^k) \bigg) d \bs
		- \bigg( \sum_{k = 1}^n \cF_{K}( \tilde \beta_k) + \frac{1}{2} \log|\bQ| \bigg) \bigg|
		= 0,
		\end{equation}
		where the outermost $\sup$ is over symmetric $\bQ \in [-1,1]^n$ with $1$s on the diagonal and $\tilde \beta_1, \ldots,\tilde \beta_n$ 
        are the eigenvalues of $\bb^{1/2} \bQ \bb^{1/2}$.
	\end{lemma}
	
	\begin{proof}
	
We use
\[
f\left(\bs\right)=\sum_{k=1}^{n}\beta_{k}\tilde{H}_{N}^{K}(\sigma^{k}),
\]
in Lemma \ref{lem: from I to Q}.
From \eqref{eq: binned hamilt def} and using $|\beta|\le C$
this $f$ has Lipschitz constant at most $c\left(C\right)N$, and
we obtain
\begin{equation}\label{eq: sandwich}
\begin{array}{l}
\frac{1}{N}\log\int\1_{\bI_{c^{-1}\varepsilon}}\exp\left(\sum_{k=1}^{n}\beta_{k}\tilde{H}_{N}^{K}((\bQ^{1/2}\bs)^{k})\right)d\sigma+\frac{1}{2}\log\left|\bQ\right|-c\varepsilon\\
\le\frac{1}{N}\log\int\1_{\bQ_{\varepsilon}}\exp\left(\sum_{k=1}^{n}\beta_{k}\tilde{H}_{N}^{K}(\sigma^{k})\right)\\
\le\frac{1}{N}\log\int\1_{\bI_{c\varepsilon}}\exp\left(\sum_{k=1}^{n}\beta_{k}\tilde{H}_{N}^{K}((\bQ^{1/2}\bs)^{k})\right)d\sigma+\frac{1}{2}\log\left|\bQ\right|+c\varepsilon,
\end{array}
\end{equation}
for any $\bQ$ as in the statement of the lemma. Next writing $\bQ^{1/2}\bb \bQ^{1/2} = \bO^{\Trans} \tilde \bb \bO$ for an $n \times n $ orthogonal matrix and $\tilde \bb$ the diagonal matrix of eigenvalues of $\bb^{1/2} \bQ \bb^{1/2}$ (recall \eqref{eq: QbQ bQb same eigenvalues}) we have using \eqref{eq: disc hamilt in terms of X} that
\[
\begin{array}{rcl}
    \sum_{k=1}^{n}\beta_{k}\tilde{H}_{N}^{K}((\bQ^{1/2}\bs)^{k})
    &=&N\text{Tr}\left(\bb(\bQ^{1/2}\bs) \bX_{K} (\bQ^{1/2}\bs)^{\Trans}\right)
    \\
    &=&N\text{Tr}\left(\bQ^{1/2}\bb \bQ^{1/2}\bs\bX_{K}\bs^{\Trans}\right)
    \\
    &=&N\text{Tr}\left(\bO ^{\Trans} \tilde{\bb}\bO\bs\bX_{K}\bs^{\Trans}\right)
    \\
    &=&N\text{Tr}\left(\tilde{\bb}\bO\bs\bX_{K}(\bO\bs)^{\Trans}\right)
    \\
    &=&\sum_{k=1}^{n}\tilde{\beta_{k}}\tilde{H}_{N}^{K}((\bO\bs)^{k}).
\end{array}
\]
We have $\int g(\bO\bs) d\bs = \int g(\bs) d\bs$ for any measurable $g$ by symmetry so we can use Lemma \ref{lem:sperhicalintidentity}
to estimate the first and last line of \eqref{eq: sandwich} we obtain that for any fixed $\bQ$ and $\beta$
\begin{equation}
    \lim_{\varepsilon\to 0}\lim_{N\to\infty}
    \left|
    \frac{1}{N}\log\int\1_{\bQ_{\varepsilon}}\exp\left(\sum_{k=1}^{n}\beta_{k}\tilde{H}_{N}^{K}(\sigma^{k})\right) d\bs
    -
    \left(\sum_{k=1}^n\cF_{K}( \tilde \beta_k) + \frac{1}{2} \log|\bQ|\right)
    \right|
    = 0
    .\label{eq: fixed Q}
\end{equation}
As in Proposition \ref{prop:freenergy-whileplefka} we can then deduce the uniformity in $\bQ$ and $\beta$ in \eqref{eq: binned free energy limit} by making two lattices of finitely many $\bA^{1},...,\bA^{M} \in [-1,1]^{n\times n}$ and $b^{1},...,b^{L} \in (0,C]^n$ and then use Lipschitz continuity (see Lemma \ref{lem:FK Lipschitz}). More precisely we can choose two lattices such that for all $\bQ$ we have
$$
    |\bQ - \bA^i| \le \varepsilon \quad \text{ and } \quad \bA_{\frac{\varepsilon}{2}}^i \subset \bQ_{\varepsilon} \subset \bA_{2\varepsilon}^i
$$
for some $i\in\{1,...,M\}$ (cf. \eqref{eq:A-Q small}-\eqref{eq: inclusion}), as well as for all $\beta$ 
$$
    \max_k|\beta_k - b^{j_k}| \le \varepsilon
    \quad \text{ and } \quad \left| \sum_{k=1}^n\beta_{k}\tilde{H}_{N}^{K}(\sigma^{k}) - \sum_{k=1}^n b^{j_k}\tilde{H}_{N}^{K}(\sigma^{k}) \right| \le c(C) \varepsilon N
$$
for some $j_1,\ldots,j_k\in\{1,...,L\}$ (cf. \eqref{eq: F(beta)-F(beta^j)}). Then using Lemma \ref{lem:FK Lipschitz} completes the proof (cf. \eqref{eq: F(Q)-F(A)}).
	\end{proof}

	To recover the free energy $\tilde{F}_{N}^\varepsilon$ of \eqref{eq: det FE} from the binned version $\tilde{F}_{N,K}^\varepsilon$, we will need to send the number of bins $K \to \infty$. The next lemma shows that if $\beta$ is in the high temperature region ${\rm HT}\left(\bQ\right)$ (recall \eqref{eq:HT-condition}) the sum $\sum_{k=1}^n\mathcal{F}_{K}	(\tilde{\beta}_{k})$ from \eqref{eq: binned free energy limit} converges to the simple expression that appears in the annealed free energy (recall \eqref{eq:firstmomentlemma}, \eqref{eq:freenergy-withplefka}). 
	\begin{lemma}\label{lem:Plefka}
	It holds that
	\begin{equation}
	\lim_{K\to\infty}\sup_{\beta\in{\rm HT}\left(\bQ\right)}\left|
	\sum_{k=1}^n\mathcal{F}_{K}
	(\tilde{\beta}_{k})-\frac{1}{2}\beta^{\Trans}\bQ^{\odot2}\beta\right|=0,\label{eq:Klimitnoexternalfield}
	\end{equation}
    uniformly over symmetric positive definite matrices $\bQ$,
	where $\tilde{\beta}_1,\ldots,\tilde{\beta}_n$ are the eigenvalues of  $\bb^{1/2} \bQ\bb^{1/2}$. 
    \end{lemma}
	\begin{proof}
	Recall \eqref{eq:thisFKandBKFK}. By \cite[Lemma 14 + (4.28)]{belius-kistler} we get that 
	\begin{equation}\label{eq:K-bin convergence}
	    \lim_{K \to \infty} 
	    \sup_{\tilde{\beta} \in [0,\frac{1}{\sqrt{2}}]}
	    \left|\cF_K(\tilde{\beta}) - \frac{\tilde{\beta}^2}{2}\right|
	    =
	    0.
	\end{equation}
        If $\beta \in {\rm{HT}}(\bQ)$ then $\tilde{\beta}_k \le \frac{1}{\sqrt{2}}$ for all $k \in\{1,\ldots,n\}$ by the definition \eqref{eq:HT-condition}.
		Thus by \eqref{eq:K-bin convergence} we have that 
		$$
		\lim_{K \to \infty} 
	    \sup_{{\beta} \in {\rm{HT}}(\bQ)}
	    \left|\sum_{k=1}^{n}\mathcal{F}_K(\tilde{\beta}_k) - \sum_{k=1}^{n} \frac{\tilde{\beta}_k^2}{2}\right|
	    =
	    0.
		$$
 		We can now write $\tilde \beta$ back in terms of $\beta$ and $\bQ$ using
 		\[
\sum_{k=1}^{n}\tilde{\beta}_{k}^{2}=\Tr\left(\left(\bb^{1/2}\bQ\bb^{1/2}\right)^{2}\right)=\text{Tr}\left(\bQ\bb\bQ\bb\right)=\sum_{ij}\left(\bQ\bb\right)_{ij}\left(\bQ\bb\right)_{ji}
=\beta^{\Trans}\bQ^{\odot2}\beta.
\]
	\end{proof}

	\subsection{Upper bound in terms of modified TAP free energy}	\label{section:with_external_field}
	In this subsection, we prove an upper bound of the free energy in the presence of external fields in terms of a modified TAP free energy.
    
    The main idea is to divide each of the $n$ spheres into two parts: A subspace $\mathcal{M}_N$ of dimension much smaller than $N$, where most of the effect of the external fields is felt, and the complementary space $\mathcal{M}_N^\perp$ which is almost orthogonal to all the external fields (as in \cite[Section 4]{belius-kistler}). We write the partition function integral as a double integral over first the lower dimensional $\mathcal{M}_N$ and then the higher dimensional $\mathcal{M}_N$, where the inner integral is the partition function of the recentered the Hamiltonian. The inner integral is essentially a partition function without external field, so it can be estimated using the results of the previous subsection. In this way we obtain an estimate for the partition function where the remaining outer integral is now the integral of $N$ times the exponential of a modified TAP free energy whose Onsager term is an expression involving $\mathcal{F}_K$ rather than $\frac{N}{2} \beta^\Trans (\bQ - \bM)^{\odot 2} \beta$. Since the dimension of the outer integral is much smaller than $N$ we can then estimate it in terms of the maximum of the modified TAP free energy using the Laplace method.

	The following lemma constructs the spaces $\mathcal{M}_N$. Recall that the external fields are denoted by $\bh \in \mathbb{R}^{n \times N}$ and satisfy $|h^k| = h_k$ for each $k \in\{1,...,n\}$ for fixed values $h_1,...,h_n \ge 0$, and that the external fields in the diagonalizing basis of the Hamiltonian is denoted by $\tilde{\bh} = (\tilde{h}^1,...,\tilde{h}^n)$.    
	\begin{lemma}\label{lem:subspace-construction}
        Let $N\ge1$. For any $\beta_1,\ldots, \beta_k$ and $h^1,\ldots,h^n \in \mathbb{R}^N$, there exists a sequence of linear subspaces $\mathcal{M}_{1}, \mathcal{M}_{2}, ...$ such that $\mathcal{M}_{N} \subset \R^{N}$,
		\[
		\dim(\mathcal{M}_{N}) \le n N^{3/4}
		\]  
		and $\mathcal{M}_{N}^n = (\mathcal{M}_{N})^n$ is approximately invariant under the map 
		$$\bbm = (m^1,...,m^n) \rightarrow \left(\beta_1\tfrac{1}{N}\nabla\tilde{H}_{N}(m^1)+\tilde{h}^1, \ldots, 
		\beta_n\tfrac{1}{N}\nabla\tilde{H}_{N}(m^n)+\tilde{h}^n
		\right)$$
		in the sense that
		\begin{equation}\label{eq: disappearing eff ext field}
			\lim_{N\rightarrow\infty} \sup_{\stackrel{m\in\mathcal{M}_{N}^n}{|m^1|,...,|m^n| \le 1}}
			\max_{k=1,...,n} 
			\bigg|
			P^{\mathcal{M}_{N}^\perp}\left(
			\frac{\beta_k}{N}\nabla\tilde{H}_{N}(m^k)+\tilde{h}^k
			\right)
			\bigg | = 0.
		\end{equation}
	\end{lemma}	

	\begin{proof}
	    By \cite[Lemma 17]{belius-kistler} with $\beta=1$ there exists for each $k$ a subspace
		$\mathcal{M}_{N, k} \subset \mathbb{R}^N$ such that
		\begin{equation}\label{eq: MNk invariance}
    		\lim_{N\rightarrow\infty} \sup_{\stackrel{m\in\mathcal{M}_{N,k}}{|m^k| \le 1}}
    		\bigg |
    		P^{\mathcal{M}_{N,k}^\perp}\left(
    		\frac{1}{N}\nabla\tilde{H}_{N}(m^k)+\tilde{h}^k
    		\right)
    		\bigg | = 0.
		\end{equation}
		Letting $\mathcal{M}_N := \mathcal{M}_{N, 1}+ \ldots + \mathcal{M}_{N, n}$ we have
		that $\dim(\mathcal{M}_{N}) \le n N^{\frac{3}{4}}$, 
        and for any $k$ and $m^k\in\mathcal{M}_{N}$ with $|m^k| < 1$ one can decompose
		\begin{equation}\label{eq: m decomposition on Mnks}
		    m^k = v^1 + ... +v^n
		\end{equation}
		for some $v^l\in\mathcal{M}_{N,l}$, $|v^k| < 1,l=1,\ldots,n$. Therefore (using that $\nabla \tilde{H}_N(m)$ is linear in $m$ and $\mathcal{M}_{N,l}\subset\mathcal{M}_{N}$ for all $l$) 
\[
\begin{array}{l}
\sup_{\substack{m^{k}\in\mathcal{M}_{N},|m^{k}|<1}
}\bigg|P^{\mathcal{M}_{N}^{\perp}}\left(\frac{\beta_{k}}{N}\nabla\tilde{H}_{N}(m^{k})+\tilde{h}^{k}\right)\bigg|\\
=\sup_{\substack{\forall l:v^{l}\in\mathcal{M}_{N,l},|m^{l}|<1}
}\bigg|P^{\mathcal{M}_{N}^{\perp}}\left(\frac{\beta_{k}}{N}\sum_{l=1}^{n}\nabla\tilde{H}_{N}(v^{l})+\tilde{h}^{k}\right)\bigg|\\
\le \beta_{k} \sum_{l=1}^{n}\sup_{\substack{v\in\mathcal{M}_{N,l},|v|<1}
} \bigg|P^{\mathcal{M}_{N,l}^{\perp}}\left(\frac{1}{N}\nabla\tilde{H}_{N}(v^{l})\right)\bigg| +\left|P^{\mathcal{M}_{N,k}^{\perp}}\tilde{h}^{k}\right|\\
\le\beta_{k} \sum_{l=1}^{n}\sup_{\substack{v\in\mathcal{M}_{N,l},|v|<1}
} \bigg|P^{\mathcal{M}_{N,l}^{\perp}}\left(\frac{1}{N}\nabla\tilde{H}_{N}(v^{l})+\tilde{h}^{l}\right)\bigg| +c\left(\beta\right)\max_{l=1,...,n}\left|P^{\mathcal{M}_{N,l}^{\perp}}\tilde{h}^{l}\right|
\end{array}
\]
and thus by \eqref{eq: MNk invariance} 
    we obtain \eqref{eq: disappearing eff ext field}.
	\end{proof}	
	\bigskip

	\noindent 
    The next lemma shows that in the absence of external fields, the  partition function restricted to the complements of the previously constructed subsets satisfies the same approximation as the unrestricted partition function. Recall from the beginning of Subsection~\ref{subsec:lwbdextfield} that $E^U$ denotes the expectation with respect to $\bs \in \mathcal{S}_{N-1}^n$ conditioned on $\bs \in U$ for some set $U$.
	\begin{lemma}\label{lem:outersubspace-difference}
		For any $C > 0, K>0,\delta>0$
		\[
		\lim_{\varepsilon \to 0} \limsup_{N\rightarrow\infty} 
		\sup_{\bQ \ge \delta \bI}\sup_{|\beta|\le C}
		\Bigg |
		\frac{1}{N} \log 
		E^{(\mathcal{M}_{N}^n)^\perp}\left[ \1_{\bQ_{\varepsilon}}
		e^{\sum_{k= 1}^n \beta_k \tilde H_N(\sigma^k)}
		\right]
		- \sum_{k=1}^n\mathcal{F}_K(\tilde \beta_k) - \frac{1}{2} \log |\bQ|
		\Bigg | \le \frac{c}{K}
		\]
		where $(\mathcal{M}^n_{N})_{N \ge 1}$ is the sequence of subspaces from Lemma \ref{lem:subspace-construction} and $\tilde{\beta}_1,\ldots,\tilde{\beta}_n$ are the eigenvalues of $\bb^{1/2} \bQ\bb^{1/2}$.
	\end{lemma}	
	\begin{proof} 
	     Recall $N' = \dim(\mathcal{M}_{N}^n) \le n N^{\frac{3}{4}}$.
	     Similarly to in the proof of Lemma \ref{lem:dimensiondrop}, let $w_1,...,w_N$ be an orthonormal basis of $\R^{N}$ such that the space $\mathcal{M}_{N}$ is spanned by the last $N-N'$ of these vectors. Let $\bm{D}$ be the diagonal matrix with $D_{jj} = N \theta_{j/N}$ so that $\tilde{H}_N(\sigma) = \sigma^{\Trans} \bm{D} \sigma$. Let $\bA$ be the $(N-N')\times(N-N')$ minor of $\bm{D}$ when written in the basis $w_1,\ldots,w_N$. By the eigenvalue interlacing inequality and \eqref{eq:classical error} the eigenvalues $Na_1,...,Na_{N-N'}$ of $\bA$ satisfy $Na_j = N\theta_{j/N} + o(1) = (N-N')\theta_{j/(N-N')} + o(N).$ We have
		\begin{align*}
			&E^{(\mathcal{M}_{N}^n)^\perp} \left[ \1_{\bQ_\varepsilon}
			\exp\left(
			{\sum_{k=1}^n \beta_k 
			                            \tilde H_N(\sigma^{k})}
			\right)
			\right] 
			\\
	    	=
			& E^{N-N'} \left[ \1_{\bQ_{\varepsilon}}
			\exp\left(
			{(N-N') \sum_{k=1}^n \beta_k \sum_{j=1}^{N-N'} \theta_{j/N-N'}({\sigma_j^{k}})^2}
			\right)
			\right] e^{o(N)}.
		\end{align*}
		\noindent
		Also
		$$
		| \tilde H^K_N(\sigma^i) - \tilde H_N(\sigma^i) |
		= 
		\bigg | N \sum_{k = 1}^K \sum_{j \in I_k} (x_k-\theta_{j/N}) (\sigma_j^i)^2
		\bigg |
		\le
		\frac{2\sqrt{2}}{K} N
		\sum_{k = 1}^K \sum_{j \in I_k} (\sigma^i_j)^2
		=
	    N	\frac{2\sqrt{2}}{K}, 
		$$
		so we get for bounded $\beta$
		\begin{align*}
			&E^{N-N'} \left[ \1_{\bQ_\varepsilon}
			\exp\left(
			{(N-N') \sum_{k=1}^n \beta_k \sum_{j=1}^{N-N'} \theta_{j/(N-N')}({\sigma_j^{k}})^2}
			\right)
			\right] 
			\\
			= &
			E^{N-N'} \left[ \1_{\bQ_\varepsilon}
			\exp\left(
			{\sum_{k=1}^n \beta_k \tilde H^K_{N-N'}(\sigma^k)}
			\right)
			\right] e^{\mathcal{O}(\frac{N}{K})}.
		\end{align*}
		The claim follows from Lemma \ref{lem:sperhicalintgeneral}.
	\end{proof}	
	\bigskip
	
    The next lemma will be used to show that $\bbm$ with some $|m^k|$ close to 1 give a negligible contribution to the partition function.
	
	\begin{lemma}\label{lem:on too large m}
	    Let $\mathcal{U}_N \subset \mathbb{R}^N$ be a sequence of linear subspaces of dimension 
	    $N' = o\left(\frac{N}{\log N}\right)$.
	    For all $\eta \in (0,1)$ it holds that
	    $$
	    \limsup_{N\to\infty}
	        \frac{1}{N}\log \int_{\mathcal{S}_{N-1}^n} 
	        \1_{\left\{\bs : \ \exists j \in\{1,\ldots,n\} : \left|P^{\mathcal{U}_N}(\sigma^j)\right|^2 > 1 - \eta \right\}}  
	        d\bs
	        < \frac{1}{2} \log \eta.
	    $$
	\end{lemma}
	\begin{proof} First note that
	    \begin{align*}
	        \int_{\mathcal{S}_{N-1}^n} 
	        \1_{\left\{\bs : \ \exists j \in\{1,\ldots,n\} : \left|P^{\mathcal{U}_N}(\sigma^j)\right|^2 > 1 - \eta \right\}}   d\bs
	        \le &
	        \sum_{j=1}^n
	        \int_{\mathcal{S}_{N-1}^n} \1_{\left\{\bs : \left|P^{\mathcal{U}_N}(\sigma^j)\right|^2 > 1 - \eta \right\}}  d\bs
	        \numberthis\label{eq:integral with some large m}
	        \\
	        =&
	        \sum_{j=1}^n
	        \int_{\mathcal{S}_{N-1}} \1_{\left\{\sigma^j : \left|P^{\mathcal{U}_N}(\sigma^j)\right|^2 > 1 - \eta \right\}}  d\sigma^j.
	    \end{align*}
	    By \cite[(2.9)]{belius-kistler}
		    \begin{align*}
		        \int_{\mathcal{S}_{N-1}} \1_{\left\{\sigma^j : \left|P^{\mathcal{U}_N}(\sigma^j)\right|^2 > 1 - \eta \right\}}  d\sigma^j
		        =&
		        \frac{\Gamma\left(\frac{N}{2}\right)}{\pi^{\frac{N'}{2}} \Gamma\left(\frac{N-N'}{2}\right)}
		        \int_{\mathcal{B}_{N'}} \1_{\{m: |m|^2 > 1 - \eta\}}
		        (1-|m|^2)^{\frac{N-N'-2}{2}} 
		        d m,
		    \numberthis\label{eq:integral with some large m 2}
		    \end{align*}
		    where $dm$ denotes Lebesgue measure on $\mathcal{B}_{N'} = \{m \in\mathbb{R}^{N'}: |m| < 1\}$.
		    Since
		    $$
		        \frac{1}{N}
		        \log\left(\frac{\Gamma\left(\frac{N}{2}\right)}{\pi^{\frac{N'}{2}} \Gamma\left(\frac{N-N'}{2}\right)}\right) = o(1)
		    $$
		    and
		    \begin{align*}
		        \frac{1}{N} \log \int_{\mathcal{B}_{N'}} \1_{\{m: |m|^2 > 1 - \eta\}}  (1-|m|^2)^{\frac{N-N'-2}{2}} 
		        d m
		        \ &= \ 
		        \frac{1}{2} \log \eta 
		        + o(1)
		        + 
		            \frac{1}{N} \log \int_{\mathcal{B}_{N'}} d m
		            \\
		        \ &= \ 
		       \frac{1}{2}\log \eta + o(1)
		        + \frac{1}{N}\underbrace{
		      \log\left(\frac{\pi^{\frac{N'}{2}}}{\Gamma(\frac{N'}{2}+1)}\right)
		     }_{=o(N)},
		    \end{align*}
		  the claim follows from \eqref{eq:integral with some large m} and \eqref{eq:integral with some large m 2}.
	\end{proof}
	\bigskip

	\bigskip 

    We now prove that the free energy \eqref{eq:deterministicHamiltonian} of the deterministic Hamiltonian \eqref{eq: binned hamilt def} is bounded above by the corresponding modified TAP free energy from \eqref{eq:TAPK}.
		
	\begin{proposition}\label{prop:upperbound}
        For $K \ge 2$ there is a $C=C(\beta, \bh, \bQ, K)$, such that for $\varepsilon \in (0,C)$, $N$ large enough and $c=c(\beta)$
		\begin{equation}\label{eq:upperbound}
		\tilde{F}^{\varepsilon}_N(\beta, \bh, \bQ)
		\le \frac{1}{N} \sup_{ \bbm: \bM < \bQ } 
		\tilde{F}^K_{\textrm{TAP}}(\bbm) + \frac{c}{K}.
		\end{equation}
	\end{proposition}
	\begin{proof}
		Let $\mathcal{M}_{N}^n = \mathcal{M}_{N} \times ... \times \mathcal{M}_{N}$ be the space from Lemma \ref{lem:subspace-construction} with each of the $n$ components having dimension $N' \le n N^{\frac{3}{4}}$. For any $\bs = (\sigma^{1},...,\sigma^{n}) \in \R^{n \times N}$ let $\bbm$ be the projection onto $\mathcal{M}_{N}^n$, i.e. $\forall i\in\{1,...,n\}, $ $m^{i} := P^{\mathcal{M}_{N}} \sigma^{i}$ and $\bbm = (m^{1},...,m^{n})$.
		\\
		By recentering $\tilde{H}_N(\sigma^k)$ around the $m^k$ as in \eqref{eq:LB-recentering} we get
		\begin{align*}
			& \color{white}=\color{black}
			\int_{\bQ_{\varepsilon}}\exp\left(\sum_{k = 1}^n \left( \beta_k \tilde{H}_N(\sigma^k) + N h^k \cdot \sigma^k \right) \right) d\bs
			\\
			& =
			E\left[
			\1_{\bQ_{\varepsilon}}
			e^{
				\sum_{k = 1}^n \left(
				\beta_k \tilde{H}_N(m^k) + N \tilde{h}^k \cdot m^k
				\right)
			}
			e^{
				\sum_{k = 1}^n \left(
				N \left(\tfrac{\beta_k}{N}\nabla\tilde{H}_N(m^k) + \tilde{h}^k \right) \cdot (\sigma^k - m^k) + \beta_k\tilde{H}_N(\sigma^k - m^k)\right)
			}
			\right].
			\numberthis\label{eq:prop-eq-01}
		\end{align*} 
		Since
		Lemma \ref{lem:subspace-construction} implies that
		\[
		\lim_{N\rightarrow\infty}
		\sup_{m \in \mathcal{M}_{N}^n}
		\sup_{\stackrel{\hat{\bs}  \in (\mathcal{M}^n_{N})^\perp}{|\sigma^{i}|\le 1, \forall i \in \{1,...,n\}}} 
		\left|
		\sum_{k = 1}^n \left(\tfrac{\beta_k}{N}\nabla\tilde{H}_N(m^k) + \tilde{h}^k \right)\cdot (\sigma^k-m^k)
	    \right|
		= 0,
		\]
		the effective external field vanishes and
		\eqref{eq:prop-eq-01} is at most
		\begin{align*}
			E\left[
			\1_{\bQ_{\varepsilon}}
			e^{
				\sum_{k = 1}^n
				\left(
				\beta_k \tilde{H}_N(m^k) + N \tilde{h}^k \cdot m^k
				\right)
			}
			\exp{\left(
				\sum_{k = 1}^n \beta_k\tilde{H}_N(\sigma^k - m^k)
				\right)}
			\right]e^{o(N)}.
			\numberthis\label{eq:prop-eq-02b}
		\end{align*}
		The expectation equals
		\begin{align*}
			E\left[
			e^{\sum_{k = 1}^n 
			\left(
				\beta_k \tilde{H}_N(m^k) + N \tilde{h}^k \cdot m^k
				\right)
				}
			E\left[\1_{\bQ_\varepsilon}\exp\left(
			\sum_{k = 1}^n 
			\beta_k\tilde{H}_N(\sigma^k-m^k)
			\right) \bigg | \bbm \right]\right]	\numberthis\label{eq:prop-eq-03b}
		\end{align*} 
		where the $E[\cdot | \bbm ]$-law of $\bs - \bbm$ is the uniform distribution on the cartesian product of the $n$ spheres $\mathcal{M}_{N}^\perp \cap \mathcal{S}_{N-1}(\sqrt{1 - |m^{k}|^2})$ for $k\in\{1,...,n\}$. 
        
		Note that for all $k,\ell \in \{1,...,n\}$
		\begin{align*}
			&(\sigma^k-m^k) \cdot (\sigma^\ell-m^\ell) - (Q_{k,\ell} - (\bM)_{k,\ell})
			\\ = &
			\sigma^k \cdot \sigma^\ell - Q_{k,\ell} - 
			\underbrace{(\sigma^k - m^k) \cdot m^\ell}_{=0} -
			\underbrace{(\sigma - m)^\ell \cdot m^k}_{=0} -
			\underbrace{(m^k \cdot m^\ell - (\bM)_{k,\ell})}_{=0}
			,
		\end{align*}
		since $m^k=P^{\mathcal{M}_{N}} \sigma^{k} \in \mathcal{M}_N$ and $\sigma^k-m^k \in \mathcal{M}_N^\perp$,
		so
		\begin{equation}\label{eq: Qeps inclusion}
			\{\bs: \bs \in \bQ_\varepsilon\} = \{\bs: \bs-\bbm \in (\bQ-\bM)_\varepsilon\}.
		\end{equation}
		Let $\hat{\sigma}^{k} = \tfrac{\sigma^k-m^k}{\sqrt{1-|m^k|^2}}$. Using also that $\tilde{H}_N$ is $2$-homogeneous (recall \eqref{eq:deterministicHamiltonian}) the expression in \eqref{eq:prop-eq-03b} equals

		\begin{align}\label{eq:prop-eq-04b}
			E\left[ 
			e^{\sum_{k = 1}^n 
			    \left(
				    \beta_k \tilde{H}_N(m^k) + N \tilde{h}^k \cdot m^k
				\right)
			}
			E\left[\1_{\{ \bs-\bbm \in (\bQ-\bM)_\varepsilon\}}
            e^{\sum_{k=1}^{n} \beta_k (1-|m^k|^2) \tilde{H}_{N}\left(\hat{\sigma}^{k}\right)}
            \bigg | \bbm \right]\right].
		\end{align}
		Let $\eta > 0$ and define $W_j(\eta) = \{\bs: |m^j|^2 \le 1 - \eta\}$ and $W(\eta)=\bigcap_{j=1}^n W_j(\eta)$. Using that $N^{-1}\tilde{H}_N(\sigma), \beta_k, |\tilde{h}^k|$ are all bounded and Lemma \ref{lem:on too large m} we obtain
        \begin{equation}\label{eq:W int}
        \begin{array}{l}
        E\left[\1_{W_j\left(\eta\right)^{c}}e^{\sum_{k=1}^{n}\left(\beta_{k}\tilde{H}_{N}(m^{k})+N\tilde{h}^{k}\cdot m^{k}\right)}E\left[\1_{\{\bs-\bbm \in (\bQ-\bM)_\varepsilon\}}
        e^{\sum_{k=1}^{n}\beta_{k}(1-|m^{k}|^{2})\tilde{H}_{N}\left(\hat{\sigma}^{k}\right)}
        \bigg|\bbm\right]\right]\\
        \le e^{cN}E\left[\1_{{W_j(\eta)}^{c}}\right]\le e^{N\left(c+\log\eta\right)}\le\exp\left(N\tilde{F}_{{\rm TAP}}^{K}\left(0\right)\right),
        \end{array}
        \end{equation}
        if $\eta$ is picked small enough depending on $\bQ,\beta$, and $N$ is large enough. To conclude \eqref{eq:upperbound} it thus suffices to bound
       $$ E\left[\1_{W\left(\eta\right)}e^{\sum_{k=1}^{n}\left(\beta_{k}\tilde{H}_{N}(m^{k})+N\tilde{h}^{k}\cdot m^{k}\right)}E\left[\1_{\{\bs-\bbm \in (\bQ-\bM)_\varepsilon\}}
       e^{\sum_{k=1}^{n}\beta_{k}(1-|m^{k}|^{2})\tilde{H}_{N}\left(\hat{\sigma}^{k}\right)}
       \bigg|\bbm\right]\right].$$       
       	Recall the matrix $\hat{\bQ}(\bbm)$ given by
		\begin{align*}
		    \hat{\bQ}(\bbm)_{ij} = \frac{Q_{ij} - m^i \cdot m^j}{\sqrt{1-|m^i|^2}\sqrt{1-|m^j|^2}}.
		\end{align*}
		Let $\varepsilon' = \varepsilon\eta^{-1}$. For $\bs \in W(\eta)$	\begin{align}\label{eq:Q-inclusion}
    		\{\bs: \bs-\bbm \in (\bQ-\bM)_\varepsilon\}
    		\subset
    		\{\bs: \hat{\bs} \in \hat{\bQ}(\bbm)_{\varepsilon'}\}.
		\end{align}
		Using this we can bound \eqref{eq:W int} from above by
		\begin{align}\label{eq:prop-eq-05b}
			E\left[ 
			\1_{W(\eta)}
			e^{\sum_{k = 1}^n 
				\beta_k \tilde{H}_N(m^k) + N \tilde{h}^k \cdot m^k}
			E\left[\1_{\{\hat{\bs}\in\hat{\bQ}(\bbm)_{\varepsilon'}\}}\exp\left(
			\sum_{k=1}^{n} \beta_k (1-|m^k|^2) \tilde{H}_{N}\left(\hat{\sigma}^{k}\right)
			\right)\bigg | \bbm \right]\right].
		\end{align}
		Because $\hat{\sigma}^{k}$ is distributed uniformly on $\mathcal{S}_{N-1}^n \cap (\mathcal{M}_{N}^n)^\perp$ under $E[\cdot|\bm{m}]$ we can also write this as
		\begin{align}\label{eq:prop-eq-06b}
			E\left[ 
			\1_{W(\eta)}
			e^{\sum_{k = 1}^n 
				\beta_k \tilde{H}_N(m^k) + N \tilde{h}^k \cdot m^k}
			E^{(\mathcal{M}_{N}^n)^{\perp}}\left[\1_{\hat{\bQ}(\bbm)_{\varepsilon'}}\exp\left(
			\sum_{k=1}^{n} \beta_k (1-|m^k|^2) \tilde{H}_{N}\left(\hat{\sigma}^{k}\right)
			\right)\right]\right].
		\end{align}
		Note that \eqref{eq:prop-eq-06b} is bounded from above by
		\begin{equation}\label{eq:prop-eq-indicator-on-Q}
			\begin{array}{rcl}
				&&
				E\left[ 
				\1_{W(\eta)}
				\1_{\{\bbm: \hat{\bQ}(\bmm)>\delta\bI\}}
				e^{\sum_{k = 1}^n 
					\beta_k \tilde{H}_N(m^k) + N \tilde{h}^k \cdot m^k}
				E^{(\mathcal{M}_{N}^n)^{\perp}}\left[\1_{\hat{\bQ}(\bbm)_{\varepsilon'}}e^{
					\sum_{k=1}^{n} \beta_k (1-|m^k|^2) \tilde{H}_{N}\left(\hat{\sigma}^{k}\right)}
				\right]\right]
				\\
				&& +
				e^{cN} E\left[
				\1_{\{\bbm: \hat{\bQ}(\bmm)>\delta\bI\}^c}
			    E^{(\mathcal{M}_{N}^n)^{\perp}}\left[\1_{\hat{\bQ}(\bbm)_{\varepsilon'}}
				\right]\right],
			\end{array}
		\end{equation}
		for any $\delta>0$, where we have crudely bounded all terms in $\exp$ by $cN$ to arrive at the second term. Since under $E^{(\mathcal{M}_{N}^n)^\Trans}$ the $\hat{\sigma}^k$ are \iid uniformly distributed on a sphere of radius $1$ in the subspace $\mathcal{M}_{N}^\perp$ of dimension $N-N'$ we have by \eqref{eq: constrained volume small eval} (with $N-N'$ in place of $N$) and \eqref{eq: mat det general} that $E^{(\mathcal{M}_{N}^n)^{\perp}}[\1_{\hat{\bQ}(\bbm)_{\varepsilon'}}] \le e^{ \frac{1}{2}\log (2\delta n^{n-1}) (N-N')}$ for $\epsilon' \le \delta$, so there is $\delta>0$ such that the second term of \eqref{eq:prop-eq-indicator-on-Q} is at most $\exp(N \tilde{F}_{{\rm TAP}}^{K}(0))$. It thus suffices to bound the first term of \eqref{eq:prop-eq-indicator-on-Q} to prove \eqref{eq:upperbound} (cf. \eqref{eq:W int}).
		Now we can apply Lemma \ref{lem:outersubspace-difference} with $({\beta}_{\bbm})_k=\beta_k (1 - | m^k|^2)$ in place of $\beta_k$,
		$\bb_{\bbm} = \diag\, {\beta}_{\bbm}\in \mathbb{R}^{n \times n}$ in place of $\bb$ and $\hat{\bQ}$ in place of $\bQ$ to bound the first term of \eqref{eq:prop-eq-indicator-on-Q} by 
		\begin{align*}
    		E\Bigg[\1_{W(\eta)}\1_{\{\bbm: \hat{\bQ}(\bmm)>\delta\bI\}} \exp\Bigg(
    		\sum_{k = 1}^n \beta_k & \tilde{H}_N(m^k) + N \tilde{h}^k \cdot m^k
    		\numberthis\label{prop-eq-05}
    		\\
    		&
    		+
    		N \sum_{k=1}^n \mathcal{F}_K( \tilde{\beta}_k{(\bbm)} ) +  \frac{N}{2}\log |\hat{\bQ}(\bbm)| \Bigg)
    		\Bigg] e^{o(N) + \frac{c N}{K}},
		\end{align*}
  recalling from \eqref{eq:TAPK} that $\tilde{\beta}_k(\bbm)$ are the eigenvalues of the symmetric positive semi-definite matrix 
		$$\bb^{1/2}(\bQ - \bM) \bb^{1/2}=\bb_{\bbm}^{\frac{1}{2}}  \hat\bQ(\bbm) \bb_{\bbm}^{\frac{1}{2}}.$$   
		Since each $m^k$ is a projection onto $\mathcal{M}_{N}$ we can use \cite[(2.9)]{belius-kistler} to write the expectation as
		\begin{align*}
			& \left(\frac{1}{\pi^{\frac{N'}{2}}}\frac{\Gamma\left(\frac{N}{2}\right)}{\Gamma\left(\frac{N-N'}{2}\right)}\right)^n 
			\int_{(\mathcal{B}_{N}(\sqrt{1-\eta})\cap \mathcal{M}_{N})^n} 
			    \1_{\{\bbm: \hat{\bQ}(\bmm)>\delta\bI\}}
			\prod_{k=1}^n  (1-|m^{k}|^2)^{\frac{N-N'-2}{2}}\color{black}
			\numberthis\label{eq:prop-eq-06}
			\\
			& \quad
			\times
			\exp\left(
			\sum_{k = 1}^n \left(\beta_k\tilde{H}_N(m^k) + N \tilde{h}^k \cdot m^k\right)
			+ N \sum_{k =1}^n \mathcal{F}_K(\tilde{\beta}_k{(\bbm)}) + \frac{N}{2}\log |\hat{\bQ}(\bbm)|\right) d \bbm,
		\end{align*}
		where $\mathcal{B}_N(r)$ denotes the ball of radius $n$ in $\mathbb{R}^N$ and $d \bbm$ is the $n N'$-dimensional Lebesgue measure on $\mathcal{M}_{N}^n$.
        We have
        \begin{equation*}
			\log |\hat{\bQ}(\bbm)| + \sum_{k=1}^n \log(1-|m^k|^2)
			 \stackrel{\eqref{eq: determinant transformation Q hat to Q-M}}{=}
			\log |\bQ - \bM|.
		\end{equation*}
		%
		%
		Therefore recalling \eqref{eq:TAPK} we have that \eqref{eq:prop-eq-06} equals
		\[
		\left(\frac{1}{\pi^{\frac{N'}{2}}}\frac{\Gamma\left(\frac{N}{2}\right)}{\Gamma\left(\frac{N-N'}{2}\right)}\right)^n
		\int_{{(\mathcal{B}_{N}(\sqrt{1-\eta})\cap \mathcal{M}_{N})^n}} 
		    \1_{\{\bbm: \hat{\bQ}(\bmm)>\delta\bI\}}
		\exp\left(\tilde{F}_{\text{TAP}}^K(\bbm) - \tfrac{N'+2}{2}\sum_{k=1}^n \log(1-|m^{k}|^2) \right) d\bbm.
		\numberthis\label{eq:prop-eq-07}
		\]
		Since the prefactor in \eqref{eq:prop-eq-07} is at most $e^{o(N)}$ and $\int_{{\mathcal{M}_{N}^n}} d\bbm \le \int_{B_{nN'}(n)} d\bbm = \frac{\pi^{\frac{nN'}{2}}}{\Gamma\left(\frac{nN'}{2}+1\right)} n^{n N'} = e^{o(N)}$ the expectation in \eqref{prop-eq-05} is bounded from above by
		\[
		\exp\left(
		\sup_{\bm{m}\in{\mathcal{M}_{N}^n}, \hat{\bQ}(\bbm)>\delta\bI} \tilde{F}_{\text{TAP}}^K(\bm{m)} +  \frac{cN}{K}\right),
		\numberthis\label{eq:prop-eq-09}
		\]
		for $N$ large enough. As $\bm{D}^{\Trans} \bA \bm{D}>0$ and $\bm{D}$ invertible implies that $\bA>0$ we have that $\hat{\bQ} > \delta \bI$ implies
		$\bQ - \bM > 0$, so the claim \eqref{eq:upperbound} follows.
	\end{proof}

	\subsection{Location of the maximizer}	\label{subsec:location_max}
    
    In this subsection we will derive Proposition~\ref{prop:TAP_UB} for the free energy in terms of the TAP free energy $F_{\text{TAP}}^K(\bbm)$ from the upper bound Proposition~\ref{prop:upperbound} for the free energy in terms of the modified TAP free energy $\tilde{F}_{\text{TAP}}^K(\bbm)$. To do so we will show that the maximum of $\tilde{F}_{\text{TAP}}^K(\bbm)$, is attained at some $\bbm \in \cM$. 
    Similarly to \eqref{eq: mmT plefka equivalence} we have
    \begin{equation}\label{eq:plefka threeway}
        \bbm \in \cM
        \ \stackrel{\eqref{eq:HT-condition}}{\Leftrightarrow} \ 
        \beta \in {\rm{HT}}(\bQ-\bbm\bbm^\Trans)
        \ \Leftrightarrow \  
        \tilde{\beta}(\bbm) \in {\rm{HT}}(\bQ),
    \end{equation}
    i.e. $m$ satisfies the Plefka condition if and only if the ``effective temperature after recentering" $\tilde{\beta}(\bbm)$ lies in the high temperature region ${\rm{HT}}(\bQ)$.    
    Therefore once we have proven that the maximizer of $\tilde{F}_{\text{TAP}}^K(\bbm)$ satisfies $\bbm \in \cM$ we will be able to derive the upper bound Proposition~\ref{prop:TAP_UB} for the free energy $F_N^\varepsilon$ from Proposition~\ref{prop:upperbound} and Lemma~\ref{lem:Plefka} by taking the limit $K\to\infty$.

	To obtain nice formulas for the derivatives of $\tilde{F}_{\text{TAP}}^K(\bbm)$, we will interpret the terms
	\[
	\sum_{k =1}^n \mathcal{F}_K(\tilde{\beta}_k{(\bbm)})
	\quad \text{and} \quad  \log|\bQ-\bM|
	\] 
	of \eqref{eq:TAPK} as traces of primary matrix functions \cite[Chapter~1]{HighamFunctionsMatrices}. 
	
	\begin{defn}
		Given a scalar function $f$ and a real symmetric matrix $\bA = \bU \bd \bU^\Trans \in \R^{n\times n}$ we define the primary matrix function $f(\bA)$ associated with $f$ by
		\[
		f(\bA) := \bU f(\bd) \bU^\Trans \qquad \text{where} \qquad f(\bd) = \diag(f(\lambda_1), \dots, f(\lambda_n)).
		\]
		These matrix valued functions are well-defined if $f(\lambda_i)$ is well-defined for all $i \leq n$.
	\end{defn}
	
	\begin{remark}
		The primary matrix functions of \cite[Chapter~1.2]{HighamFunctionsMatrices} are defined more generally in terms of the Jordan canonical form. However, in this work, we only deal with diagonalizable matrices, so the definition simplifies.
	\end{remark}

	It follows that
	\[
	   \sum_{k = 1}^n \cF_K(\tilde{\beta}_k (\bbm) ) = \Tr \big(  \cF_K (\bb^{\frac{1}{2}} ( \bQ - \bM ) \bb^{\frac{1}{2}})   \big),
	\]
	where $\cF_K (\bA)$ is the primary matrix function associated with $\cF_K(x)$, and (for $\bmm < \bQ$)
	\[
	   \log|\bQ-\bM| = \Tr\big( \log ( \bQ - \bM))  \big),
	\]
	 where $\log (\bA)$ is the primary matrix functions associated with $\log(x)$.
	Replacing the corresponding terms of \eqref{eq:TAPK} we arrive at the matrix form of $\tTAP^K$
	\begin{align}
		\tTAP^K (\bbm) &=  \sum_{k = 1}^n \bigg( \beta_k \tilde H_N ( m^k ) + N m^k \cdot \tilde{h}^k \bigg) \notag
		+ N \Tr ( \cF_K( \bb^{\frac{1}{2}} (\bQ - \bM) \bb^{\frac{1}{2}} ) )
		\\&\quad
		+ \frac{N}{2}  \Tr ( \log ( \bQ - \bM )) \label{eq:TAPevform}.
	\end{align}
	We want to study the critical point condition of the maximizers of this function, which will require a formula to differentiate primary matrix functions.  
	
\begin{lemma}\label{lem:matderiv} 
Let $f$ be a
scalar function which is smooth in its domain
, and let $\bA(\alpha)$ be a smooth map from a subset of $\mathbb{R}$ into the subset of $\mathbb{R}^{n\times n}$ on which $f(\bA)$ is well-defined. Then $f_{ab}(\bA(\alpha)),a,b=1,\ldots,n$ is smooth and
\begin{equation}\label{eq:matrixderivativefact}
    \partial_{\alpha}\Tr (f\left(\bA\left(\alpha\right)\right)) = \Tr\left(f'\left(\bA\left(\alpha\right)\right)\partial_{\alpha}\bA(\alpha)\right).
\end{equation}

In particular, for positive definite $\bA(\alpha)$
	\begin{equation}\label{eq: tr log deriv}
		\partial_\alpha \Tr ( \log(\bA(\alpha)) ) = \Tr ( \bA(\alpha)^{-1} \partial_{\alpha}\bA(\alpha) )
	\end{equation}
		and
	\begin{equation}\label{eq: tr F deriv}
		\partial_\alpha \Tr ( \cF_K(\bA(\alpha)) ) = \Tr (  v_{\mu_K} (2\bA(\alpha)) \partial_{\alpha}\bA(\alpha) )  - \frac{1}{2} \Tr (  \bA(\alpha)^{-1} \partial_{\alpha}\bA(\alpha) ).
	\end{equation}
\end{lemma}
\begin{proof}
    By linearity we have
	\begin{equation}
		\partial_{\alpha}\Tr (f\left(\bA\left(\alpha\right)\right)) = \Tr (\partial_{\alpha} f\left(\bA\left(\alpha\right)\right)).
	\end{equation}
	
	To manipulate the right-hand side we use the concepts of \cite[Chapter~3.2]{HighamFunctionsMatrices}. Let $L(\bA,\bC)$ denote the Fr\'echet derivative of $f(\bA)$  in the direction $\bC$ defined in \cite[(3.6)]{HighamFunctionsMatrices}. Then
	\begin{equation}\label{eq: gateaux}
		\partial_{\alpha} f\left(\bA\left(\alpha\right)\right)  = L(\bA(\alpha),\partial_{\alpha}\bA(\alpha)).
	\end{equation}
	We write $\bA(\alpha) = \bU(\alpha) \bd(\alpha) \bU(\alpha)^\Trans$ in its eigendecomposition where $\bd(\alpha) = \diag(\lambda_1(\alpha), \dots, \lambda_n(\alpha))$ are the eigenvalues of $\bA$. Let $\odot$ denote the Hadamard product. By \cite[Corollary~3.12 (see also the top of p. 61 and the remark before equation (3.13))]{HighamFunctionsMatrices} we have
	\[
	L(\bA(\alpha),\partial_{\alpha}\bA(\alpha)) = \bU(\alpha) ( \bD(\alpha) \odot \bU(\alpha)^\Trans \partial_{\alpha}\bA(\alpha) \bU(\alpha)) \bU(\alpha)^\Trans
	\]
	where $\bD$ is given by
	\begin{equation}\label{eq:diffmatrix}
		\bD = \bD_{f(\bA)}= [ \Delta f(\lambda_i, \lambda_j) ]_{i,j \leq n}  \quad\text{and}\quad \Delta f(\lambda, \lambda') = \begin{cases}
			\frac{f(\lambda) - f(\lambda')}{ \lambda - \lambda'} & \lambda \neq \lambda'\\
			f'(\lambda)& \lambda = \lambda'.
		\end{cases} 
\end{equation}
    Thus using the invariance of the trace under cyclic permutations
	\[
	   \Tr( \partial_{\alpha} f(\bA(\alpha)) )= \Tr( \bU ( \bD \odot \bU^\Trans \partial_{\alpha}\bA \bU) \bU^\Trans  ) =  \Tr( \bI ( \bD \odot \bU^\Trans \partial_{\alpha}\bA \bU) ) = \Tr( \bU  ( \bI \odot \bD) \bU^\Trans \partial_{\alpha}\bA ),
	\]
	 where the last inequality follows since for any symmetric matrices $\bA, \bB, \bC$ 
	\[
	   \Tr( \bA (\bB \odot \bC) ) = \Tr( (\bA \odot \bB) \bC) ).
	\]
	The claim then follows since $\bI \odot \bD = \diag(f'( \lambda_1 ), \dots, f'(\lambda_n))$ so
	\[
	   \bU  ( \bI \odot \Delta ) \bU^\Trans = f'(\bA).
	\]
	This proves \eqref{eq:matrixderivativefact}.
	From \eqref{eq:matrixderivativefact} we now derive \eqref{eq: tr log deriv}-\eqref{eq: tr F deriv}.  Recall that both the scalar functions $\log(x)$ 	
and $\cF_K(x)$ are smooth on $(0,\infty)$.
To prove the first formula, we have $\frac{d}{dx} \log(x) = \frac{1}{x}$, so \eqref{eq:matrixderivativefact} implies
	\[
	   \partial_\alpha \Tr ( \log(\bA(\alpha)) ) = \Tr ( \bA(\alpha)^{-1} \partial_{\alpha}\bA(\alpha) ),
	\]
	where $\bA^{-1}$ is the primary matrix function arising from $f(x) = x^{-1}$ applied to $\bA$, which coincides with the usual matrix inverse of $\bA$.
 
    By \cite[Lemma~12]{belius-kistler} or \cite[Theorem~6]{GuMa05} it holds that
     	\begin{equation}\label{eq: F deriv}
     	\cF_K'(\beta) = v_{\mu_K}(2\beta) - \frac{1}{2\beta}\text{ for all }z>0,
        \end{equation}
	 so \eqref{eq:matrixderivativefact} implies
	\[
	\partial_\alpha \Tr ( \cF_K(\bA(\alpha)) ) = \Tr (  v_{\mu_K} (2\bA(\alpha)) \partial_{\alpha}\bA(\alpha) )  - \frac{1}{2} \Tr (  \bA(\alpha)^{-1} \partial_{\alpha}\bA(\alpha) ),
	\]
	where again we can interpret $\bA^{-1}$ arising from the primary matrix function $x^{-1}$ as the usual inverse.
\end{proof}

	We now study the maximizers of $\tTAP^K (\bbm)$ defined in \eqref{eq:TAPevform}. First note that the set
	\begin{equation}\label{eq: open set}
	    \{ \bbm : \bM < \bQ \}
	\end{equation}
	is an open set. Also because $\TAP^K(\bbm)$ diverges to $-\infty$ as $|\bQ - \bM| \to 0$ the global maximum lies in \eqref{eq: open set}.
	We vectorize the matrix $\bbm = (m_1^1,\dots,m_1^n,\dots, m_N^1,\dots,m_N^n) \in \R^{Nn}$ and treat $\tilde{F}_{\rm{TAP}}^K$ as a function from $\R^{Nn} \mapsto \R$. With this vectorization the gradient of $\tTAP^K (\bbm)$ is a vector in $\R^{Nn}$ and its Hessian is an $Nn \times Nn$ block matrix which consists of $N \times N$ blocks of size $n \times n$. That is, for any sufficiently regular function $f: \R^{Nn} \to \R$,
	\[
	\nabla  f(\bbm) = \big( \partial_{m_1^1} f(\bbm), \cdots, \partial_{m_1^n} f(\bbm), \cdots, \partial_{m_N^{1}} f(\bbm), \cdots , \partial_{m_N^n} f(\bbm) \big)^\Trans \in \mathbb{R}^{Nn}
	\]
	and
	\[
	\nabla^2  f(\bbm) = \begin{bmatrix} 
	\bm{f}_{1,1} & \cdots & \bm{f}_{1,N}\\
	\vdots  & \ddots & \vdots\\
	\bm{f}_{N,1}  & \cdots & \bm{f}_{N,N}
	\end{bmatrix} \in \mathbb{R}^{Nn \times Nn}
    , \
		\bm{f}_{i, j} = \begin{bmatrix} 
	\partial_{m_i^1} \partial_{m_j^1} f(\bbm)  &\cdots & \partial_{m_i^1} \partial_{m_j^n} f (\bbm)\\
	\vdots & \ddots & \vdots\\
	\partial_{m_i^n} \partial_{m_j^1}f (\bbm) & \cdots & \partial_{m_i^n} \partial_{m_j^n} f (\bbm)
	\end{bmatrix}\in \mathbb{R}^{n\times n}.
	\]

    Since $\tTAP^K$ is
    smooth
    its local maximizers $\bbm^*$ satisfy
	\[
	\nabla \tTAP^K(\bbm^*)=0 \quad \text{and} \quad \nabla^2  \tTAP^K (\bbm^*) \leq 0.
	\]
    \begin{remark}
    Since we formally only proved that $\mathcal{F}_K$ is smooth on $(0,\infty)$ and not on $[0,\infty)$ we can strictly speaking only claim that the term of $\tTAP^K$ involving $\mathcal{F}_K$ (and hence $\tTAP^K$ itself) is smooth when all entries of $\beta$ are positive, so that $\bb^{1/2}(\bQ-\bM)\bb^{1/2}$ is positive definite. In the proofs below we assume that $\beta$ has positive entries (which also simplifies the arguments) and later remove the assumption by approximation. With additional effort one could prove that $\mathcal{F}_K$ is in fact smooth on $[0,\infty)$ and extend all the arguments below to cover $\beta$ with zero components, but we refrain from this.
    \end{remark}
    
	The part 
	$$
	    f(\bbm) = 
	    \sum_{\ell=1}^n \beta_{\ell} \tilde H_N(m^\ell),
	$$
	of $\tTAP^K (\bbm)$ that depends on the Hamiltonian has a simple Hessian given by the $Nn \times Nn$ matrix
	\begin{equation}\label{eq: hamilt hess}
	\nabla^2 f(\bbm)
	= 2N\diag( \theta_1 \bb, \theta_{\frac{N - 1}{N}} \bb, \dots, \theta_{\frac{1}{N}} \bb) = 2N\begin{bmatrix}
			\theta_{1} \bb & \cdots  & \bm 0_n\\
			\vdots & \ddots  & \vdots\\
			\bm 0_n & \cdots  & \theta_{\frac{1}{N}} \bb
		\end{bmatrix}
    \end{equation}
    (recall \eqref{eq:discrete1}-\eqref{eq: binned hamilt def} and \eqref{eq:TAPevform}).
		
The other part of $\tTAP^K(\bbm)$ 
is $Ng(\bM)$ for
$$ g(\bA) = \Tr(\cF_{K}(\bb^{\frac{1}{2}}(\bQ-\bA)\bb^{\frac{1}{2}}))+\frac{1}{2}\Tr(\log(\bQ-\bA)).$$
Its Hessian is given by the next lemma. 
\begin{lemma}\label{lem: hessian}
Assume that $\beta_k>0$ for $k=1,\ldots,n$.
Then for $\bQ - \bM > 0$, the Hessian $\nabla^2 g(\bM)$ is the $Nm \times Nm$ matrix
	\begin{equation}\label{eq:Hessianformula}
		-2 \begin{bmatrix}
			\bb^{\frac{1}{2}} v_{\mu_K}(2 \bm{Q}_{\bm{m}} )\bb^{\frac{1}{2}} & \cdots  & \bm 0_n\\
			\vdots & \ddots  & \vdots\\
			\bm 0_n & \cdots  &  \bb^{\frac{1}{2}} v_{\mu_K}(2 \bm{Q}_{\bm{m}} )\bb^{\frac{1}{2}} 
		\end{bmatrix}+\bm{L},
	\end{equation}
    where
    \begin{equation}\label{eq:Q_m}
        \bm{Q}_{\bm{m}} := \bb^{\frac{1}{2}} (\bQ - \bM) \bb^{\frac{1}{2}},
    \end{equation}
	and $\bm{L}$ is a matrix of rank at most $n^4$.
\end{lemma}
\begin{proof}
We have
\[
\frac{1}{2}\Tr(\log(\bQ-\bM)) = \frac{1}{2}  \Big( \Tr ( \log ( \bb^{\frac{1}{2}} (\bQ - \bM) \bb^{\frac{1}{2}} ) ) - \Tr ( \log(\bb) ) \Big).
\]
By Lemma~\ref{lem:matderiv},
$$ \partial_{m^\ell_i} \frac{1}{2} \Tr\Big( \log ( \bb^{\frac{1}{2}} (\bQ - \bM) \bb^{\frac{1}{2}} \Big) = - \frac{1}{2} \Tr( \bm{Q}_{\bm{m}}^{-1}  \bb^{\frac{1}{2}} \partial_{m_i^\ell} \bM \bb^{\frac{1}{2}} ),$$
and
$$ \partial_{m^\ell_i} \Tr \Big( \cF_K( \bb^{\frac{1}{2}} (\bQ - \bM) \bb^{\frac{1}{2}} ) \Big) = -\Tr(v_{\mu_K}(2\bm{Q}_{\bm{m}}) \bb^{\frac{1}{2}} \partial_{m_i^\ell} \bM \bb^{\frac{1}{2}} )  + \frac{1}{2} \Tr( \bm{Q}_{\bm{m}}^{-1}  \bb^{\frac{1}{2}} \partial_{m_i^\ell} \bM \bb^{\frac{1}{2}} ).$$
Thus the first derivatives of $g(\bmm)$ equal
\begin{align*}
	\partial_{m_i^\ell} g(\bM)
	&= - \Tr(v_{\mu_K}(2\bm{Q}_{\bm{m}}) \bb^{\frac{1}{2}} \partial_{m_i^\ell} \bM \bb^{\frac{1}{2}} ).
\end{align*}
To obtain the second derivatives let $h:\mathbb{R}^{n\times n}\to\mathbb{R}^{n\times n}$ be given
by $$h(\bA)=v_{\mu_{K}}(2\bb^{1/2}(\bm{Q}-\bA)\bb^{1/2}).$$ By the
product and chain rules
\begin{equation}
\begin{array}{rcl}
\partial_{m_{j}^{\ell'}}\Tr(h(\bM)\bb^{\frac{1}{2}}\partial_{m_{i}^{\ell}}\bM\bb^{\frac{1}{2}})&=&
\sum_{ab}h_{ab}(\bM)\partial_{m_{j}^{\ell'}}(\bb^{\frac{1}{2}}\partial_{m_{i}^{\ell}}\bM\bb^{\frac{1}{2}})_{ab})+\\
& &\sum_{abcd}\partial_{\bA_{cd}}h_{ab}(\bM)\partial_{m_{j}^{\ell'}}(\bM)_{cd}(\bb^{\frac{1}{2}}\partial_{m_{i}^{\ell}}\bM\bb^{\frac{1}{2}})_{ab}).
\end{array}\label{eq: sec deriv}
\end{equation}
Therefore
\begin{equation}\label{eq:hessianWL}
    \nabla^{2}g(\bM)=\bm{W}+\bm{L},
\end{equation}
where
\[
\bm{W}_{((i,\ell),(j,\ell'))}=-\sum_{ab}h_{ab}(\bM)\partial_{m_{j}^{\ell'}}(\bb^{\frac{1}{2}}\partial_{m_{i}^{\ell}}\bM\bb^{\frac{1}{2}})_{ab}),
\]
and
\begin{equation}\label{eq:Lmatrix}
    \bm{L}_{((i,\ell),(j,\ell'))}=-\sum_{abcd}\partial_{\bA_{cd}}h_{ab}(\bM)\partial_{m_{j}^{\ell'}}(\bM)_{cd}(\bb^{\frac{1}{2}}\partial_{m_{i}^{\ell}}\bM\bb^{\frac{1}{2}})_{ab}).
\end{equation}
We have
\[ \bm{W}_{((i,\ell),(j,\ell'))}=-\Tr( v_{\mu_K}(2 \bm{Q}_{\bm{m}}) \bb^{\frac{1}{2}} \partial_{m_i^\ell}\partial_{m_j^{\ell'}} \bM \bb^{\frac{1}{2}} ).
\]
Also
\[
\partial_{m_j^{\ell'}} \bM = \begin{bmatrix}
	0 & \cdots & m_j^1 &\cdots & 0\\
	\vdots & \ddots & \vdots &\ddots & \vdots\\
	m_j^1 & \cdots & 2 m_j^{\ell'} &\cdots &m_j^n\\
	\vdots & \ddots & \vdots &\ddots & \vdots\\
	0 & \cdots & m_j^n &\cdots & 0
\end{bmatrix}
\]
where only the $\ell'$-th row and column is non-zero, and
\begin{equation}\label{eq:hessianM}
	\partial_{m_i^\ell} \partial_{m_j^{\ell'}} \bM = ( \delta_{i = j} ( \delta_{(\ell,\ell') = (a, b)} + \delta_{(\ell,\ell') = (b, a)}  ))_{a,b \leq n}
\end{equation}
which is a zero matrix if $i\ne j$ and if $i=j$ it is a zero matrix except for the entries $(\ell,\ell')$ and $(\ell',\ell)$ which takes values $1$ if $\ell \neq \ell'$ (on the off-diagonal) and value $2$ if $\ell = \ell'$ (on the diagonal). 
Using this and the symmetry of the matrix $v_{\mu_K}(2 \bm{Q}_{\bm{m}} )$ we obtain

\[
\bm{W}_{((i,\ell),(j,\ell'))}= - \delta_{i=j}2 v_{\mu_K}(2 \bm{Q}_{\bm{m}})_{\ell,\ell'}\beta^{1/2}_\ell \beta^{1/2}_{\ell'} =  -2 \delta_{i=j} (\bb^{\frac{1}{2}} v_{\mu_K}(2 \bm{Q}_{\bm{m}} )\bb^{\frac{1}{2}})_{\ell,\ell'}.
\]
This gives the first term in  \eqref{eq:Hessianformula}. 

As for $\bm{L}$, we can write its entries as
\[
\bm{L}_{((i,\ell),(j,\ell'))}=-\left(\sum_{a,b,c,d=1}^{n}d_{a,b,c,d}v^{c,d}\left(w^{a,b}\right)^{\Trans}\right)_{((i,\ell),(j,\ell'))}
\]
where
\[
d=\partial_{\bA_{cd}}h_{ab}(\bM),
\]
and $v^{c,d}\in\mathbb{R}^{Nn}$ is given by
\[
v_{j,\ell'}^{cd}=\left(\partial_{m_{j}^{\ell'}}\bM\right)_{cd},
\]
and $w^{a,b}\in\mathbb{R}^{Nn}$ by
\[
w_{i,\ell}^{ab}=(\bb^{\frac{1}{2}}\partial_{m_{i}^{\ell}}\bM\bb^{\frac{1}{2}})_{ab},
\]
so that $v^{c,d}\left(w^{a,b}\right)^{\Trans}$ is an $Nn \times Nn$ matrix. Thus $\bm{L}$ is the sum of $n^{4}$ terms of rank at most $1$,
so it has rank at most $n^{4}$. 
\end{proof}

The remainder of the proof of Proposition \ref{prop:TAP_UB} involves a slightly stronger version of Plefka's condition given by,
\[
			  \begin{array}{rcl}			  	{\rm{Plef}}_{N}^{\delta}(\bQ,\beta)&=& \left\{ \bbm \in \R^{n \times N} : \bM < \bQ,  \| \bb^{\frac{1}{2}} (\bQ - \bM) \bb^{\frac{1}{2}} \|_{2} \leq \frac{1}{\sqrt{2}} - \delta \right\}.
			  \end{array}
		  \]
Note that
\begin{equation}\label{eq: Plefka inc} {\rm{Plef}}_{N}^{\delta}(\bQ,\beta) \subset {\rm{Plef}}_{N}^{0}(\bQ,\beta) = {\rm{Plef}}_{N}(\bQ,\beta)\text{ for all }\delta\ge0,\bQ,\beta.
\end{equation}
This stronger Plefka condition is a device to allow the derivation of the upper bound for all $\beta$ from an upper bound for $\beta$ with only non-zero entries using continuity in the proof of Proposition \ref{prop:TAP_UB} below. The next lemma is a slight strengthening of \cite[Lemma~13]{belius-kistler}, and will be used below to prove that any maximizer of $\tTAP^K(\bbm)$ must satisfy the stronger Plefka condition.

\begin{lemma}\label{lem:lambda bound}
    For all $K \ge 2$ there is an $\varepsilon \in (0, \frac{2 \sqrt{2}}{K})$ and an $\delta_K > 0$ such that
    $$
        v_{\mu_K}(\beta) \ge \sqrt{2} - \varepsilon 
        \ \Rightarrow \
        \beta \le \frac{1}{\sqrt{2}} - \delta_K
    $$
\end{lemma}
\begin{proof}
    We may set $\varepsilon = \sqrt{2} - v_{\mu_K}(\tfrac{1}{\sqrt{2}}-\delta_K)$ since
    $$
        v_{\mu_K}(\beta) \ge v_{\mu_K}\left(\tfrac{1}{\sqrt{2}}-\delta_K\right)
        \Rightarrow \beta \le \tfrac{1}{\sqrt{2}}-\delta_K,
    $$
    and
    \begin{equation}\label{eq: BKlemma13 inequality}
        x_K < v_{\mu_K}\left(\tfrac{1}{\sqrt{2}}-\delta_K\right)
        < \sqrt{2},
    \end{equation}
    where the second inequality follows for some $\delta_K > 0$ small enough, because 
    $x_K < v_{\mu_K}(\tfrac{1}{\sqrt{2}}) < \sqrt{2}$ by \cite[Lemma~13]{belius-kistler} and $v_{\mu_K}$ is continuous.
\end{proof}

We now show that all maximizers of $\tTAP^K(\bbm)$ must satisfy the stronger Plefka condition. 
				
				
				\begin{lemma}[Critical point condition] \label{lem:critptcond} 
    
    Assume that $\beta_k>0$ for $k=1,\ldots,n$.
					Let $K \geq 1$. There exists a constant $c(K)$ such that if $N \geq c(K)$ then	\begin{equation}\label{eq:plefkahessian}
					\bM < \bQ \text{ and }\nabla^2 \tTAP^K(\bbm) \leq \bm{0} \implies \bmm \in {\rm{Plef}}_{N}^{\delta_K}(\bQ,\beta).
					\end{equation}
                    for $\delta_K$ as in Lemma~\ref{lem:lambda bound}.
				\end{lemma}

				\begin{proof}	
					By \eqref{eq: hamilt hess} and Lemma \ref{lem: hessian} we have	\begin{equation}\label{eq:Hessianformula used}
						\nabla^2 \tTAP^K(\bbm)= N \left( \bA +\bm{L} \right) \leq \bm{0},
					\end{equation}
					for 
					$$ \bA = 
					2\begin{bmatrix}
						\theta_{1} \bb & \cdots  & \bm 0_n\\
						\vdots & \ddots  & \vdots\\
						\bm 0_n & \cdots  & \theta_{\frac{1}{N}} \bb
					\end{bmatrix}
					-2 \begin{bmatrix} \bb^{\frac{1}{2}} v_{\mu_K}(2 \bm{Q}_{\bm{m}} )\bb^{\frac{1}{2}} & \cdots  & \bm 0_n\\
						\vdots & \ddots  & \vdots\\
						\bm 0_n & \cdots  &  \bb^{\frac{1}{2}} v_{\mu_K}(2 \bm{Q}_{\bm{m}} )\bb^{\frac{1}{2}} 
					\end{bmatrix},$$
					and $\bm{L}$ has rank at most $n^4$ (recall that $\bm{Q}_{\bm{m}} := \bb^{\frac{1}{2}} (\bQ - \bM) \bb^{\frac{1}{2}} $).
					
					Since $\bA$ is block diagonal its eigenvalues are the eigenvalues of its blocks. By Weyl's inequality \cite[Theorem~4.3.1]{HornMatrixAnalysis} all but $n^4 + 1$ of the eigenvalues of the matrix $\bA$ are bounded above by the largest eigenvalue of the entire Hessian $\nabla^2 \tTAP^K$. This means that there is a block among the last $n^4+2\le 2n^4$ that has all eigenvalues bounded by the largest eigenvalue of $\nabla^2 \tTAP^K$.
					
                Thus if $\nabla^2 \tTAP^K \le \bm{0}$ then
					\begin{equation}
						\label{eq:fullmatrix}
						2 \theta_{1 - \frac{2n^4}{N}} \bb  -2  \bb^{\frac{1}{2}} v_{\mu_K}(2 \bm{Q}_{\bm{m}} )\bb^{\frac{1}{2}} 
						\leq \bm{0}.
					\end{equation}
					If $ \bb^{1/2} \bm{B} \bb^{1/2} \le \bm{0}$ for a matrix $\bm{B}$ then $\bm{B} \le \bm{0}$, since we have assumed that $\bb$ is diagonal with positive entries on the diagonal. Therefore \eqref{eq:fullmatrix} implies that
					\[
					\theta_{1 - \frac{2n^4}{N}} \leq  v(\tilde{\beta}_i(\bbm))  \qquad \forall i \leq n.
					\]
					The properties of $v_{\mu_K}(\cdot)$ in Lemma~\ref{lem:lambda bound} imply  that there exists a $C(K)$ and $\delta_K > 0$ such that for all $\varepsilon \leq C(K)$,
					\[
					v_{\mu_K}(2\beta) \geq \sqrt{2} - \varepsilon \implies \beta \leq \frac{1}{\sqrt{2}} - \delta_K.
					\]
					Since $\theta_{1 - \frac{2n^4}{N}} = \sqrt{2} + o_N(1)$, it follows that for $N$ sufficiently large depending on $K$,
					\[
					\tilde{\beta}_{i}(\bbm) \leq \frac{1}{\sqrt{2}} - \delta_K \quad  \forall i \leq n
					\]
					(recall that $\tilde{\beta}_i(\bmm)$ are the eigenvalues of $\bQ_m$, defined in \eqref{eq:Q_m}) which implies that $\bm{m} \in {\rm{Plef}}_{K,N}(\bQ,\beta) \subseteq {\rm{Plef}}_{N}(\bQ,\beta)$.
				\end{proof}
		
	To conclude, we give the proof of Proposition~\ref{prop:TAP_UB}.

	\begin{proof}[Proof of Proposition~\ref{prop:TAP_UB}]
        We first assume that $\beta \in (0,\infty)^n$
        .
        By \eqref{eq:unifboundprob}, it suffices to study the free energy of the deterministic diagonalized Hamiltonian $\tilde H_N(\bs)$. 
		Starting from the upper bound Proposition~\ref{prop:upperbound} we have that for any $K\ge 2$ and $0 < \varepsilon \le C(\beta, \bh, \bQ, K)$
        as well as $N \ge c(\varepsilon,K)$ that
		\[
		\tilde F_N^\varepsilon (\beta, \bh, \bQ) \leq \frac{1}{N}  \sup_{ \bbm: \bM < \bQ }  \tTAP^K(\bbm) + \frac{c(\beta)}{K}.
		\]
		Recall from below \eqref{eq: open set} that $\tTAP^K$ has a global maximizer in the set \eqref{eq: open set}. This maximizer must satisfy
		\[
		    \nabla^2 \tTAP^K(\bbm) \leq 0.
		\]
		Thus it follows by Lemma~\ref{lem:critptcond} that
		\begin{equation}\label{eq: on the way}
		\tilde F_N^\varepsilon (\beta, \bh, \bQ) \leq \frac{1}{N} \sup_{\bbm \in {\rm{Plef}}_{N}^{\delta_K}(\bQ,\beta)} \tTAP^K(\bbm) + \frac{c(\beta)}{K}.
		\end{equation}
		By the definition \eqref{eq:Plefkadefn}, the equivalence in \eqref{eq:plefka threeway} and the uniform bound on $\cF_K$ from 
		Lemma~\ref{lem:Plefka} implies that for all $m \in \cM$
		\[
		\bigg|\cF_K( \bb^{\frac{1}{2}}(\bQ - \bM) \bb^{\frac{1}{2}} ) - \frac{1}{2} \beta^\Trans \bQ^{\odot 2} \beta \bigg| \leq o_K(1),
		\]
        where the term $o_K(1)$ does not depend on any parameters and tends to zero as $K\to\infty$. This allows us to replace $\cF_K$ of $\tTAP^K(\bbm)$ in \eqref{eq: on the way} with the Onsager correction term of $\tilde{F}_{\textrm{TAP}}$ (see \eqref{eq:tildeTAP}) in the upper bound, so that we obtain from \eqref{eq: on the way} that
		\begin{equation}\label{eq: UB at finite K with aK}
		  \tilde F_N^\varepsilon (\beta, \bh, \bQ) \leq \frac{1}{N} \sup_{\bbm \in {\rm{Plef}}_{N}^{\delta_K}(\bQ,\beta)} \tilde{F}_{\textrm{TAP}}(\bbm) + o_K(1),
		\end{equation}
        where the term $o_K(1)$ depends on $\beta$ and  tends to zero as $K\to\infty$ for fixed $\beta$.
		This upper bound holds for all $0 < \varepsilon < C(\beta, \bh,\bQ)$, all $K\ge 2$ and all $N\ge c(\varepsilon,K)$. 
        Using \eqref{eq: eval class location} we can bound the difference between the normalized Hamiltonian $\tfrac{1}{N}H_N$ and its diagonalized and deterministic counterpart $\tfrac{1}{N}\tilde{H}_N$ by any ${{{\eta}}} > 0$ with probability going to $1$. Thus, we get for $\TAP$ (recall \eqref{eq:TAP}) 
        \begin{equation}\label{eq: UB at finite K with aK non-deterministic}
             F_N^\varepsilon (\beta, \bh, \bQ) \leq \frac{1}{N} \sup_{\bbm \in {\rm{Plef}}_{N}^{\delta_K}(\bQ,\beta)} {F}_{\textrm{TAP}}(\bbm) + o_K(1)
        \end{equation}
        with probability going to $1$. Using \eqref{eq: Plefka inc} and picking $K$ large enough depending on ${{{\eta}}}$ we arrive at \eqref{eq:TAP_UB}. 
        We have thus proven \eqref{eq:TAP_UB} provided $\beta\in(0,\infty)^n$. 

        To handle $\beta$ with vanishing entries, 
			  note that if $\beta_{1},\beta_{2}\in[0,\infty)^{n}$ then by \eqref{def:FE} and \eqref{eq: hamilt UB} we
			  have
			  \begin{equation}
			  	\left|F_{N}^{\varepsilon}(\beta_{1},\bh,\bQ)-F_{N}^{\varepsilon}(\beta_{2},\bh,\bQ)\right|\le c\left|\beta_{1}-\beta_{2}\right|,\label{eq: first}
			  \end{equation}
			  with probability tending to $1$. Write $F_{\rm{TAP}}(\bmm;\beta)$ for $\TAP$ with the dependence on $\beta$ made explicit (recall \eqref{eq:TAP}). We similarly have
\begin{equation}\label{eq:second}
\left|\TAP(\bbm;\beta_1)-\TAP(\bbm;\beta_2)\right|\le c\left|\beta_{1}-\beta_{2}\right|,
\end{equation}
for bounded $\beta_{1},\beta_{2}$, using also that 
\begin{equation}
\bbm \to 
\beta^{\Trans} \left(\bQ-\bM\right)^{\odot2}\beta
\text{ is Lipschitz on compact subsets of }\mathbb{R}^n\label{eq: Lipschitz}
\end{equation}
uniformly in $\|\bQ\|_{\infty}\le1$ and $\bm{m}$ with $\bM \le \bQ$. Using \eqref{eq: Lipschitz} again we have that some small enough constant $\rho(K)$ depending only on $K$
\begin{equation}
\label{eq: inclusion plef}
{\rm Plef}_{N}^{\delta_K}\left(\bQ,\beta_{2}\right)\subset{\rm Plef}_{N}^{0}\left(\bQ,\beta_{1}\right)={\rm Plef}_{N}\left(\bQ,\beta_{1}\right)\text{\,for }\left|\beta_{1}-\beta_{2}\right|\le \rho\left(K\right),
\end{equation}
for bounded $\beta_1,\beta_2$.

Therefore for any $\beta_{1}$ with zero entries and ${{{\eta}}}>0$ we
can pick $K$ large enough depending on $\beta_1$ and ${{{\eta}}}$, and $\beta_{2}$ with all positive entries
close enough to $\beta_{1}$, such that 
\[
\begin{array}{ccccc}
F_{N}^{\varepsilon}(\beta_{1},\bh,\bQ) & \overset{\eqref{eq: first}}{\le} & {F}_{N}^{\varepsilon}(\beta_{2},\bh,\bQ)+\frac{{{{\eta}}}}{3} & \overset{\eqref{eq: UB at finite K with aK non-deterministic}}{\le} & {\displaystyle \sup_{\bbm\in{\rm {Plef}}_{N}^{\delta_K}(\bQ,\beta_{2})}}{F}_{\textrm{TAP}}(\bbm;\beta_2)+\frac{2{{{\eta}}}}{3}\\
 &  &  & \overset{\eqref{eq:second},\eqref{eq: inclusion plef}}{\le}
  & {\displaystyle \sup_{\bbm\in{\rm {Plef}}_{N}(\bQ,\beta_{1})}}\TAP(\bbm;\beta_1)+{{{\eta}}}
\end{array}
\]
with probability tending to one. This proves \eqref{eq:TAP_UB} for $\beta$ with
vanishing entries.

	\end{proof}
	
	Combining the TAP lower bound Proposition~\ref{prop: TAP LB} and the TAP upper bound Proposition~\ref{prop:TAP_UB} completes the proof of Theorem~\ref{thm:TAPformula}.
	
	\section{Ground State Energy}\label{section: solution variational problem}
	
	All that remains is to prove the ground state formula in Theorem~\ref{thm:groundstate}. To avoid technical issues with the invertibility of matrices, we will first assume that $\beta$ and $\bh$ are non-zero, then extend to all $\beta$ and external fields using continuity. 
    
    By the uniform bound \eqref{eq:unifboundprob}, we can write Hamiltonian and external field in terms of its diagonalizing basis, so it suffices to compute the limit of
    \begin{equation}\label{eq:fgroundstate}
        \sup_{\bbm \bbm^\Trans = \tilde{\bm{Q}}} f(\bbm, \beta,\bh) := \sup_{\bbm \bbm^\Trans = \tilde{\bm{Q}}} \bigg( \frac{1}{N} \sum_{k = 1}^n \beta_k \tilde H_N ( m^k) + \sum_{k = 1}^n m^k \cdot \tilde h^k  \bigg)
    \end{equation}
				where $\tilde H_N$ is the deterministic counterpart of $H_N$ defined in \eqref{eq:deterministicHamiltonian} and $\tilde h^k$ is the vector $h^k$ written in the diagonalizing basis of the disorder matrix $J$, as in the previous section. We define the following variational form of the ground state functional 
					\begin{equation}
					{\rm{\widetilde {GSE}}}(\beta,h,\tilde{\bm{Q}}) = \inf_{\bL - \sqrt{2} \bI \geq 0}\bigg( \frac{1}{4}  h^\Trans \bb^{-1/2} \big( \bL - \sqrt{ \bL^2 - 2 \bI } \big) \bb^{-1/2} h   + \Tr( \bL \bb^{1/2} \tilde{\bm{Q}} \bb^{1/2} )  \bigg).
				\end{equation}
				We will now show that $\sup_{\bbm \bbm^\Trans = \tilde{\bm{Q}}} f(\bbm, \bb,\bh)$ converges in probability to ${\rm{\widetilde{GSE}}}(\beta,h,\tilde{\bm{Q}})$.
				
				\begin{proposition}\label{prop:groundstate} For $\beta_1, \dots, \beta_n, h_1, \dots, h_n \neq 0$, we have
					\begin{equation}
					\sup_{\bbm \bbm^\Trans = \tilde{\bm{Q}}} \bigg( \frac{1}{N} \sum_{k = 1}^n \beta_k \tilde H_N ( m^k) + \sum_{k = 1}^n m^k \cdot \tilde h^k  \bigg) \stackrel{\mathbb{P}}{\to} 	{\rm{\widetilde{GSE}}}(\beta,h,\tilde{\bm{Q}}).
					\end{equation}
				\end{proposition}

				\begin{proof}

            We use Lagrange multipliers to explicitly solve the constrained maximization problem. Consider the Lagrangian 
				\begin{align*}
				f(\bbm,\bL)&=\sum_{k=1}^{n}\beta_{k}\sum_{i=1}^{N}\theta_{i/N}(m_{i}^{k})^{2}+\sum_{k=1}^{n}\tilde{h}^{k}\cdot m^{k}+ \Tr(  \bL \bb^{\frac{1}{2}} (\tilde{\bm{Q}} - \bbm \bbm^\Trans) \bb^{\frac{1}{2}} ) 
				\\&=  \sum_{i = 1}^N \left(      (\bb^{\frac{1}{2}} m_{i})^\Trans (\theta_{i/N} \bI - \bL) (\bb^{\frac{1}{2}} m_i)  + (\bb^{-\frac{1}{2}} \tilde h_i) \cdot \bb^{\frac{1}{2}} m_i \right) + \Tr(\bL \bb^{\frac{1}{2}} \tilde{\bm{Q}} \bb^{\frac{1}{2}}).
				\end{align*}
				Note that
				\[
				\partial_{\Lambda_{ij}}f(\bbm,\bL)=\left(1+\delta_{i\ne j}\right)\sqrt{\beta_i \beta_j}\left(\tilde{\bQ}-\bbm \bbm^{\Trans}\right)_{ij}.
				\]
				Thus if $\bL^{*}>\sqrt{2}\bI$ and $\bmm^{*}$ is a critical point of $f(\bbm,\bL)$
				then 
				\[
				f(\bbm^{*},\bL^{*})\le\sup_{\bbm\bbm^{\Trans}=\tilde{\bm{Q}}}f(\bbm, \beta,\bh).
				\]
				Also
				\[
				\sup_{\bbm\bbm^{\Trans}=\tilde{\bm{Q}}}f(\bbm, \beta,\bh)\le\inf_{\Lambda>\sqrt{2} \bI}\sup_{\bbm\bbm^{\Trans}=\tilde{\bm{Q}}}f(\bbm,\bL)\le\inf_{\Lambda>\sqrt{2}\bI}\sup_{\bbm}f(\bbm,\bL).
				\]
				Therefore since $f$ is differentiable for fixed $N,\tilde{h}$, if
				a finite optimizer of the r.h.s. such that $\bL^{*}>\sqrt{2}\bI$
				exists then 
				\begin{equation}
					\sup_{\bbm\bbm^{\Trans}=\tilde{\bm{Q}}}f(\bbm, \beta,\bh)=\inf_{\Lambda>\sqrt{2}\bI}\sup_{\bbm}f(\bbm,\bL).\label{eq: lagrange equal}
				\end{equation}
				Consider $\sup_{\bbm}f(\bbm,\bL)$ for fixed $\bL$. We have
				\[
				\partial_{m_{i}}f(\bbm,\bL)= 2    (\theta_{i/N} \bI - \bL) \bb^{1/2} m_i + \bb^{-1/2} \tilde h_i.
				\]
				If $\bL>\sqrt{2}\bI$ then the unique critical point of $\bbm\to f(\bbm,\bL)$
				is thus
				\begin{equation}\label{eq:critptconditionlagrange}
					m_{i}\left(\bL\right)= \frac{1}{2} \bb^{-1/2}  ( \bL - \theta_{i/N} \bI)^{-1} \bb^{-1/2} \tilde  h_i
				\end{equation}
				and by concavity this critical point corresponds to a local maximizer, and thus
				\begin{equation}\label{eq:supinm}
					\sup_{\bbm}f(\bbm,\bL)=\sum_{i = 1}^N \bigg( \frac{1}{4}  \tilde h_i^\Trans \bb^{-1/2} (\bL - \theta_{i/N} \bI)^{-1} \bb^{-1/2} \tilde h_i \bigg)  + \Tr( \bL \bb^{1/2} \tilde{\bm{Q}} \bb^{1/2} ).
				\end{equation}
				Note that if $\bb^{1/2} \tilde{\bm{Q}} \bb^{1/2}>0$ then since $\left(\bU^{\Trans}\bb^{1/2} \tilde{\bm{Q}}\bb^{1/2} \bU\right)_{ii}=u_{i}^{T}\bb^{1/2} \tilde{\bm{Q}} \bb^{1/2}u_{i}>0$
				for all orthogonal $\bU$ we have 
				\[
				\sup_{\bbm}f(\bbm,\bL)\ge\Tr(\bL\bb^{1/2}\tilde{\bm{Q}}\bb^{1/2})\to\infty\text{ if }\bL>\sqrt{2}\bI,~ \sup_{k}\lambda_{k}\left(\bL\right)\to\infty.
				\]
				Also if $\bL\to\sqrt{2}\bI$ and
				$\bL>\sqrt{2}\bI$ then almost surely
				\[
				\begin{array}{ccl}
					\sup_{\bbm}f(\bbm,\bL) & \ge & \frac{1}{4}\tilde{h}_{N}^{\Trans}\bb^{-1/2} (\bL - \theta_{i/N} \bI)^{-1} \bb^{-1/2} \tilde{h}_{N}\\
					& \ge & \frac{1}{4}\frac{\left|\tilde{h}_{N}\right|^{2}}{\lambda_{\max}\left(\bb^{1/2}(\bL-\sqrt{2}\bI) \bb^{1/2}\right)}\\
					& \to & \infty,
				\end{array}
				\]
				since $\tilde{h}_{N}\ne0$ a.s. and $\lambda_{\max}\left(\bb^{1/2}(\bL-\sqrt{2}\bI) \bb^{1/2}\right)\to0$.
				This shows that minimizer of 
				\[
				\inf_{\bL>\sqrt{2}\bI}f(\bm{m}\left(\bL\right),\bL),
				\]
				is attained at a point in $\left\{ \bL:\bL>\sqrt{2}\bb\right\} $,
				and thus that there exists an optimizer of 
				\[
				\inf_{\bL>\sqrt{2}\bI}\sup_{\bm{m}}f(\bm{m},\bL),
				\]
				which is a critical point of $f$, so that \eqref{eq: lagrange equal}
				holds.
				
				We now show that $f(\bm{m}\left(\bL\right),\bL)$ converges to the
				limiting function of $\bL$ so that
				\[
				\inf_{\bL>\sqrt{2}\bI}\sup_{\bm{m}}f(\bm{m},\bL)\to\inf_{\bL>\sqrt{2}\bI} \bigg( \frac{1}{4}  h^\Trans \bb^{-1/2} \big( \bL - \sqrt{ \bL^2 - 2 \bI } \big) \bb^{-1/2} h   + \Tr( \bL \bb^{1/2} \tilde{\bm{Q}} \bb^{1/2} )  \bigg) .
				\]
				By Proposition~\ref{lem:cont} and \eqref{eq: lagrange equal}, this convergence is uniform on compact subsets of $\bb$ and $\bh$ and the limit is $\sqrt{2}$-Lipschitz because the left hand side is. 

                Recall \eqref{eq:supinm}. Note that since $\mathbb{E}[\tilde{h}_{i,k}\tilde{h}_{i,l}]=h_{i}h_{l}$
it holds that
\[
\mathbb{E}\left[\frac{1}{4}\tilde{h}_{i}^{\Trans}\bb^{-1/2}(\bL-\theta_{i/N}\bI)^{-1}\bb^{-1/2}\tilde{h}_{i}\right]=\frac{1}{4}h^{\Trans}\bb^{-1/2}(\bL-\theta_{i/N}\bI)^{-1}\bb^{-1/2}h,
\]
and also the $\tilde{h}_i$ are independent, so by the law of large
numbers
				\begin{align*}
					f(\bL,\bbm) 
					&= \sum_{i = 1}^N \bigg( \frac{1}{4}  \tilde h_i^\Trans \bb^{-1/2} (\bL - \theta_{i/N} \bI)^{-1} \bb^{-1/2} \tilde h_i \bigg)  + \Tr( \bL \bb^{1/2} \tilde{\bm{Q}} \bb^{1/2} )
					\\&\to  \frac{1}{4}  h^\Trans \bb^{-1/2} \bigg( \int_{-\sqrt{2}}^{\sqrt{2}} (\bL - x \bI)^{-1} \, d\mu_{\rm{sc}}(x) \bigg) \bb^{-1/2} h  + \Tr( \bL \bb^{1/2} \tilde{\bm{Q}} \bb^{1/2} ),
				\end{align*}
				in probability.
    For $\bL$ such that $\lambda_{min}(\bL) > \sqrt{2}$, we can compute the integral explicitly. Let $\bL = \bU \bd_\lambda \bU^\Trans$. We see that
				\[
				\int_{-\sqrt{2}}^{\sqrt{2}} ( \bL - x \bI )^{-1}  \mu_{\rm{sc}}(x) \, dx = \int_{-\sqrt{2}}^{\sqrt{2}} \bU ( \bd_\lambda - x \bI )^{-1} \bU^\Trans \mu_{\rm{sc}}(x) \, dx =  \bU \int_{-\sqrt{2}}^{\sqrt{2}} ( \bd_\lambda - x \bI )^{-1} \mu_{\rm{sc}}(x) \, dx \bU^\Trans
				\]
				and the integral on the inside is easy to compute. In fact, using the formula for the one dimensional case, we see that
				\[
				\int_{-\sqrt{2}}^{\sqrt{2}} ( \bL - x \bI )^{-1}  \mu_{\rm{sc}}(x) \, dx = \bU \bd_{\lambda - \sqrt{\lambda^2 - 4}} \bU^\Trans = \bU \bd_{\lambda} \bU^\Trans - \bU \bd_{\sqrt{\lambda^2 - 4}} \bU^\Trans = \bL - \bU \bd_{\sqrt{\lambda^2 - 4}} \bU^\Trans.
				\]
				Since
				\[
				\bU \bd_{\sqrt{\lambda^2 - 2}} \bU^\Trans = \sqrt{ \bU \bd_{\lambda^2 - 2} \bU^\Trans } = \sqrt{ \bU \bd_{\lambda^2}  \bU^\Trans - 2 \bI } = \sqrt{ (\bU \bd_{\lambda}  \bU^\Trans)^2 - 2 \bI } = \sqrt{ \bL^2 - 2 \bI }
				\]
				we have
				\[
				\int_{-\sqrt{2}}^{\sqrt{2}} ( \bL - x \bI )^{-1}  \mu_{\rm{sc}}(x) \, dx = \bL - \sqrt{ \bL^2 - 2 \bI }.
				\]
				With this formula, it follows that
				\[
				\sup_{\bbm} f(\bL,\bbm) \to  \frac{1}{4}  h^\Trans \bb^{-1/2} \big( \bL - \sqrt{ \bL^2 - 2 \bI } \big) \bb^{-1/2} h   + \Tr( \bL \bb^{1/2} \tilde{\bm{Q}} \bb^{1/2} ) 
				\]
				and that the critical point corresponds to a maximum. Notice that this formula is well defined for $\bL  \geq \sqrt{2} \bI$.  It follows from \eqref{eq: lagrange equal} that the maximum of \eqref{eq:fgroundstate} is attained at
				\begin{equation}\label{eq:optimization}
					\inf_{\bL \geq \sqrt{2} \bI}\bigg( \frac{1}{4}  h^\Trans \bb^{-1/2} \big( \bL - \sqrt{ \bL^2 - 2 \bI } \big) \bb^{-1/2} h   + \Tr( \bL \bb^{1/2} \tilde{\bm{Q}} \bb^{1/2} )  \bigg)
				\end{equation}
				in the limit.
				\end{proof}
	We now explicitly solve the optimization in ${\rm{\widetilde{GSE}}}(\beta,h,\tilde{\bm{Q}})$ to arrive at the closed form expression from \eqref{eq:Lfunctopt}.

	\subsection{The One Dimensional Case}\label{sec:highdimvar}
    We first address the case $n=1$ as a warm-up. 
	Solving the variational problem in this case is considerably easier because we do not have to worry about the non-commutativity of the matrices. When $n=1$ the variational problem is
	\begin{equation}\label{eq: 1d case}
	\inf_{\lambda  \geq \sqrt{2}} \bigg( \frac{1}{4}  \frac{h^2}{\beta}  \big( \lambda - \sqrt{ \lambda^2 - 2} \big) + \lambda \beta \tilde{q}  \bigg).
	\end{equation}    

    Let $A=\frac{1}{4}\frac{h^{2}}{\beta}$ and $B=\beta \tilde{q}$. With the
    change of variables $\lambda=\frac{1}{\sqrt{2}}\left(x+\frac{1}{x}\right),x\in\left(0,1\right],$
    and using that $\sqrt{\lambda^{2}-2}=\frac{1}{\sqrt{2}}\left(\frac{1}{x}-x\right)$
    one obtains that \eqref{eq: 1d case} equals
    \[
    \inf_{x\in\left(0,1\right]}\left(A\sqrt{2}x+B\frac{1}{\sqrt{2}}\left(\frac{1}{x}+x\right)\right)=\sqrt{2B\left(2A+B\right)}=\sqrt{\tilde{q}h^{2}+2\beta^{2}\tilde{q}^{2}}
    \]
 	which proves \eqref{eq: case n=1} and is indeed the formula from \cite[(1.6) and Lemma 20]{belius-kistler}. 
	
	\subsection{The \texorpdfstring{$n$}{TEXT} Dimensional Case}
    A matrix version of this change of variables allows one to solve also the case $n>1$, giving rise to ${\rm{GSE}}(\beta,h,\tilde{\bm{Q}})$ from \eqref{eq:Lfunctopt}.
    \begin{proposition}\label{prop:LFormula}
        For $\tilde \bQ > 0$ and $\bb > 0$ it holds that 
     \begin{equation}\label{eq:GSE}
         {\rm{\widetilde{GSE}}}(\beta,h,\tilde{\bm{Q}})  = {\rm{GSE}}(\beta,h,\tilde{\bm{Q}}).
     \end{equation}
				\end{proposition}
				
				\begin{proof}
Let 
$$\bA =  \frac{1}{4} \bb^{-1/2} h h^\Trans \bb^{-1/2} \text{ and } \bB = \bb^{1/2} \tilde{\bm{Q}} \bb^{1/2}.$$ 
                Writing the first term of ${\rm{GSE}}(\beta,h,\tilde{\bm{Q}})$ as the trace of a $1\times1$ matrix and using the cyclical property of the trace we have

                $$                 {\rm{\widetilde{GSE}}}(\beta,h,\tilde{\bm{Q}}) =\inf_{\bL - \sqrt{2} \bI \geq 0}\bigg( \frac{1}{4} \Tr ( \big( \bL - \sqrt{ \bL^2 - 2 \bI } \big)\bA) + \Tr( \bL \bB )  \bigg).
                $$
                We use the change of variables $\bL = \frac{1}{\sqrt{2}} (\bX + \bX^{-1})$ where $\bm{0} < \bX \leq \bI$ and $\bX$ is symmetric (to see that $\bL$ can always be written in this form recall that $\bL$ is symmetric and has eigenvalues larger or equal to $\sqrt{2}$). It follows that $(\bL^2 - 2\bI)^{\frac{1}{2}} = \frac{1}{\sqrt{2}} (\bX^{-1} - \bX)$. Thus
				\begin{equation}\label{eq: trace form}
                    {\rm{\widetilde{GSE}}}(\beta,h,\tilde{\bm{Q}}) = \inf_{\bm{0} < \bX \leq \bI}  \bigg(\sqrt{2} \Tr\big(   \bX \bA \big)  + \frac{1}{\sqrt{2}} \Tr( ( \bX + \bX^{-1}) \bB ) \bigg).
				\end{equation}
            
		      Consider the critical point equation for the quantity in the $\inf$:
				\[
				\partial_{\bX} \bigg(\sqrt{2} \Tr\big(   \bX \bA \big)  + \frac{1}{\sqrt{2}} \Tr( \bB( \bX + \bX^{-1}) ) \bigg) = \sqrt{2} \bA + \frac{1}{\sqrt{2}} \bB - \frac{1}{\sqrt{2}} \bX^{-1} \bB \bX^{-1} = \bm{0}.
				\]
				Using the change of variables $\bY = \bB^{\frac{1}{2}} \bX^{-1}$, it is equivalent to
				\[
				\bY^\Trans \bY= 2 \bA + \bB.
				\]
				If we diagonalize $2 \bA + \bB= \bU_1 \bd \bU_1^\Trans$, it is further equivalent to $\bY = \bU_{2} \bd^{\frac{1}{2}} \bU^\Trans_1$ for some orthogonal $\bU_2$, 
                and therefore to $\bX = \bU_{1} \bd^{-\frac{1}{2}} \bU^\Trans_2 \bB^{\frac{1}{2}} = (2 \bA + \bB)^{-\frac{1}{2}} \bU \bB^{\frac{1}{2}}$, where $\bU = \bU^\Trans_1 \bU_2^\Trans$
                Thus if $\bX$ is symmetric and
                \begin{equation}\label{eq:criticalptU}
				\bX = (2 \bA + \bB)^{-\frac{1}{2}} \bU \bB^{\frac{1}{2}},
				\end{equation}
                for some orthogonal $\bU$ it is a critical point. Such a $\bU$ can be found as follows. The symmetry condition is equivalent to
			\[
    			(2 \bA + \bB)^{-\frac{1}{2}} \bU \bB^{\frac{1}{2}} =  \bB^{\frac{1}{2}} \bU^\Trans (2 \bA + \bB)^{-\frac{1}{2}} \iff \bU \bB^{\frac{1}{2}} (2 \bA + \bB)^{\frac{1}{2}} = (2 \bA + \bB)^{\frac{1}{2}}  \bB^{\frac{1}{2}} \bU^\Trans.
			\]
			Write $\bB^{\frac{1}{2}} (2 \bA + \bB)^{\frac{1}{2}} = \bm{S} \bm{\Sigma} \bm{T}^\Trans$ in its singular value decomposition. The condition then becomes,
			\[
			 \bU \bm{S} \bm{\Sigma} \bm{T}^\Trans  = \bm{T} \bm{\Sigma} \bm{S}^\Trans   \bU^\Trans,
			\]            
            so $\bU = \bm{T} \bm{S}^\Trans$ is orthogonal and makes $\bX$ symmetric. We have thus proven than this $\bX$ is a critical point of the expression in the infiumum of \eqref{eq: trace form}.

            Now note that $\bX \mapsto \Tr(\bB \bX^{-1})$ is convex because $\bX \mapsto \bX^{-1}$ is convex and $\bX \mapsto \Tr(\bB \bX)$ is increasing for positive definite $\bB$, see \cite[Corollary~V.2.6]{bhatia1996matrix}.  Therefore the expression in the infiumum of \eqref{eq: trace form} is convex in $\bX$, and thus the exhibited critical point is a global minimizer.
            
			Next we compute the value at this minimizer. Substituting it into \eqref{eq: trace form} and using $\bX^{-1} = (\bX^{-1})^{\Trans} = (2 \bA + \bB)^{\frac{1}{2}} \bU \bB^{-\frac{1}{2}}$ one obtains
                \begin{align*}
                    &{\rm{\widetilde{GSE}}}(\beta,h,\tilde{\bm{Q}})
                    \\
                    =& \frac{1}{\sqrt{2}}  \bigg(
                            2 \Tr\big(  \bA (2 \bA + \bB)^{-\frac{1}{2}} \bU \bB^{\frac{1}{2}} \big)
                            +  \Tr(  (2 \bA + \bB)^{-\frac{1}{2}} \bU \bB^{\frac{1}{2}} \bB)
                            + \Tr( (2 \bA + \bB)^{\frac{1}{2}} \bU \bB^{-\frac{1}{2}} \bB )
                        \bigg)\\					
                    =& \frac{1}{\sqrt{2}}  \bigg( 2 \Tr\big( \bB^{\frac{1}{2}} \bA (2 \bA + \bB)^{-\frac{1}{2}} \bU   \big)  +  \Tr( \bB^{\frac{3}{2}} (2 \bA + \bB)^{-\frac{1}{2}} \bU ) + \Tr( \bB^{\frac{1}{2}}  (2 \bA + \bB)^{\frac{1}{2}} \bU)\bigg)
					\\=& \frac{1}{\sqrt{2}} \Tr\bigg( \bB^{\frac{1}{2}} \bigg( 2\bA  + \bB +  (2 \bA + \bB)  \bigg) \bigg( (2\bA + \bB)^{-\frac{1}{2}} \bU \bigg) \bigg)
					\\=& \sqrt{2} \Tr\big( \bB^{\frac{1}{2}} ( 2\bA  + \bB)^{\frac{1}{2}}  \bU \big)
                    \\=& \sqrt{2} \Tr\big( \bm{S} \bm{\Sigma} \bm{T}^\Trans \bm{T} \bm{S}^\Trans \big) = \sqrt{2}\Tr(\bm{\Sigma}).
				\end{align*}

    Note that $\bm{\bm{\Sigma}}$ is the diagonal matrix of singular values
of $\bB^{1/2}\left(2\bA+\bB\right)^{1/2}$, i.e. of square roots of
the eigenvalues of 
\[
\bB^{1/2}\left(2\bA+\bB\right)^{1/2}\left(\bB^{1/2}\left(2\bA+\bB\right)^{1/2}\right)^{T}.
\]
Using repeatedly the property that $\bC\bm{D}$ and $\bm{D}\bC$ have the
same eigenvalues for square $\bC,\bm{D}$ we get that these eigenvalues
coincide with those of 
\[
\left(2\bA+\bB\right)\bB=\left(2\bA+\bB\right)\bb^{1/2}\tilde{\bQ}\bb^{1/2}.
\]
Using the same property again this r.h.s. in turn has the same eigenvalues
as 
\[
\bb^{1/2}\left(2\bA+\bB\right)\bb^{1/2}\tilde{\bQ}=\Big(\frac{1}{2}hh^{\Trans}+\bb\tilde{\bm{Q}}\bb\Big)\tilde{\bQ},
\]
which in turn has the same eigenvalues as 
\[
\Big(\frac{1}{2}hh^{\Trans}+\bb\tilde{\bm{Q}}\bb\Big)^{\frac{1}{2}}\tilde{\bQ}\Big(\frac{1}{2}hh^{\Trans}+\bb\tilde{\bm{Q}}\bb\Big)^{\frac{1}{2}}.
\]
This proves that
\[
\Tr\left(\bm{\Sigma}\right)=\Tr\left(\sqrt{\Big(\frac{1}{2}hh^{\Trans}+\bb\tilde{\bm{Q}}\bb\Big)^{\frac{1}{2}}\tilde{\bQ}\Big(\frac{1}{2}hh^{\Trans}+\bb\tilde{\bm{Q}}\bb\Big)^{\frac{1}{2}}}\right),
\]
and recalling the definition \eqref{eq:Lfunctopt} of ${\rm{GSE}}$ this completes the proof.
            \end{proof}
Thus for $\beta$ and $h$ with only non-zero components
\begin{equation}\label{eq:conv non-zero}
    \sup_{\bbm \bbm^\Trans = \tilde{\bm{Q}}} \bigg( \frac{1}{N} \sum_{k = 1}^n \beta_k \tilde H_N ( m^k) + \sum_{k = 1}^n m^k \cdot \tilde h^k  \bigg) \stackrel{\mathbb{P}}{\to} 	{\rm{GSE}}(\beta,h,\tilde{\bm{Q}})
\end{equation}
(by combining Propositions \ref{prop:groundstate} and \ref{prop:LFormula}). 
The formula for ${\rm{GSE}}(\beta,h,\tilde{\bm{Q}})$ is well-defined also if some entry of $\beta$ or $h$ is zero. To extend \eqref{eq:conv non-zero} to this case we will using a continuity argument enabled by the next lemma, which shows that \eqref{eq:fgroundstate} is Lipschitz in $\beta$ and $h$.

				\begin{lemma}\label{lem:cont}
					If $\tilde \bQ$ is positive definite with entries bounded by $1$, then 
					\[
					\Big|\sup_{\bbm \bbm^\Trans = \tilde{\bm{Q}}} f(\bbm, \beta^1,\bh^1) - \sup_{\bbm \bbm^\Trans = \tilde{\bm{Q}}} f(\bbm,  \beta^2,\bh^2) \Big| \leq \sqrt{2} \| \beta^1 - \beta^2 \|_\infty + \| h^1 - h^2 \|_\infty .
					\]
				\end{lemma}
				
        \begin{proof}

        This follows since
            \begin{align*}\sup_{\bbm\bbm^{\Trans}=\tilde{\bm{Q}}}\left|f(\bbm,\beta^{1},\bh^{1})-f(\bbm,\beta^{2},\bh^{2})\right| & \leq\sup_{|m^{1}|,\dots,|m^{n}|\leq1}\left|f(\bbm,\beta^{1},\bh^{1})-f(\bbm,\beta^{2},\bh^{2})\right|\\
             & \leq\sqrt{2}\|\beta^{1}-\beta^{2}\|_{\infty}+\sup_{k \leq n}|(h^k)^{1}-(h^k)^{2}|\\
             & \leq\sqrt{2}\|\beta^{1}-\beta^{2}\|_{\infty}+\|h^{1}-h^{2}\|_{\infty},
            \end{align*}
            because 
            \[
        |(h^k)^{1}-(h^k)^{2}| = | h_k^{1} u -h_k^{2} u|^2 = |h_k^1 - h_k^2|^2.
            \]
        \end{proof}
        Note also from the formula in 
        \eqref{eq:Lfunctopt} that
        \begin{equation}\label{eq:GSE Lipschitz} 
            \text{ for all } \tilde{\bQ}>0 \text{ the map } (h,\beta)\to{\rm{GSE}}(\beta,h,\tilde{\bm{Q}})\text{ is continuous}.
        \end{equation}
        
        Theorem~\ref{thm:groundstate} is now immediate 
from Propositions \ref{prop:groundstate} and \ref{prop:LFormula} and continuity.
        \begin{proof}[Proof of Theorem~\ref{thm:groundstate}]
            The reduction above \eqref{eq:fgroundstate} and
            Propositions \ref{prop:groundstate} and \ref{prop:LFormula} prove the claim \eqref{eq: max TAP free energy} when all entries of $\beta$ and $h$ are non-zero. A simple approximation argument using Lemma \ref{lem:cont} 
            and \eqref{eq:GSE Lipschitz} extends this to all $\beta,h$.
        
        \end{proof}
 
	\printbibliography
\end{document}